\documentclass[12pt,a4paper]{amsart}

\usepackage{hyperref}

\usepackage[foot]{amsaddr}
\usepackage{amsmath}
\usepackage{amssymb}
\usepackage{amsthm}
\usepackage{cleveref}
\usepackage[usenames]{color}
\usepackage{enumitem}
\usepackage{geometry}
\usepackage{mathrsfs}
\usepackage{mathtools}
\usepackage[active]{srcltx}
\usepackage[normalem]{ulem}

\theoremstyle{plain}
\newtheorem{lemma}{Lemma}[section]
\newtheorem{theorem}[lemma]{Theorem}
\newtheorem{proposition}[lemma]{Proposition}
\newtheorem{corollary}[lemma]{Corollary}

\theoremstyle{definition}

\newtheorem{definition}[lemma]{Definition}
\newtheorem{remark}[lemma]{Remark}
\newtheorem{example}[lemma]{Example}

\numberwithin{equation}{section}

\let\temp\phi
\let\phi\varphi
\let\temp\epsilon
\let\epsilon\varepsilon
\let\varepsilon\temp

\newcommand{\R}{\mathbb{R}}
\newcommand{\N}{\mathbb{N}}

\newcommand{\Q}{\mathbb{Q}}

\DeclareMathOperator*{\argmin}{arg\,min}

\DeclareMathOperator{\Dom}{Dom}
\DeclareMathOperator{\Ent}{Ent}
\DeclareMathOperator{\Exp}{Exp}

\DeclareMathOperator{\Graph}{Graph}
\DeclareMathOperator{\Hess}{Hess}

\DeclareMathOperator{\Length}{Length}

\DeclareMathOperator{\OptCaus}{OptCaus}
\DeclareMathOperator{\OptGeo}{OptGeo}

\DeclareMathOperator{\Ric}{Ric}
\DeclareMathOperator{\Span}{span}

\DeclareMathOperator{\dom}{Dom}
\DeclareMathOperator{\ee}{e}

\DeclareMathOperator{\id}{Id}

\DeclareMathOperator{\supp}{supp}

\DeclareMathOperator{\tr}{tr}
\DeclareMathOperator{\vol}{Vol}

\newcommand{\de}{{\mathrm{d}}}

\newcommand{\CD}{\mathsf{CD}}

\newcommand{\Leb}{\mathcal{L}}

\newcommand{\NC}{\mathsf{NC}}

\newcommand{\haus}{\mathcal{H}}

\newcommand{\weak}{\rightharpoonup}

\newcommand{\Prob}{\mathcal P}

\newcommand{\M}{\mathcal{M}}

\newcommand{\dotargument}{\,\cdot\,}

\renewcommand{\L}{\mathcal{L}}

\newcommand{\q}{\mathfrak{q}}

\newcommand{\QQ}{\mathfrak Q}

\newcommand{\cbf}{\mathbf{f}}
\newcommand{\gflow}{\Psi}

\newcommand{\indicator}{\mathbf{1}}
\newcommand{\lparam}{\zeta}
\newcommand{\mm}{\mathfrak m}

\newcommand{\qq}{\mathfrak q}
\newcommand{\relation}{\mathcal{R}}

\newcommand{\vfield}{\mathfrak{X}}
\newcommand{\normal}{\mathscr{N}}

\setlist[enumerate]{leftmargin=1cm}
\setlist[itemize]{leftmargin=1cm}

\title[Optimal transport on null hypersurfaces and the NEC]{Optimal transport on null hypersurfaces and the null energy condition}

\author{Fabio Cavalletti}
\email{fabio.cavalletti@unimi.it}
\author{Davide Manini}
\email{dmanini@campus.technion.ac.il}
\thanks{D.\;M.\;acknowledges support from the European Research Council (ERC)
  under the European Union's Horizon 2020 research and innovation
  programme, grant agreement No.~101001677 ``ISOPERIMETRY''}
\author{Andrea Mondino}
\email{andrea.mondino@maths.ox.ac.uk}
\thanks{A.\;M.\;acknowledges support from the European Research Council (ERC) under the European Union's Horizon 2020 research and innovation programme, grant agreement No.~802689 ``CURVATURE''}

\begin{document}
\begin{abstract}
The goal of the present work is to study optimal transport on null hypersurfaces inside Lorentzian manifolds. The challenge here is that optimal transport along a null hypersurface is completely degenerate, as the cost takes only the two values $0$ and $+\infty$.
The tools developed in the manuscript enable to give an optimal transport characterization of the null energy condition (namely, non-negative Ricci curvature in the null directions) for Lorentzian manifolds in terms of convexity properties of the Boltzmann--Shannon entropy along null-geodesics of probability measures.
 We obtain as applications: a stability result under convergence of spacetimes, a comparison result for null-cones, and the Hawking area theorem (both in sharp form, for possibly weighted measures, and with apparently new rigidity statements).
\end{abstract}



\maketitle
\setcounter{tocdepth}{1}

\tableofcontents

\section{Introduction}
Optimal transport revealed to be an extremely powerful tool in Riemannian signature.  In particular, one can employ optimal transport in Riemannian manifolds in order to characterize Ricci curvature lower bounds in terms of convexity properties of the Boltzmann--Shannon entropy along $W_2$-geodesics of probability measures,   see McCann~\cite{McCannThesis}, Otto--Villani~\cite{OttoVillani}, Cordero Erausquin--McCann--Schmuckenschl\"ager~\cite{CEMS}, von Renesse--Sturm~\cite{vRS}. Such a point of view has been extremely fruitful; for instance, it has been the starting point for developing a theory of non-smooth metric measure spaces with Ricci curvature bounded below and dimension bounded above in a synthetic sense, the so-called ${\mathsf{CD}}(K,N)$ spaces of Lott--Sturm--Villani~\cite{sturm:I, sturm:II, lottvillani}.

The goal of the present work is to give an optimal transport characterization of Ricci curvature lower bounds in the null directions for Lorentzian manifolds in terms of convexity properties of the Boltzmann--Shannon entropy along null-geodesics of probability measures, in the spirit  of the aforementioned~\cite{McCannThesis, OttoVillani,CEMS, vRS}. The challenge here (see later in the introduction for more details) is that optimal transport in the null directions is completely degenerate (the cost takes only the two values $0$ and $+\infty$). Thus, a major task of the manuscript, is to analyse such a degenerate optimal transport problem and develop tools to characterize lower Ricci bounds (in particular,  non-negative Ricci in the null directions, which corresponds to the renowned  \emph{null energy condition} in general relativity). In order to illustrate the power of the methods developed in the paper, we will obtain as applications: a stability result under convergence of space-times, a comparison result for null-cones and the Hawking area theorem (both in sharp form, for possibly weighted measures, and with apparently new rigidity statements). 
In a forthcoming work~\cite{CMM24b}, we will use the optimal transport characterization of the null energy condition obtained in the present work as a starting point to develop a theory of non-smooth Lorentzian spaces satisfying such Ricci bounds in a synthetic setting, in the spirit of the ${\mathsf{CD}}(K,N)$ spaces of Lott--Sturm--Villani~\cite{sturm:I, sturm:II, lottvillani}.

Before discussing more in detail the main results of the paper, let us give some background and motivations.

The first part of the paper will be devoted to study volume distortion inside null hypersurfaces and optimal transport thereof. Recall that a hypersurface inside a Lorentzian manifold is said to be \emph{null} if the restriction of the ambient Lorentzian metric has a 1-dimensional kernel. Null hypersurfaces are fundamental in general relativity. Indeed, they play an important role in  understanding the propagation of light and gravitational waves, the causal structure of space-time, the nature of black holes and their horizons, and various theoretical constructs that underpin modern gravitational theory, such as holography and AdS/CFT correspondence. 

The second part of the paper will connect optimal transport inside null hypersurfaces with the Null Energy Condition (NEC), which writes as non-negativity of the Ricci curvature in the null directions. It is essentially a mathematical formulation of the idea that  ``energy density" should be non-negative.  The NEC is an important concept in general relativity, and it is expected to be satisfied by all forms of matter (see for instance~\cite[Ch.~4.6]{Carroll}), at least classically when quantum effects are not taken into account. For instance, it plays a pivotal role in:

\begin{enumerate}
\item \emph{Singularity Theorems}: The NEC is a key assumption in the Penrose singularity theorem~\cite{Penrose65} predicting the formation of singularities under gravitational collapse, awarded the 2020 Nobel Prize in Physics.

\item \emph{Black Hole Thermodynamics}: The NEC is used in proving important results in black hole thermodynamics, including the Hawking area theorem~\cite{Haw71}, which states that the area of the event horizon of a black hole cannot decrease over time, a striking analog to the second law of thermodynamics (see also the more recent revisits by Chru{\'s}ciel--Delay--Galloway--Howard~\cite{CDGH-AHP-2001} and Minguzzi~\cite{MinguzziCMP15}, lowering the regularity assumptions).   
\end{enumerate}

\subsection*{Background on optimal transport and Ricci curvature in Lorentzian setting.}

Optimal transport in Loretzian manifolds was studied, among others, by
Eckstein--Miller~\cite{EM17}, Suhr~\cite{Suhr} and
Kell--Suhr~\cite{KellSuhr}. After developing  further tools for
\emph{time-like} optimal transport of probability measures (i.e.,
pairs of probability measures admitting a coupling concentrated on the
time-like relation) McCann~\cite{McCann} and Mondino--Suhr~\cite{MoSu}
gave an optimal transport characterization of \emph{time-like Ricci
  curvature lower bounds} (resp.\ of the Einstein equations) in terms of convexity properties of the Boltzmann--Shannon entropy along time-like geodesics (with respect to a suitable Lorentz--Wasserstein structure) of probability measures. Building on such a characterization in the smooth setting, Cavalletti--Mondino~\cite{CaMo:20} developed a synthetic theory of non-smooth Lorentzian spaces satisfying time-like Ricci curvature lower bounds  (see also the more recent developments \cite{Braun,Braun-Ohta,Braun-McCann}). The non-smooth setting is given by the class of Lorentzian (pre-)length spaces, introduced by Kunzinger--S\"amann~\cite{KS} (after the work by Kronheimer--Penrose~\cite{CausalSpace} on causal spaces), as a Lorentzian counterpart of (length) metric spaces. For variants in the axiomatization of Lorentzian length spaces, see McCann~\cite{McCann-NEC} and Minguzzi--Suhr~\cite{Minguzzi-Suhr}. 
A related line of research has been to investigate convergence of (possibly non-smooth) Lorentzian spaces, see for instance \cite{Sormani-Vega,Allen-Burtscher,Kunzinger-Steinbauer,Minguzzi-Suhr,Muller}.

Let us stress that the aforementioned works analyzed optimal transport in the \emph{time-like} directions, i.e.,  the transport is performed along time-like geodesics.  A special case of time-like Ricci bounds is the one of non-negative lower bound, also known as the \emph{strong energy condition} (SEC). A weaker energy condition is the \emph{null energy condition} (NEC), requiring the non-negativity of the Ricci curvature in null directions. Unlike the SEC, the NEC is expected to be satisfied by all forms of matter (see, for instance, Carroll~\cite[Ch.~4.6]{Carroll}), at least classically when quantum effects are not taken into account.

A first synthetic characterization of the NEC for smooth spacetimes was given by McCann~\cite{McCann-NEC}, as a limiting case of timelike Ricci lower bounds.  The advantage of such a characterization is that it makes sense for non-smooth Lorentzian length spaces; the drawback, as already remarked in~\cite{McCann-NEC},  is that it is not stable under convergence of spacetimes. Moreover, it remains an open question to draw applications from such a characterization.  

The first time that the NEC was characterized -- for smooth (possibly weighted) spacetimes -- by studying \emph{optimal transport along null hypersurfaces} was in the recent work of Ketterer~\cite{Ket24}, who  analyzed displacement convexity of the $(n-2)$-R\'enyi entropy along Wasserstein geodesics of probability measures concentrated on  space-like $(n-2)$-dimensional submanifolds inside a null $(n-1)$-dimensional hypersurface of an $n$-dimensional spacetime. This corresponds to studying optimal transport in codimension 2 along null geodesics (in the spirit of the Riemannian higher codimension optimal transport of~\cite{KettererMondino}). In \cite{Ket24} classical results, such as Penrose's singularity theorem and Hawking's area theorem, are  recovered via optimal transport.

However, as acknowledged in the introduction of~\cite{Ket24}: ``It is an interesting
question whether our reformulation of the null energy condition can be extended to the
nonsmooth setting of Lorentzian length spaces. At the moment, this is not clear because the
definition requires notions of sufficiently regular null hypersurfaces and null geodesic
congruences."

The goal of the present paper (and the forthcoming companion~\cite{CMM24b}) is to propose a different approach to optimal transport inside null hypersurfaces, by studying \emph{diffused} probability measures along a null hypersurface (more precisely, absolutely continuous probability measures with respect to a reference $(n-1)$-dimensional rigged measure on a null hypersurface; see below for more details). The aim is to    exploit such a diffused optimal transport inside null hypersurfaces in order to give a synthetic characterization of the NEC which is fruitful in terms of applications,  suitable to be extended to non-smooth Lorentzian spaces, and stable under convergence of spacetimes.

\subsection*{Challenges}
The question at the heart of the optimal transport problem is, given two probability measures $\mu_0, \mu_1$ on a space $X$, find the ``optimal way to transport $\mu_0$ into $\mu_1$". A  ``transport" from $\mu_0$ to $\mu_1$ is a probability measure $\pi$ on $X\times X$ (called \emph{coupling}) whose marginals are $\mu_0$ and $\mu_1$, and ``optimal"  is translated in mathematical terms by \emph{minimizing} a suitable cost function. In Riemannian signature (see for instance~\cite{villani:oldandnew}) one usually minimizes the integral of the squared  (or another convex function of the) distance among all couplings from $\mu_0$ to $\mu_1$.  In Lorentzian signature, when studying optimal transport between causally related probability measures, one instead maximizes the integral of a suitable power (now less or equal than 1) of the time separation, among all couplings from $\mu_0$ to $\mu_1$ which are concentrated on the causal relation.

Such choices are motivated by the classical fact that geodesics
minimize (resp.\ maximize) length in Riemannian (resp.\ Lorentzian) signature. Moreover, from the technical point of view, such a choice has the advantage that   minimization (resp.\ maximization) problems are particularly well-behaved for convex (resp.\ concave) Lagrangians.

A first challenge in transporting two probability measures inside a null hypersurface is that the cost $c(\dotargument, \dotargument)$ is completely degenerate: $c(x,y)$ is either $0$, if the pair $(x,y)$ is  causally related, or $\infty$, if $(x,y)$ is not causally related. 

A second fundamental point is that the optimal transport characterization of Ricci lower bounds is formulated in terms of displacement convexity of suitable entropy functionals and, in order to compute the entropy of a probability measure, one needs a reference measure. From the physical point of view, this is because the entropy is a relative concept, of a (probability) measure with respect to a reference one. Here the challenge is that, if one restricts the ambient Lorentzian metric to a null hypersurface, the induced quadratic form has a one dimensional kernel, and thus the induced volume form is identically zero.  Therefore, one has to define a natural non-trivial volume form on null hypersurfaces, which captures the geometric properties (in particular, the Ricci curvature) of the space-time. 

\subsection*{Main results} In the next paragraphs, we briefly illustrate the main results of the paper. For the sake of the introduction, the discussion below will be in a simplified form; the reader is referred to the body of the work for the precise (and more general) statements.

\subsubsection*{\textbf{Rigged measure and Ricci curvature, the $\NC^1(N)$ condition}}

Recall that, given an $n$-dimensional Lorentzian manifold $(M^{n}, g)$, a hypersurface $H\subset M$ is said to be null if the restriction of $g$ to $H$, denoted by $g_H$, has non-trivial kernel. In other terms, the tangent space to $H$ at every point contains a null vector, i.e. a vector $v\neq 0$ such that $g(v,v)=0$. As a consequence, the determinant (and thus the volume form) of $g_H$ is identically zero.

\smallskip
For simplicity, throughout the introduction, we will assume that all the objects are smooth. Moreover, we assume that $H$ is causal (i.e.,there are no closed causal loops contained in $H$) and it admits a global space-like and acausal cross section $S$. We refer to the body of the paper for the more general statements.
\smallskip

In order to define a non-trivial volume form on $H$, we fix a  future-directed null vector field $L$ along $H$ satisfying $\nabla_L L=0$. Such fields $L$ will be called \emph{null-geodesic vector fields}. It is not hard to see that, locally, the class of null-geodesic vector fields is non-empty (\Cref{P:local-null-geodesics-smart}).

We prove that, fixed a null hypersurface $H$ and a null-geodesic vector field $L$, there is a canonical volume form $\vol_L$ on $H$, depending on $L$ (\Cref{cor:rigged-volume}). Moreover, the dependence on $L$ is controlled (\Cref{cor:rescale-volume}): roughly, scaling $L$ by a function $\varphi$ will produce a scaling of the volume by $1/\varphi$. 

The construction is compatible, but independent of, the rigging technique~\cite{Rigging} and~\cite[Sect.~2.6]{MinguzziCMP15} (see ~\Cref{SubSec:RiggingTech} for details). For this reason, $\vol_L$ will be called \emph{rigged measure}.

We point out that the rigged measure $\vol_L$ is mutually absolutely
continuous w.r.t.\ the volume measure induced on $H$ by any auxiliary Riemannian metric.
In the sequel, we will say that a measure $\mu\in\M^+(H)$ is
absolutely continuous w.r.t.\ $H$, if it is absolutely continuous
w.r.t.\ the volume measure on $H$ induced by some (hence any) metric.
Accordingly, the family of probability measures absolutely continuous w.r.t.\ $H$
will be denoted by $\Prob_{ac}(H)$.

The distortion of the rigged meausure $\vol_L$  along null geodesics is strictly related to the Ricci curvature of the ambient space-time in null directions (\Cref{eq:riccati-for-jacobi}). In \Cref{Subsec:WeightedRigged}, we also consider the case of a weighted rigged measure $\mm_L= e^{\Phi} \,\vol_L$, for some continuous function $\Phi:M\to \R$.
For the reader's convenience, the presentation 
in~\Cref{S:volume} and~\ref{Sec:VolDist} is as self-contained as
possible, nevertheless, some parts of the arguments may be well-known
to experts (e.g., the estimates involving Jacobi fieds).

In \Cref{sec:CD1}, in order to analyse the interplay between the rigged measure and Ricci curvature, we partition a null hypersurface into null geodesics. Such a partition will be phrased in a terminology borrowed from  optimal transport, in order to be suitable for generalizations to the non-smooth setting (see the forthcoming~\cite{CMM24b}). The main result of the section, \Cref{T:summary} expresses the (resp. weighted) rigged measure in terms of a disintegration compatible with the aforementioned partition in null geodesics, and makes manifest the connection with the (resp. weighted) Ricci curvature of the ambient space-time in the null directions: namely, the null Ricci dictates the concavity properties of the densities of the measures on the null geodesics in the partition. Let us state the result in a simplified form, referring to  \Cref{T:summary} for the more general (and precise) statement. 

\begin{theorem}\label{T:summaryIntro}
  Let $(M^{n},g)$ be a Lorentzian manifold, let $H$ be a causal null hypersurface admitting a space-like and acausal cross-section $S$, and a null-geodesic vector field $L$.

Then the following representation formula holds
$$
\vol_L =  \int_{S} \left( e^{a_z(t)} \de t \right) \, \mathcal{H}^{n-2}(\de z),
$$
where:
\begin{itemize}
\item $\mathcal{H}^{n-2}$ is the $(n-2)$-Hausdorff measure on the  space-like cross-section $S$ with respect to the Riemannian metric induced by the restriction of $g$ on tangent spaces of $S$;
\item for every $z\in S$, there is a unique (maximally extended) null geodesic $\gamma_z$ contained in $H$, satisfying $\gamma_z(0)=z$, $\frac{\de \gamma_z}{\de t}(t)=L(\gamma_z(t))$ for all $t\in {\rm Dom}(\gamma_z)$;
\item for every $z\in S$,  $e^{a_z(t)} \de t$ is a weighted measure on ${\rm Dom}(\gamma_z)\subset \R$;
\item for every $z\in S$, the function $t\mapsto a_z(t)$ satisfies the concavity property 
$$ 
    a_z''(t)
    +\frac{(a_z'(t))^2}{n-2}
    \leq
    -\Ric^{g}_{\gamma_z(t)}(L,L)
    .
$$
\end{itemize}
\end{theorem}

Motivated by such a result, we introduce our first synthetic notion of null energy condition, denoted by $\NC^1$, which amounts to require that the densities of the measures on the null geodesics in the aforementioned partition are $\log$-concave (in a quantified sense). Here, we give a simplified version, referring to \Cref{def:NC1Quadruples} for a more general one, admitting a $C^0$-weight on the rigged measure and weaker assumptions on the null hypersurface.

\begin{definition}[$\NC^{1}(n)$ condition]\label{def:NC1QuadruplesIntro}
  Let $(M^n,g)$, $H,\, S,\, L$ be as in \Cref{T:summaryIntro}. We say that the triplet $(M,g,H)$ satisfies the null energy condition $\NC^1(n)$ if, for all $z\in S$, the function
$t\mapsto a_z(t)$ is locally-Lipschitz and it  satisfies
\begin{equation}
  a_z''
  +
  \frac{(a_z')^2}{n-2}
  \leq 0
  ,
  \qquad
  \text{ in the sense of distributions}.
\end{equation}
  We say that the space-time $(M^n,g)$ satisfies the
   null energy condition $\NC^1(n)$ if, for
  any null hypersurface $H\subset M$ as above, the triplet $(M,g,H)$ satisfies the null
  energy condition $\NC^1(n)$.
\end{definition}
Since a non-negative lower bound on the  Ricci curvature is invariant under scaling, it is not hard to see that the $\NC^1(n)$ condition is independent of $L$ (see \Cref{Rem:indepNC1L}). Clearly, \Cref{T:summaryIntro} shows that the classical NEC implies $\NC^1(n)$. The reverse implication is proved in \Cref{thm:NC1toNEC}. 
Let us stress that this reverse implication is proved by assuming the $\NC^1(n)$ condition only on (local) future light-cones.
This feature differs form~\cite{Ket24} where the displacement
convexity of the entropy is required for certain null hypersurfaces
constructed ad-hoc, whereas here the $\NC^1(n)$ condition is required
only for light-cones, which are null hypersurfaces with a clean
physical interpretation (namely the events spanned by the light radiating from a point).

The advantage of the $\NC^1(n)$ condition, with respect to the NEC, is that it is suitable to be generalized to the non-smooth setting (see the forthcoming paper~\cite{CMM24b}). Moreover, in \Cref{th:stability} we prove that $\NC^1$ is stable under $C^1$-convergence of the Lorentzian metrics and $C^0$-convergence of the weights on the rigged measures. More general stability results, allowing non-smooth limit spaces, will be established in the forthcoming paper~\cite{CMM24b}.

\subsubsection*{\textbf{Optimal transport inside a null hypersurface}}
One of the novel aspects of the present work is the one regarding optimal transport inside a null hypersurface. As mentioned above, the challenge here is that the cost is completely degenerate (taking only the values $0$ and $\infty$). Nevertheless, the ruled geometry of null hypersurfaces allows us to give a quite precise picture. Let us start with a definition (see \Cref{defn:inside}). Let $e_t: C([0,1]; H)\to H$ be the evaluation map, i.e.,\ $e_t(\gamma):=\gamma(t)$; also, denote by $\tau:M\times M\to [0,\infty)$ the time separation function (i.e.,\ the $\sup$ of the proper time of causal curves joining the two points in the argument of $\tau$, with the convention that $\tau(x,y)=0$ if $x\not\leq y$, where $\leq$ denotes the causal relation), and by $\Pi_\leq(\mu_0, \mu_1)$ the set of causal couplings from $\mu_0$ to $\mu_1$ (i.e., the set of couplings  from $\mu_0$ to $\mu_1$ which are concentrated on the causal relation). In the next definition, we give a slight variation of terminology introduced in~\cite{Ket24}.

\begin{definition}[Transport inside a null hypersurface]
  \label{defn:insideIntro}
 Let $(M,g)$ be a Lorentzian manifold, $H$ a null hypersurface and $\mu_i\in\Prob(H)$, $i=0,1$, be
  two probability measures.
  We say that \emph{$\mu_0$ is null connected to $\mu_1$ along $H$}  if there
  exists a probability measure  $\nu \in \mathcal{P}(C([0,1];H ))$  such that
  \begin{equation}\label{eq:CausnuIntro}
  \pi:=(e_{0},e_{1})_{\sharp} \nu\in \Pi_{\leq}(\mu_0, \mu_1)\,,\quad \tau(x,y)=0 \text{ for $\pi$-a.e. $(x,y)$, }\quad \text{and}\quad \nu\text{-a.e. } \gamma {\rm\ is \ causal}.
  \end{equation}
  We also denote
  \begin{equation}\label{def:OptCausIntro}
\OptCaus^H(\mu_0,\mu_1) := \{ \nu \in \mathcal{P}(C([0,1];H )) \text{ satisfying~\eqref{eq:CausnuIntro}}\}.
\end{equation}
\end{definition}
By the very definition, it holds that $\pi:=(e_{0},e_{1})_{\sharp}
\nu$ is an optimal coupling from $\mu_0$ to $\mu_1$ for  any
Lorentz--Wasserstein distance $\ell_p$, for all $p\leq 1$, (see
Eckstein--Miller~\cite{EM17} for $p\leq 1$; Suhr~\cite{Suhr} for $p=1$;
McCann~\cite{McCann} and Cavalletti--Mondino~\cite{CaMo:20} for
$p<1$). Moreover, from~\eqref{eq:CausnuIntro}, it follows that
$\nu$-a.e.\ $\gamma$ is a null pre-geodesic (i.e., it can be re-parameterized into a null geodesic).

When dealing with the dynamical approach to optimal transport in the
Riemannian setting or in the setting of purely time-like transport,
it is well-known that optimal dynamical transport plans are
concentrated on geodesics.
In the case of transport with  degenerate cost, this is not the case.
For instance, as observed above, one can reparameterize the geodesics where a null optimal
dynamical transport plan is concentrated, still finding an admissible
optimal dynamical transport plan.
We therefore introduce the following definition (see \Cref{def:nullgeodDynTP}).

\begin{definition}[Null-geodesic dynamical transport plan]\label{def:nullgeodDynTPIntro}
 Let $(M,g)$ be a Lorentzian manifold and $H$ a null hypersurface.
  Let $\mu_0,\mu_1\in \Prob(H)$ be two probability measures null connected
  along $H$.
  We say that a dynamical transport plan
  $\nu\in\OptCaus^H(\mu_0,\mu_1)$ is \emph{null-geodesic} if it
  is concentrated on null geodesics.
  The set of null-geodesic dynamical transport plans from $\mu_0$ to $\mu_1$ is denoted by $\OptGeo^H(\mu_0,\mu_1)$. In other terms, 
   \begin{equation*}
\OptGeo^H(\mu_0,\mu_1) := \{ \nu \in \OptCaus^H(\mu_0,\mu_1) \colon \text{ $\nu$-a.e.\ $\gamma$ is a null geodesic}\}.
\end{equation*}
\end{definition}
A particularly useful class of transports is the one of \emph{monotone} couplings, defined below (see \Cref{def:MonotoneCoupling}). Let us denote by  $J \subset X^2$ the subset of causally related pairs in $X\times X$.

\begin{definition} [Monotone couplings and plans]
  Let $\mu_0,\mu_1\in \Prob(H)$ be two probability measures null connected along $H$.
  We say that an optimal coupling $\pi$ between $\mu_0$ and
  $\mu_1$ is \emph{monotone}, if:
  \begin{enumerate}
  \item 
  $\pi=(e_0\otimes e_1)_{\sharp}\nu$ for some
  $\nu\in\OptCaus^H(\mu_0,\mu_1)$;
  \item there exists a Borel set
  $A\subset (H\times H)\cap J$ such that $\pi(A)=1$ and
  \begin{equation}
    \forall (x_1,y_1),(x_2,y_2)\in A
    ,
    \quad
    x_1 \leq x_2
    ,\,
    x_1\neq x_2
    \implies
    y_1 \leq y_2
    .
  \end{equation}
  \end{enumerate}
  A null-geodesic dynamical transport plan $\nu\in\OptGeo^H(\mu_0,\mu_1)$ is said to be \emph{monotone} if the  optimal coupling $\pi:=(e_0\otimes e_1)_{\sharp}\nu$ is monotone, in the above sense.
\end{definition}

The main result of \Cref{sec:OTinsideH} is that for every pair  $\mu_0,\mu_1\in \Prob_{ac}(H)$ of probability measures absolutely continuous w.r.t.\ $H$, such that  $\mu_0$ is null connected to $\mu_1$ along $H$, there exists a unique monotone, null-geodesic dynamical transport plan $\nu$ from $\mu_0$ to $\mu_1$; moreover, the associated optimal coupling  $\pi:=(e_0\otimes e_1)_{\sharp}\nu$ is induced by a map (see \Cref{T:good-representation2} for the precise statement).

\subsubsection*{\textbf{NEC as displacement convexity of the Boltzmann--Shannon entropy}}
Building on top of the results about optimal transport inside a null hypersurface 
obtained in \Cref{sec:OTinsideH}, in \Cref{sec:NCe} we  give a synthetic characterization of the NEC  in terms of displacement convexity of the Boltzmann--Shannon entropy along monotone  null-geodesic optimal transport plans. Below, we give a simplified version of the result, referring to \Cref{T:convexity-of-entropy-null} for the general (and more precise) statement. Recall that the Boltzmann--Shannon entropy $\Ent(\mu|\mm)$ of a probability measure $\mu\in \Prob(M)$ with respect to the reference measure $\mm$ is defined by
\begin{equation*}
\Ent(\mu|\mm)=\int \rho \log \rho \, \de\mm,
\end{equation*}
if $\mu=\rho\,\mm$ is absolutely continuous w.r.t.\ $\mm$ and $\left(\rho\, \log \rho\right)_+$ is $\mm$-integrable; otherwise, we adopt the convention that $\Ent(\mu|\mm)=+\infty$. We will consider the dimensional variant 
\begin{equation*}
{\rm U}_N(\mu|\mm):=\exp\left(-\frac{1}{N}\Ent(\mu|\mm) \right),
\end{equation*}
which was studied in~\cite{EKS} in connection to Ricci bounds. This kind of functional is well-known in information theory as the
Shannon entropy power (see e.g.,~\cite{DCT-1991}).
For instance, the entropy power on $\R^N$ is the functional ${\rm U}_{N/2}$.

\begin{theorem}
  \label{T:convexity-of-entropy-null-Intro}
  Let $(M,g)$ be a Lorentzian manifold of
  dimension $n$ satisfying the NEC, and let $H\subset M$ be a null hypersurface endowed with a rigged measure $\vol_L$.
  Let $\mu_0,\mu_1\in \Prob_{ac}(H)$ be two probability measures
  such that $\mu_0$ is null connected to $\mu_1$ along $H$
  and let $\nu\in\OptGeo^H(\mu_0,\mu_1)$ be the unique monotone null-geodesic
  dynamical transport plan joining them.
  Denote by $\mu_t:=(e_t)_{\sharp}\nu$, $t\in [0,1]$, and let  
${\rm u}_{n-1}(t):= {\rm  U}_{n-1}(\mu_t|\vol_L)$.

Then  the function $[0,1]\ni t \mapsto {\rm u}_{n-1}(t)$ satisfies 
\begin{equation}\label{eq:u(t)Conv}
{\rm u}_{n-1}(t) \geq (1-t)\, {\rm u}_{n-1}(0)+t\, {\rm u}_{n-1}(1), \qquad \forall t\in [0,1].
\end{equation}
\end{theorem}

\Cref{T:convexity-of-entropy-null-Intro} suggests the following two definitions, one
concerning null hypersurfaces and one concerning the ambient space-time. We give here a simplified version, referring to \Cref{D:entropynullCD} and \Cref{def:NCEspace} for the more general (and precise) notions, also allowing a $C^0$-weight on the rigged measure.

\begin{definition}[The $\NC^e(n)$ condition]\label{D:entropynullCD-intro}
  Let $(M,g)$ be a Lorentzian manifold
  and $H\subset M$ be a null hypersurface endowed with a rigged measure $\vol_L$. 

  We say that the triplet $(M,g,H)$ satisfies the
  $\NC^e(n)$ condition if the following holds:
  For every pair of probability measures $\mu_0,\mu_1\in \Prob_{ac}(H)$ 
  such that $\mu_0$ is null connected to $\mu_1$ along $H$,
  there exists  a null-geodesic
  dynamical transport plan $\nu\in\OptGeo^H(\mu_0,\mu_1)$ satisfying~\eqref{eq:u(t)Conv}.

  We say that $(M,g)$ satisfies the  $\NC^e(n)$ condition if every null hypersurface does so.
\end{definition}

It is not hard to see that the definition of $\NC^e(n)$ does not depend on the choice of $L$ (\Cref{rem:NCeNoL}). Moreover NEC, $\NC^1(n)$ and $\NC^e(n)$ are all equivalent for smooth space-times: we already recalled that $\NC^1(n)$ is equivalent to NEC, \Cref{T:convexity-of-entropy-null}  proves the implication $\NC^1(n)\Rightarrow \NC^e(n)$ and, finally, the reverse implication $\NC^e(n)\Rightarrow \NC^1(n)$ is established in \Cref{th:entropic-to-distributional}.

\subsubsection*{\textbf{Applications}} As applications of the theory
developed in the present paper, we obtain weighted versions (in low
regularity for the weight, assumed merely to be continuous) of the
light-cone theorem and of the Hawking area theorem. Both in sharp form
and with an apparently new rigidity statement for the
  former. The proofs build on the tools developed in the paper. Indeed, for a continuous weight, the classical NEC is not well-defined (as it involves a linear term in the second derivatives, and a quadratic term in the first derivatives of the weight); as a replacement, we assume its synthetic counterparts $\NC^1$ or $\NC^e$ discussed above.
\smallskip

\textit{The weighted light-cone theorem}. The light-cone theorem, first established by Choquet-Bruhat, Chru\'sciel and Mart\'in-Garc\'ia~\cite{CBCMG-2009}, states that the area of cross-sections of light-cones, in spacetimes satisfying the Einstein equations with vanishing cosmological constant and with the energy–momentum tensor obeying the dominant energy condition, is bounded above by the area of corresponding  cross-sections of light-cones in Minkowski spacetime. The rigidity question  is addressed in~\cite{CBCMG-2009} under additional assumptions on the energy–momentum tensor. 

In a subsequent work, Grant~\cite{Grant}, revisited the light-cone theorem: the assumption was relaxed to asking the NEC, and the conclusion was sharpened to a monotonicity statement, reminiscent of the Bishop--Gromov  volume comparison in Riemannian geometry. Moreover,~\cite{Grant} proves the sharpness of the result (equality is attained on model spaces), leaving open the question of the rigidity. 

In \Cref{SS:WLCT}, we extend the light-cone theorem in the sharp monotone form to Lorentzian manifolds with a continuous weight on the rigged measure satisfying $\NC^e(N)$ (see \Cref{T:WLCT}), and we address the rigidity question by showing that equality is attained if and if the light-cone is isometric to the light-cone in Minkowski spacetime  (see \Cref{T:WLCT-rigid}).   
\smallskip

\textit{The weighted Hawking area theorem.}
In \Cref{SS:WHAT}, we provide an extension of  the celebrated Hawking's area theorem to Lorentzian manifolds with a continuous weight on the rigged measure, satisfying $\NC^1(N)$. We also obtain an apparently new rigidity statement: equality is attained in the Hawking area theorem if and only if the metric on the horizon of the black hole is static (in the sense that it splits isometrically as a product). Let us briefly frame our result in the perspective of the existing literature.

The area theorem was proved by Hawking~\cite{Haw71} in 1971; in its original formulation it states that, in the setting of smooth space-times satisfying the null energy condition, the area of a smooth black hole horizon can never decrease.

In classical general relativity, the result paved the way to black hole thermodynamics; indeed, the area of a black hole horizon is interpreted as a measure of its entropy, and the second law of thermodynamics states that entropy does not decrease in time.

The result was revisited  by  Chru{\'s}ciel--Delay--Galloway--Howard~\cite{CDGH-AHP-2001} and Minguzzi~\cite{MinguzziCMP15}, lowering the regularity assumptions on the null hypersurface (the ambient spacetime is assumed to be smooth) and discussing rigidity:~\cite{CDGH-AHP-2001} proves that if equality is attained then the second fundamental form of the null hypersurface has to vanish; ~\cite[Th.~12,~Th.~13]{MinguzziCMP15} show that equality holds in the area theorem if and only if the portion of the null hypersurface between the two cross-sections achieving equality can be written as a graph of a $W^{2,1}$-function with vanishing hessian.

In \Cref{Thm:WHAT}, the assumption required in the classical form of Hawking's Theorem that the Ricci curvature is non-negative on all null vectors in the space-time, is replaced by the weaker null energy condition  $\NC^1(N)$ of \Cref{def:NC1QuadruplesIntro} (actually, the more general \Cref{def:NC1Quadruples}). Moreover, we allow a $C^0$ weight on the rigged measure, and  we obtain the rigidity in an apparently new formulation (see also \Cref{rem:InterpRigidity}).

Let us mention that Ketterer~\cite{Ket24} recently extended Hawking's area theorem to smooth weighted space-times satisfying the NEC. The methods in~\cite{Ket24} are also based on optimal transport; however, the approach in the present work is more tailored to extensions to non-smooth spacetimes.  Indeed, building on the tools developed in the present paper, in the forthcoming work~\cite{CMM24b} we will push the results beyond the differentiable setting by extending the $\NC^1$ condition to Lorentzian length spaces, and establishing Hawking's area theorem in such a synthetic framework.

\subsection*{Plan of the paper}
The paper is organized as follows.
In \Cref{S:conventions} we recall a few facts about null
hypersurfaces; in \Cref{S:volume} we investigate the rigged volume
measure, in particular its interplay with the Jacobi fields.
Sections~\ref{Sec:VolDist} and~\ref{sec:CD1} investigate the
distortion of the (weighted) rigged measure, culminating in the
definition of $\NC^1(N)$.
In Sections~\ref{sec:OTinsideH} and~\ref{sec:NCe} we study optimal
transport inside null hypersurfaces and the displacement convexity of
the entropy.
\Cref{S:converse} proves the equivalence of $\NC^1(N)$, $\NC^e(N)$,
and NEC; \Cref{S:stability} presents a stability result for $\NC^1$.
Finally, in \Cref{S:applications} we present two applications: the
light-cone theorem and Hawking area theorem.

\section{Conventions and preliminaries}
\label{S:conventions}

In this section we recall a few facts regarding the geometry of
Lorentzian manifolds.

\subsection{General facts of Lorentzian geometry}

A Lorentzian manifold is a couple $(M,g)$, where $M$ is a smooth
$n$-dimensional manifold (the case $n=4$ corresponds to classical general relativity) and $g\in\vfield (T^*M^{\otimes 2})$ is a Lorentzian
product, i.e., a symmetric non-degenerate bilinear form, whose
signature is $(- +\cdots + )$.
In order to keep the presentation simple, throughout the
paper, $g$ is always tacitly assumed to be of class $C^2$.
However, in certain technical lemmas, we may point out the validity of the result in lower regularity.
For instance, let us just mention that if $g$ is not $C^{1,1}$, one needs a suitable
definition of exponential map (see~\cite{Graf}).

Note that, for $C^{2}$ Lorentzian metrics, the natural class of differentiability of the manifolds is $C^{3}$. However, a $C^{3}$ manifold always possesses a $C^{\infty}$-sub-atlas, and one can choose some such sub-atlas whenever convenient. Thus, the smoothness assumption on $M$ is not restrictive.
A vector $X\in T_xM$ will be called
\begin{itemize}
\item
  time-like, if $g(X,X)<0$;
\item
  space-like, if $g(X,X)>0$ or $X=0$;
\item
  light-like, if $g(X,X)=0$ and $X\neq 0$;
\item
  causal, if it is either time-like or light-like.
\end{itemize}
We say that a Lorentzian manifold $(M,g)$ is \emph{time-oriented}, if there
exists a continuous time-like vector field. 
In this paper, all Lorentzian manifolds are assumed to be time-oriented.
If $X\in T_xM$ is a causal vector, we say that $X$ is \emph{future-directed},
if $g(X,Y)<0$, where $Y$ is a vector giving the temporal orientation to
the manifold.

One can always endow a Lorentzian manifold with an auxiliary smooth
Riemannian metric; locally-Lipschitz regularity will always be understood w.r.t.\ some (hence any) auxiliary smooth Riemannian metric.
We say that a locally-Lipschitz curve $\gamma:I\to M$ is
\emph{causal}, if $\dot\gamma$ is future-directed (at the differentiability points); we say that it is
\emph{chronological}, if it is causal and $g(\dot\gamma,\dot\gamma)<0$. We say that a Lorentzian manifold is \emph{causal} if there exists no periodic
causal curve.

A causal curve is called \emph{inextensible}, if its domain of definition $I$
cannot be extended.
If $x,y\in M$, we say $x\leq y$ (resp.\ $x\ll y$), if there exists a
causal (resp.\ chronological) curve connecting $x$ to $y$.
The (Lorentzian) \emph{length} of a causal curve $\gamma:I\to M$ is defined as
\begin{equation}
  \Length(\gamma)
  :=
  \int_I
  \sqrt{-g(\dot\gamma_t,\dot\gamma_t)}
  \,
  \de t;
\end{equation}
the \emph{time-separation} (or \emph{Lorentzian distance}) between $x, y\in M$ is given by
\begin{equation}
  \tau(x,y)
  :=
  \begin{cases}
    \sup
    \Length(\gamma)
    \quad&
           \text{ if } x\leq y,
    \\
    0
    &\text{ otherwise},
  \end{cases}
\end{equation}
where the supremum is taken among all causal curves $\gamma$ connecting $x$ to $y$.

A causal curve $\gamma:I\to M$ is called a \emph{geodesic} if it 
satisfies the equation of geodesics $\nabla_{\dot\gamma}\dot\gamma=0$, 
and thus it can be expressed
  using the exponential map.

\smallskip
A submanifold $S\subset M$ is called \emph{space-like} (resp.\ \emph{null}) if the metric $g$
restricted to the tangent space $TS$ is positive definite (resp.\
degenerate).
In the following, all submanifolds will be assumed to be
  of class $C^2$, unless otherwise stated.
The normal bundle of a submanifold $S$ is denoted by
$$\normal S:=\{v\in TM: g(v,w)=0,\forall w\in TS\}.$$
The normal bundle is of class $C^1$.
Indeed, a local basis of $\normal S$ can be built as follows: fix
$v_1,\dots, v_j$ a local basis for $TS$ of class $C^{1}$ and add
$v_{j+1},\dots, v_n$ obtaining a local trivialization for $TM|_S$; then
apply the Gram--Schmidt orthogonalization process and obtain a basis
for $\normal S$. 
Throughout the paper, sections of class $C^k$ of a vector bundle $F$ (of at least the same class) will be denoted by $\vfield^k(F)$.
In the sequel, unless otherwise stated, all normal vector
  fields will be assumed to be of class $C^1$.

We now recall a classical fact for submanifolds of co-dimension $1$ in Lorentzian signature.

\begin{proposition}[Pag.~45~in~\cite{HawEll}]\label{Prop:NHinTH}
  Let $(M,g)$ be a Lorentzian manifold.
  Let $H\subset M$ be a hypersurface.
  Then the following are equivalent
  \begin{enumerate}
  \item $H$ is a null hypersurface, i.e., $g|_{TH^{\otimes 2}}$ is
    degenerate;
  \item
    $\normal H\subset TH$.
  \end{enumerate}
\end{proposition}

  \begin{remark}
    Proposition \ref{Prop:NHinTH} holds  under the milder assumptions $H$
    of class $C^1$ and $g\in C^0$.
  \end{remark}

\subsection{Null hypersurfaces}\label{Ss:null}

In this section we recall some basics on the geometry of null hypersurfaces. 
Let $g$ be Lorentzian metric on $M$, and let $H\subset M$ be a null
hypersurface.
Since $g$ restricted to $TH$ is degenerate, then $H$ does
not possesses a Levi--Civita connection.
However, we can restrict the Levi--Civita connection $\nabla$ of $g$
obtaining a bilinear operator, still denoted by $\nabla$, $\nabla: \vfield^1(TH)\times \vfield^1(TH)
\to\vfield^0 (TM|_{H})$.

More precisely, if $X,Y\in\vfield^1(TH)$ are vector fields tangent to
$H$, we define
$$\nabla_X Y:=\nabla_X \tilde Y\in\vfield^0(TM|_H),$$
where
$\tilde Y\in\vfield^1(TM)$ is any extension of $Y$. 
The
definition is well-posed (see for instance~\cite{GallotHulinLafontaine04}).

%
%

This operator inherits a few properties from the
Levi--Civita connection. For instance, for all  $X,Y,Z\in\vfield^1(TH)$ and $f\in C^1(H)$, it holds:
\begin{align*}
  &
    \nabla_{fX} Y
    =f\nabla_X Y
    ,
    \quad
    \nabla_X (fY)
    =X(f) Y+f\nabla_X Y
    ,
  \\
  &
    X( g(Y,Z))
    =
    g(\nabla_X Y,Z)
    +
    g(Y,\nabla_X Z)
    ,
  \\
  &
    \nabla_X Y
    -\nabla_Y X
    =
    [X,Y]
    .
\end{align*}

The following proposition 
states a fundamental property of vector fields orthogonal 
to a null hypersurface. 
The proof is included for readers' convenience. 

\begin{proposition}
\label{P:proportional}
  Let $g$ be Lorenztian metric on the smooth manifold $M$.
  Let $H \subset M$ be a null hypersurface and let $L\in
  \vfield^1(\normal H)$ be a normal vector field.
  Then it holds that $\nabla_LL\in\vfield^0( \normal H)$, i.e.,
  $\nabla_LL$ is normal to $H$.
\end{proposition}
\begin{proof}
  Let $X\in \vfield^1(TH)$ and compute
  \begin{align*}
    g(X,\nabla_L L)
    &
      =
      L(g(X,L))
      -
      g(\nabla_L X,L)
      =
      -
      g(\nabla_L X,L)
      =
      -
      g(\nabla_X L,L)
      -g([X,L],L)
    \\&
      =
      -
      g(\nabla_X L,L)
    =-\frac{1}{2}
    X(g(L,L))
    =0,
  \end{align*}
  having used the properties of the operator $\nabla$, that
  $g(X,L)=0$,
  and that $[X,L]$ is a vector field tangent to $H$ (as $L\in TH$, by \Cref{Prop:NHinTH}) hence
  normal to $L$.
\end{proof}
\begin{remark}
Proposition \ref{P:proportional} holds under the weaker
    assumption $g\in C^1$.
  \end{remark}

Since $g$ is non-degenerate, the null
direction is one-dimensional and therefore \Cref{P:proportional} implies 
that $\nabla_L L = h L \in \vfield^0(TH)$ for some real valued function $h$.
Locally, one can reparameterize the vector field $L$ and obtain $h = 0$.

\begin{proposition}
\label{P:local-null-geodesics-smart}
Let $g$ be Lorenztian metric on the smooth manifold $M$, and let
$H\subset M$ be a null hypersurface.
Then locally there exists a normal vector field $L\in\vfield^{1}(\normal H)$, such that 
$\nabla_LL=0$ and $L$ is future-directed.
\end{proposition}

\begin{proof}
The time-oriented-ness of $(M,g)$ (recall this is a standing assumption) and the regularity of $g$ and $H$ imply the existence of a future-directed $\tilde L \in \vfield^1(\normal H)$ that never vanishes.
  Without loss of generality, we can assume that $\tilde L$ has unit norm, w.r.t.\ an auxiliary smooth
  Riemannian metric.
  Let $S\subset H$ be a local space-like cross-section.
  Since we are looking for a local vector field, up to restricting
  $H$, we can assume that $S$ is a  global $C^2$ cross-section for $H$.
  Let $G:H\to \R$ be a $C^{2}$ function such that 
  $G^{-1}(0)=S$ and $\de G\neq 0$
  along $S$.
  Consider the function 
  $$F:\Dom F\subset H\times \R\to\R, \quad F(x,t)=G(\exp_{x}(t\tilde L(x))).$$ 
   Since $\tilde{L}(x)$ is a null vector, both tangent and normal to the null hypersurface $H$ at $x$, then $\exp_x(t\tilde L(x))\in H$, for $|t|$ small enough. Therefore,
  the domain of definition of $F$
  contains a neighborhood of $S\times\{0\}$. 

  Of course, $F$ is of class $C^1$ and satisfies
  \begin{equation}
    F(x,0)=0
    \qquad
    \text{ and }
    \qquad
    \frac{\partial F}{\partial t}(x,0)=\de G (\tilde L(x))\neq 0
    ,
    \qquad
    \forall x\in S
    .
  \end{equation}
  We are in position to apply the Implicit Function Theorem,  yielding a
  neighborhood $U\subset H$ of $S$ and a $C^1$ function $t:U\to
  \R$, such that $F(x,t(x))=0$.

  Define the function $\lambda:U\to \R$, in such a way that
  \begin{equation}
    \left.\frac{\de}{\de t}\right|_{t=t(x)}
    \exp_x(t\tilde L(x))
    =
    \lambda(x)
    \tilde L_{\exp_x(t(x)\tilde L(x))}
    .
  \end{equation}
  The function $\lambda$ is $C^1$, because it coincides
  with the norm (in the auxiliary metric) of the vector
  $\left.\frac{\de}{\de t}\right|_{t=t(x)} \exp_x(t\tilde L(x))$ (recall $\tilde L$ is of unit norm).
  Moreover, $\lambda(x)=1$ for all $x\in S$ and $\lambda>0$
  everywhere.
  Therefore,  we can define
  $L:=\lambda^{-1} \tilde L\in\vfield^1(\normal H)$.
  Its flow is geodesic, because it is the velocity field of the
  geodesic flow starting from points in $S$ with initial velocity
  $\tilde L$.
  For this reason $\nabla_L L=0$.
\end{proof}

\begin{remark}\label{R:nullgeodesic}
\Cref{P:local-null-geodesics-smart} proves the following: for any point $z \in H$ 
there exists an open (in $H$) neighborhood $Z \subset H$ of $z$
and a section $L\in \vfield^1(\normal H)$ that is also future-directed, different from the zero vector in $Z$
and $\nabla_L L = 0$ in $Z$.
For the elements of this particular class of normal vector fields we will adopt the terminology  of  \emph{null-geodesic vector fields}.
They will be tacitly considered only on the open set where $\nabla_L L = 0$ holds true.
\end{remark}

We now introduce the notion of transverse function.
\begin{definition}\label{D:transverse}
  Given a null hypersurface $H\subset M$, we say that a function $\varphi :H\to\R$ is
  \emph{transverse} if for all $x\in H$ and all $v\in \normal_x H\backslash \{0\}$,
  the composition $\varphi \circ \gamma$ is constant, where $\gamma$ is the
  geodesic in $M$ (and contained in $H$)  with $\gamma_0=x$ and $\dot \gamma_0=v$.
\end{definition}
The set of transverse functions is closed under point-wise sums and products.
Note that if $\varphi:H\to \R$ is $C^1$, being transverse is equivalent
to $L(\varphi)=0$, for some (hence any) normal vector field $L$.
\begin{remark}
  \label{R:change-null-geodesics}
  If $L\in\vfield^1(\normal H)$ is a null-geodesic vector field and
  $\varphi\in C^1(H)$ is a transverse positive function, then $\varphi L$ is still a
  null-geodesic vector field.

  Conversely, if $L_1,L_2\in \vfield^1(\normal H)$ are two null-geodesic vector fields on the same open set $U$, then there exists a 
  $C^1$ function $\varphi$, such that $L_1=\varphi L_2$.
 If in addition the intersection of $U$ with each null geodesic contained in $H$ is connected, then $\varphi$ is transverse.
 \end{remark}

Given a normal vector field $L\in \vfield^1(\normal H)$, it is natural to look for submanifolds whose tangent complement $L$ in $TH$.  
These will be local space-like cross-sections of which we recall the definition.
\begin{definition}
A submanifold $S\subset H$ of dimension $n-2$ is 
a \emph{local cross-section of} $H$  if the metric $g$ is positive definite on $TS$.  We will always assume that $S$ is either closed without boundary, or an open $(n-2)$-manifold.
\end{definition}
It is immediate to verify that at each point $x$ of $S$, it holds that 
$T_xH = T_x S \oplus \Span \{L(x)\}$, where $L(x)\in \normal_x H\backslash \{0\}$.
We say that a cross-section $S$ is \emph{pre-compact} if $\overline
S$ is compact (the closure is intended in $H$) and there exists a cross-section $\tilde S$
such that $\overline{S}\subset \tilde S$.

\smallskip

Given a cross-sections $S$ and a normal vector field $L$ (not necessarily null-geodesic), one can construct a new normal vector field
$\overline{L^S}$  along  $S$ in the following way. 
Define
$\overline{L^S}$  as the unique vector field in $\normal S$ verifying the following conditions:
\begin{align}\label{eq:normalS}
  &
  g(\overline{L^S}, X)=0,
  \qquad
  \forall X\in TS
  ,
  \\&
  \label{eq:pseudo-orthogonality}
  g(\overline{L^S}, L)=-1
  ,
  \quad
  \text{ and }
  \quad
  g(\overline{L^S},\overline{L^S})
  =0.
\end{align}
Notice that the construction of $\overline{L^S}$ depends both on $L$ and $S$.
Since $S$ is of class $C^{2}$ and $L$ of class $C^{1}$, then $\overline{L^S}$ is of class $C^{1}$ as well. 

If $L$ is a null-geodesic vector field (see \Cref{R:nullgeodesic}),  using parallel transport along the flow of $L$, 
$\overline{L^S}$ can be then defined in a suitable neighborhood  $Z \subset H$ of any point of $z$ of the submanifold $S$ keeping~\eqref{eq:pseudo-orthogonality} valid and preserving the  $C^{1}$ regularity. 
Moreover, if $\phi$ is a transverse function, then $\overline{(\phi
  L)^S}=\frac{1}{\phi} \overline{L^S}$.

\begin{remark} \label{rmrk:global-null-geodesics}
Let $S$ be a local cross-section for $H$ such that the (null-)integral lines of $L$ meet with $S$ at most once.   Then it is possible to construct a null-geodesic vector field, coinciding with $L$ along $S$ and whose integral lines have non-empty intersection with $S$, as follows. 
For each $x \in S$ consider 
$\gamma_{x,t} : = \exp_x(tL)$. Then, setting $\tilde L: = d\gamma_{x,t}/dt$, it is straightforward to check that $\tilde L$ is null, see \Cref{P:proportional}, and $\nabla_{\tilde L}\tilde L =0$.
\end{remark}

The three objects considered so far -- $L$, $S$ and $\overline{L^S}$ -- will be extensively used in the computations below. In the next section, we report on a slightly different method to build such objects.

\subsubsection{The rigging technique}\label{SubSec:RiggingTech}
A popular approach used for investigating null hypersurfaces and to construct Riemannian metrics on them, is the \emph{rigging technique}.
The data are a null hypersurface and a transverse vector field to it called the \emph{rigging vector}.
The rigging vector field produces a tangent null vector field, called
the \emph{rigged vector}, and a Riemannian tensor, called the \emph{rigged
  metric}. As a byproduct  one also obtains a volume form.

The approach presented so far in \Cref{Ss:null} followed the opposite direction:
from a normal vector field $L$ and a local cross-section $S$, we obtained  a vector field $\overline{L^S}$ transverse to $H$.

Following~\cite{Rigging} (see also~\cite[Sect.~2.6]{MinguzziCMP15}), we now use the rigging technique by choosing as the rigging vector the
transverse vector field $-\overline{L^S}$ that we assume to be defined on an open set $Z\subset H$; 
motivated by the previous discussion, we assume $-\overline{L^S}$  of class $C^{1}$.
As a final outcome, the rigged vector field will be $L$.

Let $\omega$ be the $g$-dual of $-\overline{L^S}$ in $TM|_Z$, i.e.,
$$
\omega(v)=g(-\overline{L^S},v),
\quad \text{for all } v\in TM|_Z. 
$$ 
Let $\alpha$ be
the restriction of $\omega$ to $TH$, i.e., $\alpha : = i^* \omega$ where $i : Z\subset  H \to M$
is the inclusion map. %
Notice that $\ker\alpha|_S \cap Z=TS$, see 
\eqref{eq:normalS}, and that $\alpha$ is of class $C^{1}$.

Consider the rigged metric $\tilde{g}_L^S:= g+\alpha\otimes\alpha\in
\vfield^{1} (T^*Z^{\otimes 2})$.
Since $\overline{L^S}$ is transverse to $H$, then 
$\tilde{g}_L^S$ is a Riemannian metric for $Z$. 
The rigged vector is obtained as the $\tilde{g}_L^S$-dual of the
form $\alpha$; it is immediate to see that it is of class $C^{1}$ and it coincides with $L$:
\begin{equation}
  \alpha(v)
  =
  g(-\overline{L^S},v)
  =
  g(L,v)
  +
  g(-\overline{L^S},L)\, 
  g(-\overline{L^S},v)
  =
  \tilde{g}_L^S(L,v)
  ,\quad
  \forall v\in TH
  ,
\end{equation}
having used the fact that $L$ is orthogonal to all vectors in $TH$
and~\eqref{eq:pseudo-orthogonality}.  Notice also that
\begin{equation}
  \label{eq:rigged-metric}
  \tilde{g}^S_L
  (v,w)
  =
  g(v,w)
  ,
  \quad
  \tilde{g}^S_L
  (v, L)
  =
  0,
  \quad
  \tilde{g}^S_L
  (L, L)
  =
  1
  ,
  \quad
  \forall v,w\in \ker\alpha
  .
\end{equation}

\begin{definition}\label{def:tildeVol}
We denote by $\tilde{\vol}_L^S \in \mathcal{M}^{+}(Z)$ the volume measure induced by the Riemannian rigged
metric $\tilde{g}_L^S$ defined over the open set $Z \subset H$. 
\end{definition}

We will prove that $\tilde{\vol}_L^S$ depends only on the null-geodesic vector $L$ and not on the rigging vector $-\overline{L^S}$.

\subsection{Local Parameterization}
\label{SS:local-parameterization}
Let $L$ be a local null-geodesic vector field on $H$. We write 
\begin{equation}\label{eq;defPsiL}
\gflow_L(z,t):=\exp_z(tL), \quad \forall z\in H,
\end{equation}
whenever this exponential makes sense and
belongs to $H$.
It is immediate to check that the map $\gflow_L$ satisfies (on its domain):
\begin{align*}
  &
    \gflow_L(z,t+s)
    =
    \gflow_L(\gflow_L(z,t),s)
    \quad
    \text{ and }
    \quad
    \gflow_{\phi L}(z,t)
    =
    \gflow_L(z,\phi(z)t)
    ,
    \qquad
    \forall
    \phi
    \text{ transverse}
    .
\end{align*}
It is also clear that, for small $t$, $(z,t)\mapsto \gflow_L(z,t)$ is a
local diffeomorphism on $H$.

We now construct the local parameterization that will be used in the computations of the next sections.
Let $f:U\subset\R^{n-2}\to S$ be a local parameterization of $S$ (recall that $S$ is assumed to be $C^{2}$).
We define 
$$\lparam(x,t):=\gflow_L(f(x),t), \quad \forall x\in U,$$ 
giving a local
parameterization for $H$. In particular, we denote by $V \subset U\times \R$  an open set where 
$\lparam$ is a diffeomorphism.  
The definition of $\lparam$ depends on the choice of $L$ and  $S$, however,  to ease the notation, we  omit  these dependences.

With a slight abuse of notation, we extend the function $\gflow_L$ to a
larger domain by setting $\gflow_L(z,0,s):=\exp_z(s\overline{L^S}).$
We now transport the vector field $L$ to $\gflow_L(z,0,s)$ by parallel
transport along the geodesic $s\mapsto \gflow_L(z,0,s)$.
Then we define 
$$\gflow_L(z,t,s):=\exp_{\gflow_L(z,0,s)}(tL).$$
This new definition of $\gflow_L$ is compatible with the former one, as
$\gflow_L(z,t,0)=\gflow_L(z,t)$, and the newly defined map is a local
diffeomorphism.
Analogously, we extend $\lparam$ by defining 
\begin{equation}\label{eq:paramM}
\lparam(x,t,s):=\gflow_L(f(x),t,s),
\end{equation}
so that $\lparam(x,t,0)=\lparam(x,t)$.
Thus $\lparam$ is  a local diffeomorphism on $M$. We let 
\begin{equation}\label{eq:defOmegaS}
\Omega_S:=\lparam(V\times\{(0)\})\subset Z \subset H.
\end{equation}
It should be noticed that the definition of $\Omega_S$ does not depend
on the choice of the parameterization $f$.

\section{Volume measures on Null hypersurfaces}

\label{S:volume}

The rigged volume measure $\tilde{\vol}_L^S$ is well-defined on the set $\Omega_S$ as in~\eqref{eq:defOmegaS}; by definition, in local coordinates,  
its density with respect to the Lebesgue measure can be written  in terms of the square root of the determinant of the metric $\tilde g_{L}^{S}$.
A convenient way to write this expression is based on classical 
Jacobi fields calculations. 
 For the reader's convenience, the presentation 
in this section is as self-contained as
possible; nevertheless, some parts of the arguments may be well-known
to experts (e.g., Jacobi fields estimates).

\subsection{Jacobi fields in null hypersurfaces}
\label{S:jacobi}

Let $H\subset M$ be a null hypersurface and let $L\in\vfield^{1}(\normal
H)$ be a null-geodesic vector field.
Given $z\in H$, a Jacobi field starting from $z$ is a $C^2$-function
$J:I\to TM|_H$, such that $J(t)\in T_{\gflow_L(z,t)} M$ and $J$ solves the
Jacobi equation
\begin{equation}
  \label{eq:jacobi}
  J''(t)
  +
  R(L,J(t))L=0,
\end{equation}
where $R$ is the Riemann curvature tensor of the Lorentzian metric
$g$. Since the ODE~\eqref{eq:jacobi} is \emph{linear} in $J$, and both $R$ and $L$ are continuous, the associated initial value problem has a unique local solution, thanks to the classical Cauchy-Lipschitz theorem.  
If $v\in T_zM$, we denote by $J_{v,L}(z,t)$ the Jacobi field
solving~\eqref{eq:jacobi}, coupled with the initial conditions
\begin{equation}
  \label{eq:initial-jacobi}
  J_{v,L}(z,0)=v
  \quad
  \text{ and }
  \quad
  J_{v,L}'(z,0)=\nabla_v L
  .
\end{equation}
Notice that if $\varphi$ is a $C^1$ transverse function, then
$J_{v,\varphi L}(z,t)=J_{v,L}(z,\varphi(z)t)$.
To ease the notation, we will omit the dependence on $z$ and $L$,
whenever no confusion arise.

We recall a few facts about Jacobi fields.
The first is that $L$ itself is a Jacobi field, with
$J_L(t)=L$.
Moreover,  Jacobi fields can be characterized as the solutions of the
following first-order ordinary differential equation:
\begin{equation}
  \label{eq:first-order-jacobi}
  J_v'(t)
  =
  \nabla_{J_v(t)} L
  ,
  \qquad
  J_v(0)=v
  ;
\end{equation}
the proof of this fact can be obtained by deriving once and using
the definition of the Riemann curvature tensor.

Recall the next useful result, often referred as Gauss' Lemma.
\begin{lemma}
  \label{lem:gauss}
  For all $v\in T_zM$, the map $t\mapsto g(J_v(t),L)$ is constant.
\end{lemma}
\begin{proof}
  If we derive twice $g(J_v(t),L)$, we obtain $-g(R(L,J_v(t))L,L)$,
  which is null by skew-symmetry of the Riemann curvature tensor.
  The first derivative computed at $t=0$ is:
  \begin{equation*}
    \left.
      \frac{\de}{\de t}
    \right|_{t=0}
    g(J_v(t),L)
    =
    g(J_v'(0),L)
    =
    g(\nabla_v L,L)
    =
    \frac{1}{2}
    v(g(L,L))
    =0.
    \qedhere
  \end{equation*}
\end{proof}

\begin{corollary}
  If $v\in TH$, then $J_v(t)\in TH$ for all $t$.
\end{corollary}
\begin{proof}
  Being tangent to the hypersurface is equivalent to be orthogonal to
  $L$, and the orthogonality to $L$ is conserved by \Cref{lem:gauss}.
\end{proof}

Given a basis  $e_1(z)\dots e_{n}(z)$  for $T_zM$,
using parallel transport along the curves $t \mapsto \gflow_L(z,t)$ we can extend $(e_i(z))_i$ to $\gflow_L(z,t)$.
Define the $i$-th Jacobi vector field as $J_{i,L}(z,t) : =J_{e_i,L}(z,t)$.
We express the field $J_{i,L}$ in the basis $(e_{j}(z))_{j}$ as
\begin{equation}\label{eq:defJL}
  J_{i,L}(z,t)
  =
  \sum_j
  J_{ij,L}(z,t) e_j(z).
\end{equation}
The matrix $J_L(z,t) = (J_{ij,L}(z,t))_{i,j}$ satisfies the following linear ODE
\begin{align}
  &
    J_{L}''(z,t)
    +J_{L}(z,t) R_L(t)=0
    ,
\end{align}
where $R_L$ is a matrix representing the Riemann curvature tensor
$R(L,e_i)L =\sum_{j} R_{ij,L}(t) e_j $.
To ease the notation, we omit the dependence on $z$ and $L$ whenever  no confusion arises.

The initial datum for the matrix $J$ is $J(0)=I$ (here $I$ is the
identity matrix).
We omit to compute the initial datum for the derivative $J'$, for it
is not in our interest. 
Nonetheless, if 
\begin{equation}\label{eq:def-eta}
\eta(z) =(\eta_{ij})_{ij} = (g(e_i,e_j))_{ij},
\end{equation}
i.e., $\eta$ is the matrix representing the Lorentzian product in this basis,  we can prove the following.
\begin{proposition}
  The matrix $J'(0)$ satisfies
\begin{equation*}
  (J'(0)\eta)^T
  = J'(0)\eta.
\end{equation*}
\end{proposition}
\begin{proof}
First we notice that $L$ is locally a gradient.
Indeed, let $u$ be the function defined by the equation
\begin{equation*}
  u(\lparam(x,t,s))=-s,
\end{equation*}
where $\lparam$ is the local parameterization~\eqref{eq:paramM}.
Its differential is given by
\begin{equation*}
  \de u(e_i)=0,\,
  i=1\dots n-2
  ,\,
  \quad
  \de u(e_{n-1})=
  \de u(L)
  =0
  ,
  \quad
  \de u(e_{n})
  =
  \de u(\overline{L^S})
  =-1,
  \end{equation*}
yielding that $L$ is the only vector such that $g(L,e_i)=\de u(e_i)$,
$i=1\dots n$.
This implies the  symmetry 
$$g(\nabla_{e_i} L,
e_j)=g(\nabla_{e_j} L, e_i).$$
Therefore
\begin{equation*}
  \sum_k J_{ik}'(0) \eta_{kj}
  =
  g(J_i'(0),e_j)
  =
  g(\nabla_{e_i} L,
  e_j)
  =
  g(\nabla_{e_j} L,
  e_i)
  =
  g(J_j'(0),e_i)
  =
  \sum_k J_{jk}'(0) \eta_{ki}.
  \qedhere
\end{equation*}
\end{proof}

\begin{lemma}\label{lem:Jdzeta}
Let $g$ be a Lorentzian metric and $S$ be a local cross-section for the null hypersurface $H$.

If $f:U\to S$ is a $C^2$-diffeomorphism inducing
the parameterization map $\lparam$ for $H$ as in~\eqref{eq:paramM}, such that $f(x)=z$ and $\lparam(x,0,0) = z$, then the vector field $Y_i(t)= \de\lparam_{(x,t,0)}
  \left(\frac{\partial}{\partial x_i}\right)$ satisfies the Jacobi field equation~\eqref{eq:jacobi} with initial conditions~\eqref{eq:first-order-jacobi}. In particular, 
\begin{equation}\label{E:Jacobi}
  J_{\de\lparam_{(x,0,0)}(\frac{\partial}{\partial x_i})}
  (t)
  =
  \de\lparam_{(x,t,0)}
  \left(\frac{\partial}{\partial x_i}\right),
  \qquad
  i=1\dots n.
\end{equation}
\end{lemma}

\begin{proof}
Since by \Cref{P:local-null-geodesics-smart} (see also \Cref{R:nullgeodesic}), the null geodesic vector field $L$ is of class $C^1$, then the parameterization map $\lparam$ for $H$ given in~\eqref{eq:paramM} is of class $C^1$ as well.
Denote $Y(t):=\frac{\partial}{\partial x_i} \lparam_{(x,t,0)}$ and notice this is a continuous vector field along $H$. Using that distributional derivatives commute and the very definition of Riemann curvature tensor, we infer that
\begin{align*}
Y''(t)&= \nabla_{\frac{\partial}{\partial t}} \nabla_{\frac{\partial}{\partial t}}\frac{\partial}{\partial x_i} \lparam_{(x,t,0)} \quad \text{in distributional sense} \\
& =\nabla_{\frac{\partial}{\partial x_i}} \nabla_{\frac{\partial}{\partial t}}\frac{\partial}{\partial t} \lparam_{(x,t,0)}- R\left(\frac{\partial}{\partial t}, \frac{\partial}{\partial x_i}  \right) \frac{\partial}{\partial t}  \quad \text{in distributional sense}.
\end{align*}
Since, by construction, the curve $t\mapsto \lparam_{(x,t,0)}$ is a geodesic, then the first term in the right hand side vanishes. Therefore, $Y$ is a distributional solution of the Jacobi fields ODE~\eqref{eq:jacobi} with initial conditions~\eqref{eq:initial-jacobi}. A standard bootstrap argument yields that $t\mapsto Y(t)$ is a $C^2$ classical solution of such a Cauchy problem, i.e. it is a Jacobi field. By the Cauchy-Lipschitz theorem, uniqueness for such an initial value problem holds, thus ~\eqref{E:Jacobi} holds.
\end{proof}

Motivated by \Cref{lem:Jdzeta}, we  choose as basis for $T_{z}M$ the vectors $e_{i}(z)
  :=
  \de\lparam_{(x,0,0)}
  \left( \partial /\partial x_i \right)
$, 
with $i = 1, \dots, n$;
in particular $e_{n-1}(z)=L$ and $e_{n}(z)=\overline{L^S}$.

The orthogonality properties of $L$ and $\overline{L^S}$ imply that,
under this choice of coordinates,
the matrix $\eta$ as in~\eqref{eq:def-eta} takes the form
\begin{equation}\label{E:eta-non-orthogonal}
  \eta
  :=
  \begin{bmatrix}
    \eta_{11} & \dots &\eta_{1,n-2} & 0 &0\\
    \vdots & \ddots &\vdots &\vdots&\vdots\\
    \eta_{n-2,1}&\dots & \eta_{n-2,n-2} &0 &0 \\
    0&\dots & 0 &0 &-1 \\
    0&\dots & 0 &-1 &0 \\
  \end{bmatrix}.
\end{equation}

Moreover with this choice of basis, we deduce some additional property for the matrix $J$. 

\begin{lemma}\label{L:minorJ}
  The matrix $J(t)$ satisfies the following identities
\begin{equation} \label{eq:J-null-transpose}
 J(t)^T \cbf_{n-1}
  =
  \cbf_{n-1}, 
\qquad 
  J(t) \cbf_n
  = \cbf_n,
\end{equation}
with the convention that $(\cbf_i)_i$ is the canonical basis
of $\R^n$. In other words, the matrix $J(t)$ is of the form
\begin{equation}
  \label{eq:J-appearence}
  J(t)
  :=
  \begin{bmatrix}
    J_{11} & \dots & J_{1,n-2} & J_{1,n-1} &0\\
    \vdots & \ddots &\vdots &\vdots&\vdots\\
    J_{n-2,1}&\dots & J_{n-2,n-2} &J_{n-2,n-1} &0 \\
    0&\dots & 0 &1 &0 \\
    J_{n,1}&\dots & J_{n,n-2} &J_{n,n-1} &1 \\
  \end{bmatrix}
  .
\end{equation}
\end{lemma}

\begin{proof}
Since $J_L(t)=L$, then $J_{n-1}(t)=L$ or in coordinates
$J_{n-1,j}(t)=\delta_{n-1,j}$, which
proves the first identity in~\eqref{eq:J-null-transpose}.
Then,  using \Cref{lem:gauss} and
\eqref{E:eta-non-orthogonal}, we compute
\begin{align*}
  g(J_i(t),L)
  &
    =
  g(J_i(0),L)
  =
  g(J_i(0),e_{n-1})
  =
  g(e_i,e_{n-1})
  =
    \eta_{i\,n-1}
  \\
 g(J_i(t),L) &
  =
  \sum_{j}
  g(J_{ij}(t) e_j,
  e_{n-1})
  =
    \sum_{j}
    J_{ij}(t) \eta_{j \,n-1}
    =
    -J_{in}(t)
  ,
\end{align*}
thus $J_{in}(t)=\delta_{in}$, which proves the second identity of~\eqref{eq:J-null-transpose}.
\end{proof}
As a consequence of~\eqref{eq:J-null-transpose} we have that
\begin{equation}
  \label{eq:properties-J}
  J(t)^{-1} \cbf_n= \cbf_n
  ,
  \quad
  (J(t)^{-1})^T \cbf_{n-1}= \cbf_{n-1}
  ,
  \quad
  J'(t) \cbf_n= 0
  ,
  \quad
  \text{ and }
  \quad
  J'(t)^T \cbf_{n-1}= 0.
\end{equation}


\subsection{Local expression of the rigged Volume} 
 
 We now  write in a convenient way the rigged volume measure induced by the rigged metric over $H$.
As it involves the $(n-2)$-dimensional Hausdorff measure on space-like cross-sections,
we point out that on a space-like cross-section, the restriction
of the Lorentzian metric and of the rigged metric associated to it coincide; thus also the Hausdorff measures coincide.

\begin{proposition}
  Let $H$ be a null hypersurface and $S$ be a space-like local
  cross-section.
  Consider $L$ a null-geodesic vector field over $H$ and $\Omega_{S}$ 
  the open coordinate patch of $H$ defined in~\eqref{eq:defOmegaS}. Denote by $\gflow_{L}$ the flow map of  $L$ as in~\eqref{eq;defPsiL}, by $\tilde{\vol}_{L}^{S}$ the rigged volume form as in \Cref{def:tildeVol}, and by $J=J_L$ the Jacobi field matrix as in~\eqref{eq:defJL}.
Then 
$$
\tilde{\vol}_{L}^{S}\llcorner_{\Omega_{S}} = (\gflow_{L} )_{\sharp} (\det (J(z,t))  \haus^{n-2}\llcorner_{S} \otimes \mathcal{L}^{1} ) \llcorner_{\Omega_{S}},
$$
where  $ \haus^{n-2}$ denotes the Hausdorff measure of dimension $n-2$ (over $S$) and  $\mathcal{L}^{1}$ the one-dimensional Lebesgue measure. 

In particular, for all $\phi \in C^{0}_c(\Omega_S)$
\begin{equation}
    \label{eq:representation-rigged-volume}
    \int_{\Omega_S}
    \phi(z)
    \,
    \tilde{\vol}_L^S
    (\de z)
    =
    \int_S
    \int_{\R}
    \phi(\gflow_L(z,t))
    \,
\det (J(z,t))
    \,
    \de t
    \,
    \haus^{n-2}(\de z).
\end{equation}
\end{proposition}

\begin{proof}
  The integral w.r.t.\ the rigged volume can be computed in
  coordinates with the formula
  \begin{equation}
    \label{eq:rigged-volume-coordinates}
    \int_{\Omega_S}
    \phi(z)
    \,
    \tilde{\vol}_L^S
    (\de z)
    =
    \int_{U}
    \int_{\R}
    \phi(\lparam(x,t))
    D(x,t)
    \,
    \de t\,
    \de x
    ,
  \end{equation}
  where, using~\eqref{E:Jacobi},
  \begin{equation}
    \begin{aligned}
    D(x,t)
      :=
      &
        \left(
    \det
    \left(
      \tilde{g}_L^S\left(
        \frac{\partial\lparam}{\partial x_i}(x,t)
        ,
        \frac{\partial\lparam}{\partial x_j}(x,t)
      \right)_{i,j=1,\dots,n-1}
        \right)
        \right)^{\frac{1}{2}}
      \\
        =
      &
    \det
    \left(
    \left(
      \tilde{g}_L^S\left(
        J_{e_{i}(z)}(t)
        ,
        J_{e_{j}(z)}(t)
      \right)_{i,j=1,\dots,n-1}
        \right)
        \right)^{\frac{1}{2}}
    .
    \end{aligned}
  \end{equation}
  By Equation~\eqref{eq:rigged-metric}
  and~\eqref{E:eta-non-orthogonal}, we deduce that the matrix
  $\tilde\eta$ representing the rigged metric in the basis
  $(e_i)_{i=1,\dots,n-1}$ is of the form
  \begin{equation}\label{E:tilde-eta-non-orthogonal}
  \tilde{\eta}
  :=
  \begin{bmatrix}
    \eta_{11} & \dots &\eta_{1,n-2} & 0 \\
    \vdots & \ddots &\vdots &\vdots\\
    \eta_{n-2,1}&\dots & \eta_{n-2,n-2} &0  \\
    0&\dots & 0 &1  \\
  \end{bmatrix}
  .
\end{equation}
In particular, from ~\eqref{E:eta-non-orthogonal}, it follows that $\det\tilde\eta=-\det\eta$. 
Denote by $\tilde J$ the upper-left $(n-1)\times(n-1)$ minor of the
matrix $J$. Taking into account~\eqref{eq:J-appearence}, if $f(x)=z$, we deduce  that
\begin{equation}\label{eq:D(x,t)detJ}
  \begin{aligned}
    D(x,t)
    &
  =
  (\det(\tilde J^T(z,t) \tilde\eta (\gflow_{L}(z,t)) \tilde J(z,t)))^{\frac{1}{2}}
  =
  \det \tilde J^T(z,t)
      (\det \tilde\eta (\gflow_{L}(z,t)) )^{\frac{1}{2}}
      \\&
  =
  \det  J^T(z,t)
    (|\det \eta (\gflow_{L}(z,t)) |)^{\frac{1}{2}}
    =
  \det  J(z,t)
    (|\det \eta (\gflow_{L}(z,t)) |)^{\frac{1}{2}}.
  \end{aligned}
\end{equation}
Plugging the identity~\eqref{eq:D(x,t)detJ}
into~\eqref{eq:rigged-volume-coordinates}, yields
\begin{equation*}
  \begin{aligned}
     \int_{\Omega_S}
    \phi(z)
    \,
    \tilde{\vol}_L^S
    (\de z)
    &
    =
    \int_{U}
    \int_{\R}
    \phi(\lparam(x,t))
    D(x,t)
    \,
    \de t
      \de x
    \\
    &
      =
      \int_U
    \int_{\R}
    \phi(\gflow_L(f(x),t))
  \det  J(f(x),t)
  (|\det \eta (f(x)) |)^{\frac{1}{2}}
    \,
    \de t \,
      \de x
    .
 \end{aligned}
  \end{equation*}
  The thesis follows immediately by observing that
  \begin{equation*}
     \int_{S}
    \psi(z)
    \,
    \haus^{n-2}(\de z)
    =
    \int_{U}
    \psi(f(x))
    (\det \tilde\eta(f(x)))^{\frac{1}{2}}
    \,
    \de t
    \,
    \de x
    ,
    \quad
    \forall
    \psi
    \in
    C^0_c(S),
  \end{equation*}
  which is the integration formula in the
  parameterization given by $f$ for the volume on the submanifold $S$.
\end{proof}

We now point out that if  $\varphi $ is a transverse function (recall \Cref{D:transverse}), then 
it is not hard to check that 
\begin{equation}\label{eq:w-transverse}
\det J_{\varphi L}(z,t) = \det J_{L}(z,\varphi(z)t).
\end{equation}
Combining then the representation  formula~\eqref{eq:representation-rigged-volume}
  with~\eqref{eq:w-transverse} it is immediate to establish the following:

\begin{corollary} \label{cor:rescale-volume}
  Let $L$ be a local null-geodesic vector field, $S$ a local space-like
  cross-section for $H$ and $\varphi$ a $C^1$ transverse function.
  Then it holds that
  $\tilde{\vol}_{\varphi L}^S=\frac{1}{\varphi}\tilde{\vol}_{L}^S$.
\end{corollary}


\subsection{Independence from local cross-sections}
\label{Ss:mesures-on-H}

By its very definition (see \Cref{def:tildeVol}), 
the rigged volume measure $\tilde \vol_{L}^{S}$ on $H$ 
depends both on $L$ and on the chosen cross-section $S$. 
We will now prove 
that it actually does not depend on $S$. 
Subsequently, we will show that 
its dependency on $L$ is controlled. 
As a consequence, $\tilde \vol_{L}^{S}$ will be globally well-defined on $H$.

We start by establishing a representation formula for the Hausdorff measure
of cross-sections that are graphs over other
cross-sections.

\begin{proposition}
  \label{P:graph-surface}
  Let $H$ be a null hypersurface and let $S$ be a local space-like
  cross-section.
  %
  Let $\bar S\subset \Omega_S$ be another local space-like
  cross-section.
  Assume that there exists a $C^2$ function $t_L:\Dom(t_L)\subset S \to\R$,
  such that $\gflow_L(z,t_L(z))\in \bar S$ and that $z\mapsto
  \gflow_L(z,t_L(z))$ is a surjection onto $\bar{S}$. Then, for every $\varphi\in C^0_c(\bar{S})$, the following identity holds :
  \begin{equation}
    \label{eq:graph-surface}
    \int_{\bar S}
    \phi(z)
    \,
    \haus^{n-2}(\de z)
    =
    \int_{\Dom(t_L)}
    \phi(\gflow_L(z,t_L(z)))
    \,
\det J (z,t_L(z))
    \,
    \haus^{n-2}(\de z)
    .
  \end{equation}
\end{proposition}
\begin{proof}
Consider the local parameterization of $H$ induced by  $f:U\subset \R^{n-2} \to S$ (i.e., $\lparam(x,t)=\gflow_L(f(x),t)$). 
Denote by $s(x) : =t_L\circ f$ and notice that $\bar{f}(x): =\lparam(x,s(x))$ induces a parameterization for $\bar{S}$.
Then
  \begin{equation}
    \label{eq:cross-section-coordinates}
    \int_{\bar S}
    \phi(z)
    \,
    \haus^{n-2}(\de z)
    =
    \int_{\dom \bar{f}}
    \phi(\bar{f}(x))
    \,
    D(x)
    \,
    \de x
    ,
  \end{equation}
  where $D(x)$ is the determinant of the matrix
  \begin{equation}
    \begin{aligned}
    A_{ij}
    :=\,
    &
    g\left(
      \frac{\partial \bar{f}}{\partial x_i}(x)
      ,
      \frac{\partial \bar{f}}{\partial x_j}(x)
    \right)
    \\
    =\,
    &
    g\left(
      \frac{\partial \lparam}{\partial x_i}(x,s(x))
      +
      \frac{\partial \lparam}{\partial t}(x,s(x))      \frac{\partial s}{\partial x_i}(x)
      ,
      \frac{\partial \lparam}{\partial x_j}(x,s(x))
      +
      \frac{\partial \lparam}{\partial t}(x,s(x))      \frac{\partial s}{\partial x_j}(x)
    \right)
    \\
    =\,
    &
    g\left(
      J_{\frac{\partial f}{\partial x_i}(x)}(s(x))
      +
      \frac{\partial s}{\partial x_i}(x)
      L
      ,
      J_{\frac{\partial f}{\partial x_j}(x)}(s(x))
      +
      \frac{\partial s}{\partial x_j}(x)
      L
      \right)
      \\
      =\,
      &
        g\left(
      J_{\frac{\partial f}{\partial x_i}(x)}(s(x))
      ,
      J_{\frac{\partial f}{\partial x_j}(x)}(s(x))
    \right), \quad \forall\, i, j = 1,\dots, n-2.
\end{aligned}
\end{equation}
 In the last identity, we have tacitly used the same notation for the metric $g$ restricted to $T\bar S$ and to $TH$; 
notice indeed that the  vectors $J_{{\partial f / \partial x_i}(x)}(s(x))$ are no longer tangent to $\bar S$.
Expanding the identity and setting $z = f(x)$, we obtain: 
\begin{align*}
 A_{ij} =      &~   g\left( \sum_{h \leq n-1}J_{i,h}(z,s(x)) e_{h} , \sum_{k \leq n-1} J_{j,k}(z,s(x)) e_{k}    \right) \\ 
      =      &~ \sum_{h,k \leq n-1} J_{i,h} (z,s(x)) \eta_{h,k}(\bar f (x)) J_{j,k}(z,s(x)).
\end{align*}
Hence  $A$ is the upper left $(n-2)\times (n-2)$ minor of $J (z,s(x)) \eta(\bar f (x)) J^{T} (z,s(x))$. 
As $L$ is a null vector for $g$, one can directly check that  $A = \bar J(s(x)) \bar \eta(\bar f (x) )\bar J^{T} (z,s(x))$, 
where $\bar J$ and $\bar \eta$ denotes the the upper left $(n-2)\times (n-2)$ minors of $J$ and $\eta$, respectively.
Therefore $D(x)= \det J (f(x),s(x)) (|\det \eta(\bar f(x)) |)^{\frac{1}{2}}$.
Plugging this identity in~\eqref{eq:cross-section-coordinates}, we
conclude that
  \begin{align*}
    \int_{\bar S}
    \phi(z)
    \,
    \haus^{n-2}(\de z)
    &
    =
    \int_{\dom \bar{f}}
    \phi(\bar{f}(x))
    \,
    D(x)
    \,
    \de x
    \\&
    =
    \int_{\dom \bar{f}}
    \phi(\gflow_L(f(x),t_L(f(x)))
    \,
\det J(f(x),t_L(f(x))) (|\det \eta(\bar f(x))|)^{\frac{1}{2}}
    \,
    \de x
    \\
    &
    =
    \int_{\dom t_L}
    \phi(\gflow_L(z,t_L(z))
    \,
\det J (z,t_L(z))
    \,
    \haus^{n-2}(\de z),
  \end{align*}
having used that $\det \eta(\bar f(x)) = \det \eta( f(x)) = \det \bar \eta (f(x))$, which follows by parallel transport.
\end{proof}

Next, we show that the rigged volume induced by two
different cross-sections coincides on the intersection of their
domain of definition (recall \Cref{SS:local-parameterization}).
To prove this claim, we will tacitly use the standard fact that the determinant of the matrix representing the Jacobi field $\det (J(z,t))$ used to represent the rigged volume measure does not depend on the choice of the basis $(e_i)_i$.
We will also use the following semi-group property of the determinant
of the matrix given by the linearity of the Jacobi equation
\begin{equation}
  \det (J(\gflow_L(z,s),t))
  \det (J(z,s))
  =
  \det (J(z,t +s)), 
\end{equation}
for all  $s\in \Dom(J(z, \dotargument))$ and $t\in \Dom (J(\gflow_L(z,s),\dotargument)$.

\begin{proposition}
[Independence from $S$]
\label{P:volumeunique}
Let $H$ be a null hypersurface, let $L$ be a null
  geodesic vector field and let $S_1,S_2\subset H$ be two
  space-like cross-sections.
  %
  Then it holds that
$$\tilde{\vol}_L^{S_1}\llcorner_{\Omega_{S_1}\cap\Omega_{S_2}} =
\tilde{\vol}_L^{S_2}\llcorner_{\Omega_{S_1}\cap\Omega_{S_2}}
$$
\end{proposition}

\begin{proof}
  Assume first that $S_2\subset \Omega_{S_1}$, and that $S_2$ can be
  parameterized by a certain function $t_L$, i.e., $\gflow_L(z,t_L(z))\in S_2$, as
  in the hypothesis of \Cref{P:graph-surface}.
  Fix $\phi\in C^{0}_c(\Omega_{S_1}\cap\Omega_{S_2})$.
  A direct computation gives
  \begin{align*}
   \int_{\Omega_{S_1}\cap\Omega_{S_2}} &
    \phi
    \,
    \de\tilde\vol_L^{S_2} \\
      \stackrel{\text{\eqref{eq:representation-rigged-volume}}}{=}
    &
      \int_{S_2}
      \int_{\R}
      \phi(\gflow_L(z,t))
      \,
      \det (J(z,t))
      \,
      \de t
      \,
      \haus^{n-2}(\de z)
    \\
    \stackrel{\text{\eqref{eq:graph-surface}}}{=}
    &
      \int_{S_1}
      \int_{\R}
      \phi(\gflow_L(\gflow_L(z,t_L(z)),t))
      \,
      \det (J(\gflow_L(z,t_L(z)),t))
      \,
      \de t
      \,
      \det (J(z,t_{L}(z)))
      \,
      \haus^{n-2}(\de z)
    \\
    %
    =
       \,\,\,
    &
      \int_{S_1}
      \int_{\R}
      \phi(\gflow_L(z,t_L(z)+t))
      \,
\det (J(z,t+t_L(z)))
      \,
      \de t
      \,
      \haus^{n-2}(\de z)
    \\
    =
    \,\,\,
    &
      \int_{S_1}
      \int_{\R}
      \phi(\gflow_L(z,t))
      \,
\det (J(z,t))
      \,
      \de t
      \,
      \haus^{n-2}(\de z) 
      \\
      \stackrel{\text{\eqref{eq:representation-rigged-volume}}}{=}
      &
    \int_{\Omega_{S_1}\cap\Omega_{S_2}}
    \phi
    \,
      \de\tilde\vol_L^{S_1}
      .
  \end{align*}
  We now consider the general case. \\
    By the following \Cref{L:gluing}, it is sufficient to show that, for every
  $z\in\Omega_{S_1}\cap\Omega_{S_2}$ there exists a neighborhood $U$
  of $z$ such that the two measures coincide therein.
  Fix $z\in \Omega_{S_1}\cap\Omega_{S_2}$ and let $S_3$ be a
  space-like local cross-section containing $z$.
  Up to taking a smaller cross-section, we can assume that $S_3$ is parameterized
  by $t_L^1:\Dom(t_L^1)\subset S_1\to\R$, w.r.t.\ the cross-section
  $S_1$.
  Analogously, we can also assume that there exists
  $t_L^2:\Dom(t_L^1)\subset S_2\to\R$ a parameterization of $S_3$
  w.r.t.\ $S_2$.
  The cross-section $S_3$ produces the rigged volume
  $\tilde{\vol}_{L}^{S_3}\in\M^{+}(\Omega_{S_3})$.
  We choose as a neighborhood of $z$ the open set
  $U=\Omega_{S_1}\cap\Omega_{S_2}\cap\Omega_{S_3}$.
  The fact that the two measures coincide on $U$ follows from the
  previous part.
\end{proof}
The following well-known gluing lemma can be proved by a standard argument via continuous partition of unity (see for instance~\cite[Th.~2.13]{Rudin}).
\begin{lemma}\label{L:gluing}
  Let $X$ be a locally-compact Hausdorff space and let $(U_i)_{i\in
    I}$ be an open cover.
  Let $\mu_i\in \M^{+}(U_i)$ be a family of locally finite positive
  measures.
  Assume that 
  $$\mu_i\llcorner_{U_i\cap U_j}=\mu_j\llcorner_{U_i\cap
    U_j}, \quad \text{for all $i,j\in I$.} $$ 
  Then there exists a positive, locally-finite measure
  $\mu\in\M^{+}(X)$, such that 
  $$\mu\llcorner_{U_i}=\mu_i ,\quad \text{for all $i\in I$}.$$
\end{lemma}

As highlighted by
\Cref{cor:rigged-volume} below, \Cref{P:volumeunique}  implies in particular that the rigged volume
measure depends only on the null-geodesic vector field, and not on
the cross-section $S$.

\begin{corollary}
  \label{cor:rigged-volume}
  Let $g$ be a Lorentzian metric on $M$.
  Given a null hypersurface $H$ of class $C^2$ and a global null-geodesic vector
  field $L$ of class $C^1$, there exists a unique measure $\vol_L$, such that for any space-like local cross-section $S$ of class $C^2$, it holds that
  $\tilde{\vol}_L^S=\vol_L\llcorner_{\Omega_S}$.
\end{corollary}

We point out that the measure $\vol_L$ is mutually absolutely
continuous w.r.t.\ the volume measure induced on $H$ by any auxiliary Riemannian metric.
In the sequel, we will say that a measure $\mu\in\M^+(H)$ is
absolutely continuous w.r.t.\ $H$, if it is absolutely continuous
w.r.t.\ the volume measure on $H$ induced by some (hence any) metric.
Accordingly, the family of probability measures absolutely continuous w.r.t.\ $H$
will be denoted by $\Prob_{ac}(H)$.

\section{Volume distortion of the (weighted) rigged measure}\label{Sec:VolDist}

For the reader's convenience, the presentation 
in this section is as self-contained as
possible, nevertheless, some parts of the arguments may be well-known
to experts (e.g., Jacobi fields estimates).
As we will have to deal with weights for the volume measures, it will be convenient  
to adopt the following notation. 
\begin{definition}
  Let $(M,g)$ be a Lorentzian manifold and let $H\subset M$ be a null hypersurface.
  Let $L$ be a null-geodesic vector field.
  We define the function $W_{L}: \Dom(W_{L}) \subset H\times \R\to \R$ as
\begin{equation}\label{E:W}
  W_{L}(z,t):=   \log(\det J(z,t))  .
\end{equation}
\end{definition}

The domain of definition of $W_{L}$ is the maximal set where the
definition makes sense, i.e., where $\gflow_{L}(z,t)\in H$ and where the
determinant is not null; in principle, this set could be larger than
$\Omega_S$, the set where the rigged volume is defined. 
We  next  show that
$$\dom(W_{L}) = \{ (z,t) \colon \det J(z,t) > 0 \}=\dom(\gflow_{L}).$$

\begin{lemma}\label{L:domainJ}
If we define
  $I_{z}$ to be the connected component of $(\gflow_{L}(z,\dotargument))^{-1} (H)$ containing $0$, then
$\det(J(z,t)) > 0$ in the interval $I_{z}$.
\end{lemma}

\begin{proof}
  Assume on the contrary that $\det J(z,t)=0$ for some $t$. Without
  loss of generality, we can assume $t>0$.
  Let $\bar t:=\inf\{t>0: \det J(z,t)=0\}>0$ and let $\bar
  z=\gflow_L(z,\bar t)$.
  By continuity, there exists $\epsilon>0$, such that $\det J(\bar z,t)>0$
  for all $t\in(-2\epsilon,2\epsilon)$.
  An iterated use of the semi-group property of the determinant gives
  a contradiction
  \begin{align*}
    \det J(z,\bar t)
    &=
      \det J(z,\bar t-\epsilon)
      \,
      \det J(\gflow_{L}(z,\bar t-\epsilon),\epsilon)
      =
      \det J(z,\bar t-\epsilon)
      \,
      \det J(\gflow_{L}(\bar z,-\epsilon),\epsilon)
    \\
    &
      =
      \det J(z,\bar t-\epsilon)
      \,
      (\det J(\bar z,-\epsilon))^{-1}
      .
      \qedhere
  \end{align*}
\end{proof}

Capitalizing on \Cref{S:jacobi}, 
we show that $W_{L}$ satisfies a
Riccati-type inequality.

\begin{proposition}
  \label{eq:riccati-for-jacobi}
  The function $W_{L}(z,\dotargument)$  is twice differentiable and it satisfies the following differential
  inequality
  \begin{equation}
    W_{L}''(z,t)
    +\frac{(W_{L}'(z,t))^2}{n-2}
    \leq
   - \Ric_{\gflow_L(z,t)}(L,L)
    .
  \end{equation}
\end{proposition}

\begin{proof}
First of all, since $t\mapsto J(z,t)$ is of class $C^2$, then also the map $t\mapsto W_L(z,t)$ is so.

Differentiate $W_{L}$ once and obtain
\begin{equation*}
  W_{L}'(z,t):=
  \tr(J(t)^{-1} J'(t))
  .
\end{equation*}
Define the matrix
\begin{equation*}
  U(t):=J(t)^{-1}\, J'(t)
  .
\end{equation*}
Differentiating $U$, we obtain
\begin{align}
  \label{eq:ode-for-U}
  U'(t)
  &=
    -J(t)^{-1}J'(t)
    J(t)^{-1}J'(t)
    + J(t)^{-1}
    J''(t)
    =
    - U(t)^2
    -R(t)
    .
\end{align}
We now claim that $\tr(U^2)\geq\frac{(\tr U)^2}{n-2}$.
If this is the case, one can differentiate $W_{L}$ twice, obtaining
\begin{align*}
  W_{L}''(t)
  &=
  \tr(U'(t))
  =
 -[\tr(U(t)^2)
  +
  \tr(R(t))]
     \leq
    - \frac{(\tr\, U(t))^2}
     {n-2}
     -
     \Ric_{\gflow_L(z,t)}(L,L)
  \\
   &
     =
    - \frac{W_{L}'(t)^2}{n-2}
     -
     \Ric_{\gflow_L(z,t)}(L,L)
     .
\end{align*}

At this point we prove the claim.
First notice that Equation~\eqref{eq:properties-J} yields
\begin{equation}
  U(t)\cbf_{n}=0
  \quad\text{ and }\quad
  U(t)^T\cbf_{n-1}=0
  ,
\end{equation}
and the same holds true for the matrix $U(t)^2$.
We deduce that the matrix $U$ is of the form
\begin{equation}
  U(t)
  :=
  \begin{bmatrix}
    U_{11} & \dots  & U_{1,n-2} & U_{1,n-1} &0\\
    \vdots & \ddots &\vdots &\vdots&\vdots\\
    U_{n-2,1}&\dots & U_{n-2,n-2} &U_{n-2,n-1} &0 \\
    0&\dots & 0 &0 &0 \\
    U_{n,1}&\dots & U_{n,n-2} &U_{n,n-1} &0 \\
  \end{bmatrix}
  ,
\end{equation}
and that $\tr U(t)=\tr \bar U(t)$, where $\bar U(t)$ is the
upper-left $(n-2)\times (n-2)$-minor of the matrix $U(t)$.
Fix $i,j=1\dots n-2$ and compute
\begin{align*}
  (U(t)^2)_{i,j}
  =
  \sum_{k=1}^{n-2}
  U_{ik}(t)
  U_{kj}(t)
  +
  U_{i,n-1}(t)
  U_{n-1,j}(t)
  +
  U_{in}(t)
  U_{nj}(t)
  =
  \sum_{k=1}^{n-2}
  \bar U_{ik}(t)
  \bar U_{kj}(t)
  =
  (\bar U(t)^2)_{ij}
  ,
\end{align*}
that is, the upper-left $(n-2)\times (n-2)$-minor of the matrix
$U(t)^2$ is precisely $\bar U(t)^2$.
The identity~\eqref{eq:ode-for-U} implies that $\bar U$ solves the
following ODE (here $\bar R$ is defined accordingly)
\begin{equation}\label{eq:ODEU'}
  \bar U'(t)
  =
  -\bar U(t)^2
  -\bar R(t)
  .
\end{equation}
Since the ODE~\eqref{eq:ODEU'} is solved also by $\bar U^T$ and the initial datum
$\bar U(0)$ is a symmetric matrix (recall that $U(0)=J'(0)$ and that
$(J'(0)\eta)^T=J'(0)\eta$), then $\bar U(t)$ is symmetric for all $t$, by the uniqueness of the solution for the initial value problem for~\eqref{eq:ODEU'}. 
In particular its eigenvalues are real. Therefore, one can apply Jensen
inequality finding that $\tr (\bar U(t)^2)\geq \frac{(\tr \bar
  U(t))^2}{n-2}$.
\end{proof}

\subsection{Weighted rigged measure}\label{Subsec:WeightedRigged}

As Ricci curvature is often used to 
describe the interplay between the volume measure $\vol_g$
 and the Riemannian metric, 
in the case an additional weight is considered, 
the generalized $N$-Ricci tensor (also known as Barky--\'Emery tensor)
is the natural object to consider. 

In the case of a null hypersurface $H$
endowed with a null-geodesic vector field $L$, given a weight $\Phi: H \to \R$ of class $C^2$, we consider 
the generalized $N$-Ricci tensor 
\begin{equation}
  \Ric^{g,\Phi,N}
  =
  \begin{cases}
  \Ric
  -
  \Hess_g\Phi
  -
  \frac{1}{ N-n}
  \nabla_g\Phi
  \otimes
  \nabla_g\Phi
      \qquad
    &
      \text{ if }N>n,
    \\
    \Ric
      \qquad
    &
      \text{ if $N=n$ and }
      \de \Phi=0
    \\
    -\infty
    &
      \text{ otherwise;}
  \end{cases}
\end{equation}
where $N$ plays the role of upper bound on the dimension.
We will consider the weighted rigged measure
$\mm_L:=e^{\Phi}\vol_L$  on the set  $\Omega_{S}$.
In particular, if $S$ is a local 
cross-section and the flow map $\gflow_{L}$ is a diffeomorphism on $\Omega_{S}$, then
 \begin{equation}\label{E:formulamL}
    \int_{\Omega_S}
    \phi(z)
    \,
    \mm_L
    (\de z)
    =
    \int_S
    \int_{\R}
    \phi(\gflow_L(z,t))
    \,
    e^{\Phi(\gflow_L(z,t))+W_{L}(z,t)}
    \,
    \de t
    \,
    \haus^{n-2}(\de z)
    ,
    \quad
    \forall
    \phi
    \in
    C^{0}_c(\Omega_S)
    .
  \end{equation}

We now complement \Cref{eq:riccati-for-jacobi} in the weighted case.

\begin{proposition}
  \label{P:convexity-weight}
  Let $L$ be a null-geodesic vector field and $S$ be a space-like cross
  section.
Then, for every $z\in S$, the function $a(t):=\Phi(\gflow_L(z,t))+W_{L}(z,t)$ is twice differentiable and it satisfies the
  following differential inequality
  \begin{equation}
    a''(t)
    +\frac{(a'(t))^2}{N-2}
    \leq
    -\Ric^{g,\Phi,N}_{\gflow_{L}(z,t)}(L,L)
    .
  \end{equation}
\end{proposition}

\begin{proof}
  By differentiating $a(\dotargument)$, we obtain
  \begin{align}
    a'(t)
    =
    W_{L}'(t)
    +
    g(L,\nabla \Phi)_{\gflow_L(z,t)}
    \qquad
    \text{ and }
    \qquad
    a''(t)
    =
    W_{L}''(t)
    +
    \Hess(\Phi)_{\gflow_L(z,t)}[L,L]
    .
  \end{align}
  A direct computation gives
  \begin{align*}
    a''(t)
    &
      =
    W_{L}''(t)
    +
      \Hess(\Phi)_{\gflow_L(z,t)}[L,L]
      \leq
     - \frac{(W_{L}'(t))^2}{n-2}
      -
      \Ric_{\gflow_L(z,t)}(L,L)
      +
      \Hess(\Phi)_{\gflow_L(z,t)}[L,L]
    \\
    &
      =
      -\frac{(W_{L}'(t))^2}{n-2}
      -
      \Ric^{g,\Phi,N}_{\gflow_L(z,t)}(L,L)
      -
      \frac{(g(\nabla \Phi,L)_{\gflow_L(z,t)})^2}{N-n}
    \\
    &
      \leq
      - \frac{(W_{L}'(t)
      +g(\nabla \Phi ,L)_{\gflow_L(z,t)})^2}{N-2}
      -
      \Ric^{g,\Phi,N}_{\gflow_L(z,t)}(L,L)
      =
     - \frac{(a'(t))^2}{N-2}
      -
      \Ric^{g,\Phi,N}_{\gflow_L(z,t)}(L,L)
      ,
  \end{align*}
  having taken into account Proposition~\ref{eq:riccati-for-jacobi}
  and Lemma~\ref{lem:useful-inequality}.
\end{proof}
\begin{lemma}
  \label{lem:useful-inequality}
 For all $\alpha, \beta>0$, and $x,y\in \R$,  it holds that
  \begin{equation}\label{xyab}
    \frac{x^2}{\alpha}
    +
    \frac{y^2}{\beta}
    \geq
    \frac{(x+y)^2}{\alpha+\beta}.
  \end{equation}
\end{lemma}
\begin{proof}
The claimed bound~\eqref{xyab} is equivalent to the trivial inequality $(\beta x- \alpha y)^2\geq 0$.
\end{proof}

\section{Rigged measure and Ricci curvature: the \texorpdfstring{$\NC^1(N)$}{NC1(N)} condition}\label{sec:CD1}

Building on top of the volume distortion estimates
recalled in~\Cref{S:volume} and~\ref{Sec:VolDist}, we next
investigate the relation between optimal transport, Ricci curvature
in the null directions, and the distortion of the rigged measure.
The present and next sections constitute the heart of the paper. In order to analyse the interplay between the rigged measure and Ricci curvature, it is convenient to partition a null
hypersurface into null geodesics also known as the generators of the null hypersurface.
Such a partition will be phrased in a terminology borrowed from optimal transport.

\begin{definition}\label{D:Transport Order}
Fix $H\subset M$ a null hypersurface in $M$.
Define the relation
\begin{equation}
  \label{eq:defn-gamma}
  \Gamma:=
  \{
  (x,y)\in H\times H:
  y=\exp_x(v), v\in\normal_xH
  \text{ future-directed},
  \forall s\in[0,1] \, \exp_x(sv)\in H 
  \}
  .
\end{equation}
\end{definition}

\begin{proposition}\label{prop:Gammapreorder}
Let $(M,g)$ be a Lorentzian manifold, 
$H$ a null hypersurface and $\Gamma$ as above.
Then the set $\Gamma$ is a pre-order on $H$, i.e., it is reflexive and
transitive.
\end{proposition}
\begin{proof}
  The fact that $\Gamma$ is reflexive is trivial. We next show the
  transitivity.
  Let $x_1,x_2,x_3\in H$, such that $(x_1,x_2), (x_2,x_3)\in\Gamma$.
  Let $v_i\in \normal_{x_{i}} H$, such that $x_{i+1}=\exp_{x_i}(v_1)$,
  $i=1,2$.
  Let $w\in \normal_{x_2}H$ be the parallel transport of $v_1$ along the geodesic
  $t\mapsto \exp_{x_1}(t v_1)$.
  Since $\dim \normal H =1$, then $v_2=a w$, for some $a>0$.
  Therefore $x_3=\exp_{x_1}((1+a)v_1)$ and $\exp_{x_1}(t(1+a)v_1)\in
  H$, for all $t\in[0,1]$.
\end{proof}
We say that a null hypersurface $H\subset M$ is  \emph{causal} if it does not contain periodic causal curves.

\begin{proposition}\label{prop:MapV}
Let $(M,g)$ be a Lorentzian manifold, let $H$ be a causal null
hypersurface and $\Gamma$ as above.
Then:
\begin{enumerate}
\item  The relation $\Gamma$ as in~\eqref{eq:defn-gamma} is anti-symmetric, hence a partial order  on $H$.

\item  Let $(x,y)\in\Gamma$, and $v_1,v_2\in\normal_x H$ be such that
$$
y=\exp_x(v_1)
    \qquad
    \text{ and }
    \qquad
    y=\exp_x(v_2).
$$
Then $v_1=v_2$.
In particular for any $(x,y)\in \Gamma$ the choice of $v$ in~\eqref{eq:defn-gamma}
is unique, and one can define the map
$V:\Gamma\to \normal H$ where $V(x,y)$ is the unique element in
$\normal_x H$, such that $\exp_x(V(x,y))=y$.
In other words, $V$ is the left inverse of the map $(P_H\otimes\exp):\normal
H\to H\times H$, $v\mapsto (P_H(v),\exp (v))$, where $P_H:\normal H\to H$
is the projection on the base point. Finally, the map $V$ is Borel.
  \end{enumerate}
\end{proposition}

\begin{proof}
First we prove that $\Gamma$ is anti-symmetric. 
Assume on the contrary that $(x,y),(y,x)\in \Gamma$, but 
$x\neq y$.
Then $y=\exp_x (v_1)$ and $x=\exp_y (v_2)$.
Let $w\in\normal_y H$ be the parallel transport of $v_1$ along
$t\mapsto \exp_x(tv_1)$; of course $v_2=a w$, for some $a>0$.
We deduce that
the causal geodesic $t\mapsto\exp_x (tv_1)$  is periodic of period $1+a$ and is contained in $H$, contradicting that $H$ is causal.

Concerning the second part, if $v_1\neq v_2$, then it is clear that $v_1=b v_2$, for some
$b\neq 1$. Up to swapping $v_1$ and $v_2$, we can assume that $b>1$.  The curve $t\mapsto\exp (t v_2)$ intersects
$y$ for $t=1$ and $t=b$; therefore, the restriction $(1, b)\ni t\mapsto \exp (t v_2)$ defines a causal periodic curve in $H$, contradicting  that $H$ is causal.

Finally, we prove that $V$ is Borel.
Let $G:=\Graph(P_H\otimes \exp)$ be the graph of the map $v\in\normal
H\mapsto (P_{H}(v),\exp (v))$; $G$ is closed, for it is a graph of a continuous function.
Moreover, since it is a closed subset of a manifold, it is also
$\sigma$-compact.
Observe that
$$
\Graph(V)= \left(\Gamma\times\normal H \right) \cap \{(v,w):(w,v)\in G\}.
$$
Thus $\Graph(V)$ is a $\sigma$-compact set.
Since a map with a $\sigma$-compact graph is $F_\sigma$-measurable,
then $V$ is $F_\sigma$-measurable, thus Borel measurable.
\end{proof}

\begin{definition}[Transport Relation]
Define $\relation=\Gamma\cup\Gamma^{-1}$ (here $\Gamma^{-1}$ is given
by swapping the two entries of all the couples in $\Gamma$).
The relation $\relation$ can also be
characterized as
\begin{equation}
  \label{eq:defn-relation}
  \relation=
  \{
  (x,y)\in H\times H:
  \exists v\in\normal_xH
  :
  y=\exp_x(v),
  \text{ and }
  \forall s\in[0,1] :\exp_x(sv)\in H 
  \}.
\end{equation}
\end{definition}

\begin{proposition}
Let $(M,g)$ be a Lorentzian manifold and $H$ a causal null 
hypersurface.
Then the set $\relation$ is an equivalence relation on $H$.
\end{proposition}
\begin{proof}
Thanks to \Cref{prop:Gammapreorder},  the only non-trivial part is the transitivity.
  Let $x_1,x_2,x_3\in H$ be such that $$(x_1,x_2),\; (x_2,x_3)\in\relation.$$
  Let $v_i\in \normal_{x_{i}} H$ be such that $x_{i+1}=\exp_{x_i}(v_1)$,
  $i=1,2$.
  Let $w\in \normal_{x_2}H$ be the parallel transport of $v_1$ along
  $t\mapsto \exp_{x_1}(t v_1)$.
  Since $\dim \normal H =1$, then $v_2=a w$, for some $a\neq 0$.
  Therefore $x_3=\exp_{x_1}((1+a)v_1)$ and $\exp_{x_1}(t(1+a)v_1)\in
  H$, for $t\in[0,1]$.
\end{proof}
Accordingly, if $\alpha\in H$, we will denote by $H_\alpha$ the equivalence class
of $\relation$ containing $\alpha$.
Each equivalence class is the image of a maximal (in $H$) 
light-like geodesic. 
The sets $H_\alpha$ are usually called the 
\emph{generators} of $H$. The set parameterizing the family of equivalence classes will be denoted by $Q$, i.e. $Q=H \big/ \relation$, and the partition will be written as $\{H_\alpha\}_{\alpha\in Q}$. 

\begin{remark}
 Notice that the differentiability of  $H$ is key to have that $\relation$ is an equivalence
 relation on all $H$.
 For instance, in the light cone in Minkowski spacetime, $\relation$ would fail transitivity, as two different causal geodesics can intersect at the tip of the cone, in case it is the initial point of both.  The setting of lower regularity will be investigated in a forth-coming work~\cite{CMM24b}. \end{remark}

\begin{proposition}
\label{P:order-gamma}
Let $(M,g)$ be a Lorentzian manifold 
and $H\subset M$ be a causal null hypersurface.
%
  %
Then, for all $(x,y)\in\relation$ it holds that
\begin{equation}
    (x,y)\in\Gamma
    \quad
    \iff
      \quad
      x\leq y      .
  \end{equation}
\end{proposition}

\begin{proof}
  The ``$\Longrightarrow$ implication'' is trivial: given $(x,y)\in\Gamma$, let $v\in\normal_xH$ be as in ~\eqref{eq:defn-gamma}; then the causal curve $s\mapsto \exp_x(s v)$ connects $x$ to $y$.
  
  Regarding the ``$\Longleftarrow$ implication'', assume that $x\leq y$, but
  $(x,y)\not\in\Gamma$.
  Then $x\neq y$ and, since $(x,y)\in \relation$, $(y,x)\in\Gamma$. 
  Therefore, by the ``$\Longrightarrow$ implication'', $y\leq x$ hence, by the causality of $H$, $x=y$,  which is a contradiction.
\end{proof}

We next introduce the following notation: if 
 $H$ is a causal null hypersurface and $S$  a local space-like
cross-section, then 
$$
H_S : = \bigcup_{\alpha \in S} H_{\alpha},
$$ 
i.e. $H_S$  is the set of all points $y$ such that $(x,y)\in\relation$,
for some $x\in S$.
We next show that $H_S$  is an open subset of $H$.

  \begin{remark}
    The results in this section presented  so far can be proven
    under the milder assumptions $g\in C^{1,1}_{loc}$ and $H$ of class
    $C^1$.
  \end{remark}

\begin{proposition}\label{prop:HSopen}
Let $(M,g)$ be a Lorentzian manifold, let $H$ be a causal null hypersurface and let $S$ be a local space-like cross-section for $H$. 
Then  $H_S$ is an open subset of  $H$.
\end{proposition}

\begin{proof}
  It is sufficient to prove that for all $z\in S$, there exists a local cross-section $S'\subset S$ still containing $z$, such that $H_{S'}$ is open in $H$.
  Fix $z\in S$ and find $S'$, a coordinated open set (inside $S$),
  whose coordinates are given by $f:U'\subset \R^{n-2}\to S'$.
  Take $L\in\vfield(\normal H|_{S'})$ of class $C^1$.
  We can construct the standard coordinates system by taking
  $\lparam(x,t):=\gflow_L(f(x),t)=\exp_{f(x)}(tL)$, as described in
  Section~\ref{SS:local-parameterization}.
 Let $U\subset \R^{n-1}$ be the maximal definition domain of $\lparam$; using that $H$ is causal, one can prove that $U\subset \R^{n-1}$ is open and  $\lparam(U)=H_{S'}$.

  If $\lparam$ is a local diffeomorphism, then its image
  is open and we conclude.
  The fact that $\lparam$ is a local diffeomorphism is a consequence of
  \Cref{L:domainJ}.
\end{proof}

 Motivated by \Cref{prop:HSopen}, we give the following:
\begin{definition}\label{def:denseSk}
Let $(M,g)$ be a Lorentzian manifold, let $H$ be a causal null hypersurface and let $(S_k)_k$ be a sequence of local space-like acausal cross sections for $H$. We say that $S_k$ is \emph{a dense sequence of local space-like acasual cross sections for $H$} if $H=\bigcup_k H_{S_k}$.
\end{definition}

Notice that if $H$ admits a global  space-like and acausal cross section $S$, then $S$ is dense in the sense of definition \Cref{def:denseSk}.

\begin{remark}[Global definition of $\mm_{L}$ and $L$] \label{R:globalcross}
If $S$ is  acausal inside $H$, i.e. any causal curve
in $H$ intersects $S$ at most once,  as pointed out in
Remark~\ref{rmrk:global-null-geodesics},
one can produce a global null-geodesic vector field $L$ on $H_S$, by defining it on $S$ and then extending it by parallel transport.
The null hypersurface $H_S$ can be globally
parameterized using the map $\gflow_L$ taking $\alpha\in S$ and $t\in\R$ as
parameters, implying that $S$ is a global cross-section for $H_S$. 
Below, we use results from the previous sections.
 \Cref{L:domainJ} guarantees moreover that $W_{L}$ as in~\eqref{E:W} is
  defined on the whole domain of definition of $\gflow_L$.
  Hence we can chose $H_S$ to play the role of $\Omega_S$.
  We recall that the open set $\Omega_S$, introduced in
  Section~\ref{SS:local-parameterization}, is a neighborhood of $S$ in $H$, where
  the parameterization given by $\gflow_L$ is a diffeomorphism.
  
  Therefore, the representation formula for the rigged volume
  holds true for the whole set $H_S$:
  \begin{equation}\label{E:representingV}
    \int_{H_S}
    \phi(z)
    \,
    \vol_L
    (\de z)
    =
    \int_S
    \int_{\R}
    \phi(\gflow_L(z,t))
    \,
    e^{W_{L}(z,t)}
    \,
    \de t
    \,
    \haus^{n-2}(\de z)
    ,
    \quad
    \forall
    \phi
    \in
    C^{0}_c(H_S)
    .
  \end{equation}
Combining this observation with \Cref{def:denseSk} we deduce that: if $H$ is causal and $(S_k)_k$ is a dense sequence of local space-like and acausal cross-sections, we obtained the following objects:  
\begin{equation}\label{E:summary}
H = \bigcup_{k} H_{S_{k}},  \quad H_{S_{k}}\subset H \textrm{ open }, \quad L_{k} \in \vfield^{1}(\normal
  H_{S_{k}}) \textrm{ null-geodesic vector field}. 
\end{equation}
such that $\nabla_{L_{k}}L_{k} = 0$ and $L_{k} \neq 0 $ over  $H_{S_{k}}$.  
Hence the measures $\vol_{L_{k}}, \mm_{L_{k}} \in \mathcal{M}_{+}(H_{S_{k}})$ (not depending on the choice of $S_{k}$ by \Cref{cor:rigged-volume}) are well defined, 
they can be represented by~\eqref{E:formulamL} and verify the conclusion of \Cref{P:convexity-weight}. 

By considering the (saturated) set
$B_k:=H_{S_k}\backslash \bigcup_{h=1}^{k-1} H_{S_h}$ 
one can then define
\begin{equation}\label{E:globalL}
 L:= \sum_{k}  \indicator_{B_k} L_k,
\qquad \vol_{L} : = \sum_k \vol_{L_k}\llcorner_{B_{k}}.
\end{equation}
The previous set of assumptions on $H$ yields that $L$ is a global null-geodesic vector field of Borel  regularity 
with a well-defined associated flow map $\gflow_L$;
finally $\vol_{L} \in \mathcal{M}_{+}(H)$.
(retaining 
~\eqref{E:formulamL} and the conclusion of \Cref{P:convexity-weight}).

Accordingly, one defines  the weighted measure 
$\mm_{L}$. 
We summarize this properties in the following 
\end{remark}

\begin{theorem}\label{T:summary}
  Let $(M,g)$ be a Lorentzian manifold, let $H$ be a causal null-hypersurface and $(S_k)_k$
be a dense sequence of local space-like and acausal cross-sections   
for $H$.
Let also $L$ be a null-geodesic vector field.

Then the following representation formula
$$
\vol_L = \sum_k \int_{B_k \cap S_k} (\Psi_L(z,\dotargument))_\sharp 
\left( e^{W_L(z,t)} dt \right) \, \mathcal{H}^{n-2}(\de z),
$$
is well-posed and defines a non-negative Radon measure over $H$ (see \Cref{SS:local-parameterization} for the definition of $\gflow_L$ and~\eqref{E:W} for the definition of $W_L$). 
Moreover if $\Phi: H \to \R$ is a $C^2$-function, 
defining for each $z\in S_k$
 the map $a_z(t):=\Phi(\gflow_L(z,t))+W_{L}(z,t)$,  then
$$ 
    a_z''(t)
    +\frac{(a_z'(t))^2}{N-2}
    \leq
    -\Ric^{g,\Phi,N}_{\gflow_{L}(z,t)}(L,L)
    .
$$
\end{theorem}

\Cref{T:summary} suggests 
the following two definitions, one for
null hypersurfaces and one for manifolds.

\begin{definition}[$\NC^{1}(N)$ for a null hypersurface]\label{def:NC1Quadruples}
  Let $(M,g)$ be a Lorentzian manifold, 
   let $H\subset M$ be a causal null hypersurface, 
    $(S_k)_k$ be a dense sequence of local space-like and acausal cross-sections for $H$
  and $\Phi:H \to\R$ be a $C^0$ function. Assume $g, H, S_k$ to be of class $C^2$.
  
We say that the quadruple $(M,g,H, \Phi)$ satisfies the null energy condition
  $\NC^1(N)$ if and only if for all $z\in S_{k}$ the function
$a_z(t):=\Phi(\gflow_L(z,t)) + W_{L}(z,t)$ is locally-Lipschitz and it  satisfies
\begin{equation}
  a_z''
  +
  \frac{(a_z')^2}{N-2}
  \leq 0
  ,
  \qquad
  \text{ in the sense of distributions}.
\end{equation}
\end{definition}

\begin{definition}[$\NC^{1}(N)$ for a space-time]\label{def:NC1Quadruplesspace}
  Let $(M,g)$ be a Lorentzian manifold with $g\in C^2$ and let $\Phi:M\to\R$
  be a $C^0$ function.

  We say that the triple $(M,g,\Phi)$ satisfies the
   null energy condition $\NC^1(N)$, if and only if, for
  any causal null hypersurface $H\subset M$ of class $C^2$ admitting a dense sequence of
  local space-like and acausal cross-sections of class $C^2$, 
  the quadruple $(M,g,H,\Phi)$ satisfies the null
  energy condition $\NC^1(N)$.
\end{definition}

\begin{remark}[Independence from $L$ in \Cref{def:NC1Quadruples}] \label{Rem:indepNC1L}
\Cref{def:NC1Quadruples}  is well-posed.  
To verify this statement it is first necessary to 
show its independence from the choice of $L$. 
So given $M$, $H\subset M$ and
$(S_k)_k$ be a dense sequence of local space-like 
and acausal cross-sections for $H$.

It is clear that is sufficient to 
argue for each single $S_k$ and $H_{S_k}$. 
By \Cref{R:change-null-geodesics}, any two null-geodesic 
vector fields $L_1$ and $L_2$ (firstly defined  over $S_k$ and the extended by parallel transport to the whole $H_{S_k}$) differ by a transverse function $\varphi$.  
Then the results in \Cref{SS:local-parameterization} give that
$$
\gflow_{L_2}(z,t) =
\gflow_{\phi L_1}(z,t)
    =
    \gflow_{L_1}(z,\phi(z)t).
$$   
Analogously, from~\eqref{E:W} and
\eqref{eq:w-transverse}: 
$$
W_{L_2}(z,t) = \log(\det (J_{\varphi L_1} (z,t)))
= 
\log(\det (J_{L_1} (z,\varphi(z)t))) 
=
W_{L_1}(z, \varphi(z)t). 
$$
Therefore, 
if
$
a_{L_i}(t):=\Phi(\gflow_{L_i}(z,t)) + W_{L_i}(z,t)
$
for $i = 1,2$, we conclude that 
$$
a_{L_2}'' + \frac{(a'_{L_2})^2}{N-2} = 
\varphi^2(z) \left( a_{L_1}'' + \frac{(a'_{L_1})^2}{N-2}
\right),
$$
proving the claim. 
To conclude, we recall that the independence of the construction on a particular choice of local cross-sections was already established in \Cref{P:volumeunique}.
\end{remark}

\begin{remark}
The global null-geodesic vector field $L$ of \Cref{T:summary}, can be used as a ``gauge'' for the map $V$.
This means that we can uniquely determine the function $t_L:\Gamma\to[0,\infty)$ such that
\begin{equation}\label{E:definitiontL}
V(x,y) = t_L(x,y)\,L.
\end{equation}
In other words, $\gflow_L(x,t_L(x,y))=y$.
If $L$ is Borel regular, then $t_L$ is Borel as well and,
moreover, if $\varphi$ is a  $C^1$ transverse
  function, $t_{\varphi L}=\frac{1}{\varphi} t_L$.
Finally, we extend $t_L$ to $\Gamma^{-1}$, by
setting $t_L(y,x)=-t_L(x,y)$.
\end{remark}

\section{Optimal Transport inside null hypersurfaces}\label{sec:OTinsideH}

In \Cref{sec:CD1}, we gave a synthetic characterization of the null energy condition (see \Cref{T:summary} and Definitions~\ref{def:NC1Quadruples}--\ref{def:NC1Quadruplesspace}), based on concavity properties of 1-dimensional densities along (null) rays of the transport set. In the next Section~\ref{sec:NCe}, we will give another synthetic characterization of the null energy condition in terms of displacement convexity of the entropy relative to the rigged measure along null hypersurfaces.  To this aim, in the present section, we establish some results about optimal transport inside null hypersurfaces.

For some basics on optimal transport in the Lorentzian setting, see, e.g., Eckstein--Miller~\cite{EM17}, Suhr~\cite{Suhr}, McCann~\cite{McCann}, and Cavalletti--Mondino~\cite{CaMo:20}.

Let us start with an example, which will serve as a motivation for the next definitions.

\begin{example}
  Consider the two-dimensional Minkowski space-time $M=\R^2$, whose
  coordinates $x_1$ and $x_2$ are the space and time coordinates
  respectively.
  We take as null hypersurface $H$ the disjoint union of the lines
  $\{x_2=x_1\}$ and $\{x_2=x_1+2\}$.
  Let $\mu_0=\delta_{(0,0)}$ and
  $\mu_1=\delta_{(-1,1)}$ be the starting and ending measures,  respectively.
  It is trivial to see that $\ell_{p}(\mu_0,\mu_1)=0$; in particular, the
  measure $\mu_0$ can be transported along causal curves to $\mu_1$.
  Nonetheless, $\mu_0$ cannot be transported to $\mu_1$ along curves
  inside $H$.
\end{example}

The example above suggests the following definition.

\begin{definition}
  \label{defn:inside}
Let $(M,g)$ be a Lorentzian manifold.
  Let $H$ be a null hypersurface and let $\mu_i\in\Prob(H)$, $i=0,1$ be
  two probability measures.
  We say that \emph{$\mu_0$ is null connected to $\mu_1$ along $H$}  if there
  exists a probability measure  $\nu \in \mathcal{P}(C([0,1];H ))$  such that
  \begin{equation}\label{eq:Causnu}
  \pi:=(\ee_{0},\ee_{1})_{\sharp} \nu\in \Pi_{\leq}(\mu_0, \mu_1)\,,\quad \tau(x,y)=0 \text{ for $\pi$-a.e. $(x,y)$, }\quad \text{and}\quad \nu\text{-a.e. } \gamma {\rm\ is \ causal}.
  \end{equation}
  We also denote
  \begin{equation}\label{def:OptCaus}
\OptCaus^H(\mu_0,\mu_1) := \{ \nu \in \mathcal{P}(C([0,1];H )) \text{ satisfying~\eqref{eq:Causnu}}\}.
\end{equation}
\end{definition}
By the very definition, it holds that $\pi:=(\ee_{0},\ee_{1})_{\sharp}
\nu$ is an optimal coupling from $\mu_0$ to $\mu_1$ for  any
Lorentz--Wasserstein distance $\ell_p$, for all $p\leq 1$.

From~\eqref{eq:Causnu}, it follows that $\nu$-a.e.\ $\gamma$ is a null pre-geodesic 
(i.e. it can be re-parameterized into a null geodesic).

When dealing with the dynamical approach to optimal transport in the
Riemannian setting or in the setting of purely time-like transport,
it is well-known that optimal dynamical transport plans are
concentrated on geodesics.
In the case of transport with null cost this is not the case.
For instance, as observed above, one can reparameterize the geodesics where a null optimal
dynamical transport plan is concentrated, still finding an admissible
optimal dynamical transport plan.
We therefore introduce the following definition.

\begin{definition}\label{def:nullgeodDynTP}
Let $(M,g)$ be a Lorentzian manifold.
  Let $H$ be a null hypersurface.
  Let $\mu_0,\mu_1\in \Prob(H)$ be two probability measures null connected
  along $H$.
  We say that a dynamical transport plan
  $\bar\nu\in\OptCaus^H(\mu_0,\mu_1)$ is \emph{null-geodesic} if it
  is concentrated on the set
  \begin{equation*}
    \begin{aligned}
      D=\{
      &
        \gamma\in C([0,1]; H),\text{ such that }
    \exists V\in \normal_{\gamma_0} H
        \text{ light-like and future-directed,}
      \\
      &
        \text{ such that }
    \gamma_t=\exp(tV)
    ,
    t\in[0,1]
    \}
    .
    \end{aligned}
  \end{equation*}
  The set of null-geodesic dynamical transport plans from $\mu_0$ to $\mu_1$ is denoted as $\OptGeo^H(\mu_0,\mu_1)$. In other terms, 
   \begin{equation}
\OptGeo^H(\mu_0,\mu_1) := \{ \nu \in \OptCaus^H(\mu_0,\mu_1) \colon \text{ $\nu$-a.e.\ $\gamma$ is a null geodesic}\}.
\end{equation}
\end{definition}

The next two lemmas recall some elementary properties of optimal transport on the real line, that will be useful in the rest of the work. We include the proofs for the reader's convenience.

\begin{lemma}
  \label{lem:monotone}
  Let ${\mathcal B}=\{(x,y)\in\R^2: y\geq x\}$.
  Let $\mu_0,\mu_1\in \Prob(\R)$ be two probability measures.
  Assume that there exists a coupling $\pi\in\Pi(\mu_0,\mu_1)$, such
  that $\pi(\mathcal{B})=1$.
  Let $\nu$ be the monotone rearrangement plan of the two measures.
  Then $\nu(\mathcal{B})=1$.

  Moreover, if $\nu$ is induced by a map $\mathrm{T}$, then
  $\mathrm{T}(t)\geq t$, for all $t\in\supp\mu_0$.
\end{lemma}
\begin{proof}
Consider first the case when $\mu_0$ and $\mu_1$ are of the form
\begin{align}
    \mu_0
    =
    \frac{1}{n}
    \sum_{i=1}^{n}
    \delta_{x_i}
    ,
    \qquad
        \mu_1
    =
    \frac{1}{n}
    \sum_{i=1}^{n}
    \delta_{y_i}
    ,
\end{align} 
with $x_i\leq x_{i+1}$ and $y_i\leq y_{i+1}$.
In this case 
\begin{align}
    \pi=
    \sum_{i,j}
    A_{ij}
    \delta_{(x_i,y_j)}
    ,
\end{align}
where $(A_{ij})_{ij}$ is a bistochastic matrix.
We have then that $A=\sum_{\sigma\in \mathcal{S}_n} b_\sigma \sigma$,
where $b_\sigma\geq 0$, $\sum_{\sigma\in \mathcal{S}_n} b_\sigma=1$;
the sums are taken over the set of rank-$n$ permutations $\mathcal{S}_n$
and a permutation is identified by its matrix representation.
Let $\tau\in \mathcal{S}_n$ be such that $b_\tau>0$.
We claim that $y_{\tau(k)}\geq x_{k}$, for all $k = 1, \dots, n$; 
if on the contrary $y_{\tau(k)}< x_{k}$ for some $k$, then
\begin{align*}
    \pi(\R^2\backslash \mathcal{B})
    &
    =
    \sum_{i,j} A_{ij}\delta_{(x_i,y_j)}(\R^2\backslash \mathcal{B})
    =
    \sum_{\sigma\in \mathcal{S}_n} b_\sigma
    \sum_{i} \delta_{(x_i,y_{\sigma(i)})} (\R^2\backslash \mathcal{B})
    \geq b_\tau \delta_{(x_k,y_{\tau(k)})}  (\R^2\backslash \mathcal{B})
    = b_\tau>0,
\end{align*}
a contradiction.
We can therefore apply  \Cref{L:permutation} below to 
$\tau$, deducing that 
$y_k\geq x_k$, for all $k$.
Let $\nu=\sum_{i=1}^{n} \frac{1}{n}\delta_{(x_i,y_i)}$; 
since $\nu$ is clearly the monotone rearrangement measure, we conclude.

We consider the general case by approximating $\pi$ with mixtures of
Dirac's deltas.
Fix $\pi$ as in the statement, and let $\pi_n\in\Prob(\mathcal{B})$ be a sequence such that $\pi_n\weak \pi$
and $\pi_n=\frac{1}{k_n}\sum_{i=1}^{k_n}\delta_{(x_{n,i},y_{n,i})}$.
To show that such a sequence exists, we proceed as follows.
Fix $\epsilon$ and let $K\subset \mathcal{B}$ be a compact, given by Prokhorov's
theorem, such that $\pi(K)\geq 1-\epsilon$.
Let $(E_i)_{i}$ be a finite partition of $K$, $x_i\in E_i$, such that
$E_i\subset B_{\epsilon}(x_i)$.
Define $\theta_i:=\pi(E_i)$.
Up to modifying $\theta_i$ by at most $2\epsilon$, we can assume $\theta_i$ to be
rational and summing to $1$.
Define $\tilde \pi:=\sum_i\theta_i\delta_{x_i}$.
It is then easy to see that the L\'evy--Prokhorov distance between
$\pi$ and $\tilde\pi$ is at most $5\epsilon$.
Since $\tilde\pi$ is of the desired form and the L\'evy--Prokhorov
distance induces the weak convergence, then we can approximate $\pi$
in the desired way.

Let $\mu_{n,i}:=(P_i)_{\sharp}\pi_n$, $i=1,2$ be the marginal probability 
measures for $\pi_n$.
We are in position to apply the previous part deducing that $\nu_n$, 
the monotone rearrangement plan between $\mu_{n,0}$ and $\mu_{n,1}$,
is concentrated on $\mathcal{B}$.
We now take the limit.
Clearly, $\mu_{n,i}\weak \mu_i$, $i=0,1$, because weak convergence is
stable under the push-forward operation.
Theorem~5.20 of~\cite{villani:oldandnew} guarantees that, up to a
not-relabeled subsequence, $\nu_n\weak\nu$, for some $c$-cyclical
monotone transport plan $\nu$ between $\mu_0$ and $\mu_1$ (here for
the cost $c$, one can take the squared distance).
In the real line, $c$-cyclical monotonicity coincides with the
monotone rearrangement, therefore, $\nu$ is the monotone rearrangement
between $\mu_0$ and $\mu_1$.
Finally, $\nu(\mathcal{B})\geq\limsup_n \nu_n(\mathcal{B})=1$, by weak
convergence and because $\mathcal{B}$ is closed.
\end{proof}

\begin{lemma}
\label{L:permutation}
    Let $n\in \N$. Let  
    $$x_1\leq\ldots\leq x_i\leq x_{i+1}\leq \dots \leq x_n \quad \text{and} \quad  y_1\leq\ldots\leq y_i\leq y_{i+1}\leq\ldots\leq y_n$$
    be two non-decreasing sequences in $\R$. 
    Assume that there exists a permutation $\sigma\in \mathcal{S}_n$ such that $y_{\sigma(k)}\geq x_k$,
    for all $k=1,\dots, n$.
    Then $y_k \geq x_k$, for all $k=1,\dots, n$.
\end{lemma}
\begin{proof}
    Assume on the contrary that $y_k< x_k$, for some $k$.
    For all $l> k$, it holds that
    \begin{equation}
         y_k
        < x_k
        \leq
        x_l
        \leq y_{\sigma(l)}
        ,
    \end{equation}
    therefore $\sigma(l)>k$, for all $l>k$.
    This means that $\sigma$ splits in two permutations, one of 
    them permuting the elements $\{k+1,\dots,n\}$.
    On the other hand $y_k<x_k\leq y_{\sigma(k)}$, therefore 
    $\sigma(k)>k$, which is a contradiction.
\end{proof}

\begin{lemma}
  \label{lem:structure-transport-plans}
Let $(M,g)$ be a Lorentzian manifold.
  Let $H$ be a null hypersurface  and let $\nu\in\Prob(C([0,1];H))$ be a
  measure concentrated on the set of causal curves.  
  Let $C$ be the set
  \begin{equation}
    C=\{\gamma\in C([0,1];H): (\gamma_s,\gamma_t)\in\Gamma,\; \forall\,
    0\leq s\leq t\leq1
    \},
  \end{equation}
  where $\Gamma$ was defined in~\eqref{eq:defn-gamma}.
  Then $\nu(C)=1$.
\end{lemma}

Next we prove a useful decomposition of an optimal transport plan along
a null hypersurface into a countable family of optimal transport plans,
each of them along a null hypersurface having a global cross-section.

\begin{proposition}
  \label{P:decomposition}
Let $(M,g)$ be a Lorentzian manifold.
  Let $H$ be a $C^1$ null hypersurface, such that there exists   a
  dense sequence of local space-like and acausal $C^1$ cross-sections $(S_k)_k$, in the sense of \Cref{def:denseSk}.
  Let $\mu_0,\mu_1\in \Prob(H)$ be two probability measures that are
  null connected along $H$.

  Then, up to taking a not relabeled subsequence in $(S_k)_k$, there
  exists a (possibly finite) real sequence $(w_k)_k\subset (0,\infty)$ and two sequences
  of probability measures $(\mu_i^k)_k$, $i=0,1$, such that the
  following hold.
  For all $k\in\N$, the measure $\mu_0^k$ is null connected to $\mu_1^k$
  along $H_{S_k}$.
  For $i\in[0,1]$, the measures $\mu_i^k$ and $\mu_i^h$ are orthogonal and
  $\mu_i=\sum_k w_k \mu_i^k$.

  Moreover, if  $\nu\in\OptCaus^{H}(\mu_0,\mu_1)$ (i.e., $\nu$
  realizes \Cref{defn:inside}), then we can choose a measure
 $\nu_k\in\OptCaus^{H_{S_k}}(\mu_{0}^{k},\mu_{1}^{k})$ so that $\nu=\sum_k w_k \nu_k$.
  Furthermore, having defined $\mu_t^k:=(e_t)_{\sharp} \nu^k$,
  $t\in[0,1]$, it holds that the measures $\mu_t^k$ and $\mu_t^h$ are
  orthogonal.
\end{proposition}

\begin{proof}
  Let $B_k:=H_{S_k}\backslash \bigcup_{h=1}^{k-1} H_{S_h}$, and let
  $C_k$ the set of causal curves contained in $B_k$.
  Since $H=\bigcup_k B_k$ and causal curves in $H$ can only move inside the generators of $H$, 
  then $(C_k)_{k}$ covers the set of causal
  curves in $H$; in particular, $\nu(\bigcup_k C_k)=1$.
  Moreover, as the sets $(B_k)_k$ are pairwise disjoint,  also the sets
  $(C_k)_k$ are pairwise disjoint and therefore the sequence
  $w_k:= \nu(C_k)$ sums to $1$.

  We extract (without relabeling) a subsequence, by dropping the
  elements such that $w_k=0$.
  Let $\nu^k:=\nu\llcorner_{C_k}/w_k$, and let
  $\mu_t^k:=(e_i)_{\sharp}(\bar\nu_k)$, $t\in[0,1]$.
  The fact that $\mu_0^k$ is null connected to $\mu_1^k$ along $H_{S_k}$
  is trivial, since
  $\nu^k\in\OptGeo_{\ell_p}^{H_{S_k}}(\mu_0^k,\mu_1^k)$.
  The orthogonality of $\mu_t^k$ and $\mu_t^h$ follows from
  the fact that $\nu_k$ gives full measure to $C_k$, thus $\mu_t^k$
  gives full measure to $B_k$ and the sets $B_k$ are disjoint.
\end{proof}

\Cref{P:decomposition} guarantees that, in
presence of a dense sequence of local space-like and acausal cross-sections,
a light-like optimal transport plan on a null hypersurface can be decomposed into a family
of optimal transport plans, each on a null hypersurface possessing a  global, space-like,
acausal cross-section.
Conversely, it is trivial to see that, given a sequence of null-cost
optimal transport plans on the same null hypersurface, we can sum them
obtaining a null-cost optimal transport transport plan.
%

\subsection{Existence and uniqueness of monotone plans along $H$.}

We will now single out a class of well-behaved optimal dynamical plans. 
We start with two technical facts. 

\begin{lemma}
  \label{lem:integral-transverse}
Let $(M,g)$ be a Lorentzian manifold.
  Let $H$ be a null hypersurface and 
  let $\mu_0,\mu_1\in \Prob(H)$ be two probability measures, such that
  $\mu_0$ is null connected to $\mu_1$ along $H$.
  Let  $\nu\in\OptCaus^H(\mu_0,\mu_1)$, and let
  $\mu_t:=(e_t)_{\sharp}\nu$, for all $t\in [0,1]$.

  If $\varphi:H\to\R$ is a transverse function,
  then it holds that
  \begin{equation}
    \label{eq:integral-transverse}
    \int_H
    \varphi
    \,\de\mu_t
    =
    \int_H
    \varphi
    \,\de\mu_s
    ,
    \qquad
    \text{ for all }
    t,s\in[0,1]
    .
  \end{equation}
  Assume also that $\varphi$ is non-negative and the integrals in~\eqref{eq:integral-transverse} are equal to $1$.
  Define $\tilde \mu_t:= \varphi\mu_t$.
  Then $\tilde\mu_0$ is null connected to $\tilde\mu_1$ along $H$, and
  $\tilde\nu:=(\varphi \circ e_0)\; \nu \in  \OptCaus^H(\tilde\mu_0,\tilde\mu_1)$ 
  is a dynamical transport plan,
  such that
  $\tilde\mu_t=(e_t)_{\sharp}\tilde\nu$ for all $t\in [0,1]$.
\end{lemma}

\begin{proof}
  Let $C$ be the set of causal curves in $H$.
  Since $\varphi$ is transverse, then $\varphi \circ\gamma$ is constant for all
  $\gamma\in C$.
   Clearly $\nu(C)=1$,  thus 
  $\varphi \circ e_t=\varphi \circ e_s$, $\nu$-a.e..
  By definition of $\mu_t$ and $\mu_s$, we have that
  \begin{equation}
    \int_H
    \varphi
    \,\de\mu_t
    -
    \int_H
    \varphi
    \,\de\mu_s
    =
    \int_{C([0,1];H)}
    (
    \varphi\circ e_t-
    \varphi\circ e_s
    )
    \,
    \de \nu
    =0
    .
  \end{equation}
  For the second part, a direct computation gives
  \begin{equation}
    \tilde\mu_t=\varphi\mu_t
    =
    \varphi(e_t)_{\sharp}(\nu)
    =
    (e_t)_{\sharp}(\varphi\circ e_t \nu)
    =
    (e_t)_{\sharp}(\varphi\circ e_0 \nu)
    =
    (e_t)_{\sharp}(\tilde\nu).
    \qedhere
  \end{equation}
\end{proof}

Before stating the next lemma, recall that a subset $B\subset \R\times \R$ is said to be \emph{monotone}, provided  the following holds: if $(s_1,t_1)\in B$ and $(s_2,t_2)\in B$ with $s_1\leq s_2$, then $t_1\leq t_2$.

\begin{lemma}
  \label{lem:monotonicity-of-sets} Let $(M,g)$ be a Lorentzian manifold.
  Let $H$ be a $C^1$ be a null hypersurface.
  Let $A\subset \Gamma$.
  The following are equivalent.
  \begin{enumerate}
  \item
    For all $(x_1,y_1),(x_2,y_2)\in A$, if 
    $(x_1,x_2)\in\Gamma$ and $x_1\neq x_2$ then
    $(y_1,y_2)\in\Gamma$.
  \item
    for all $\alpha\in Q$, there exists a monotone set
    $B_\alpha\subset\R\times\R$, such that
    \begin{equation}
      A\cap (H_\alpha\times H_\alpha)
      \subset
      \gflow_L(\alpha,\dotargument)
      \otimes
      \gflow_L(\alpha,\dotargument)
      (B_\alpha)
    \end{equation}
  \end{enumerate}
\end{lemma}
\begin{proof}
  (1) $\implies$ (2).
  \quad
  Take as
  $B_\alpha:=( \gflow_L(\alpha,\dotargument) \otimes \gflow_L(\alpha,\dotargument))^{-1}(A)$.
  The fact that $B_\alpha$ is monotone is trivial.

  (2) $\implies$ (1).
  \quad
  Fix $(x_1,y_1),(x_2,y_2)\in A$.
  If $(x_1,x_2)\in\Gamma$, then all these four points belong to the
  same class $H_\alpha$.
  Then $x_i=\gflow_L(\alpha, s_i)$ and $y_i=\gflow_L(\alpha, t_i)$, for some
  $(s_i,t_i)\in B_\alpha$, $i=1,2$.
  The assumption $(x_1,x_2)\in\Gamma$ gives $s_1\leq s_2$ and the
  monotonicity of $B_\alpha$ gives $t_1\leq t_2$, concluding the
  proof.
\end{proof}

In order to obtain a uniqueness result, it will be useful to consider \emph{monotone} optimal coupling, defined below.

\begin{definition}\label{def:MonotoneCoupling} Let $(M,g)$ be a Lorentzian manifold.
  Let $H$ be a null hypersurface.
  Let $\mu_0,\mu_1\in \Prob(H)$ be two probability measures null connected
  along $H$.
  We say that an optimal coupling $\pi$ between $\mu_0$ and
  $\mu_1$ is \emph{monotone}, if:
  \begin{enumerate}
  \item 
  $\pi=(e_0\otimes e_1)_{\sharp}\nu$ for some
  $\nu\in\OptCaus^H(\mu_0,\mu_1)$;
  \item there exists a Borel set
  $A\subset\Gamma$ such that $\pi(A)=1$ and
  \begin{equation}
    \forall (x_1,y_1),(x_2,y_2)\in A
    ,
    \quad
    (x_1,x_2)\in\Gamma
    ,
    x_1\neq x_2
    \implies
    (y_1,y_2)\in\Gamma
    .
  \end{equation}
  \end{enumerate}
\end{definition}

The next proposition guarantees the uniqueness of monotone optimal couplings. It should be compared with the uniqueness of monotone optimal couplings for the $L^1$ optimal transport obtained by Feldman--McCann \cite{FeldMc-CVPDE} in $\mathbb{R}^n$.
The existence  will be proved in
\Cref{T:good-representation2}.

\begin{proposition}
  \label{P:uniqueness-monotone}
Let $(M,g)$ be a Lorentzian manifold.
  Let $H$ be a causal null hypersurface such that there exists a countable
  dense family of local
  space-like and acausal cross-sections of $H$. 
  Let $\mu_0,\mu_1\in \Prob(H)$ be two probability measures, such that
  $\ell_p(\mu_0,\mu_1)=0$,  for some (and thus every) $p\in (0,1]$.

  If a monotone optimal coupling between $\mu_0$ and $\mu_1$
  exists, then it is unique.
\end{proposition}
\begin{proof}
  Let $\pi^1$ and $\pi^2$ be two monotone optimal transport plans, induced by
  $\nu^1$ and $\nu^2$, respectively, i.e., $\pi^k=(e_0\otimes
  e_1)_{\sharp} \nu^k$, $k=1,2$.

  We assume first that there exists $S$, a global space-like and
  acausal cross-section of $H$.
  Let $L$ be a global null-geodesic vector field over $H$.
  Let $A^1,A^2\subset\Gamma$, be two subsets given by the monotonicity
  of $\pi^1$ and $\pi^2$ respectively.
  Let $K:=\bigcup_{\alpha\in S} H_\alpha\times H_\alpha$.
  Lemma~\ref{lem:structure-transport-plans} implies that
  $(e_0,e_1)(\gamma)\in K$ for $\nu^k$-a.e.\ $\gamma$, $k=1,2$,
  therefore $\pi^1(K)=\pi^2(K)=1$.
  Let $\QQ:H\to S$ be the map such that $\QQ(\gflow_L(\alpha,t))=\alpha$, for
  all $(\alpha,t)\in\dom \gflow_L$.
  With a slight abuse of notation, we also denote by $\QQ$ the map $\QQ:  K \to
  S$ given by $\QQ\circ P_1$.
  It is easy to see that
  \begin{equation}
    \QQ_{\sharp}\pi^1
    =
    \QQ_{\sharp}\pi^2
    =
    \QQ_{\sharp}\mu^0
    =
    \QQ_{\sharp}\mu^1
    =:
    \q
    .
  \end{equation}
  We can therefore apply the disintegration theorem, finding that
  \begin{equation}
    \omega
    =
    \int_S
    \omega_\alpha
    \,
    \q(\de\alpha)
    ,
  \end{equation}
  where $\omega$ can be either $\pi^1$, $\pi^2$, $\mu_0$, or $\mu_1$, and
  the measure $\omega_\alpha$ is concentrated on $H_\alpha\times\
  H_\alpha$ or $H_\alpha$, accordingly.
  Notice that $\pi^k_\alpha(A^k)=1$ for $\q$-a.e.\ $\alpha$.
  Let
  $\tilde\pi^k_\alpha:=((\gflow_L(\alpha,\dotargument))^{-1}\otimes
  (\gflow_L(\alpha,\dotargument))^{-1})_{\sharp}\pi^k$, and
  $\tilde\mu_{i,\alpha}:=(\gflow_L(\alpha,\dotargument))^{-1}_{\sharp}
  \mu_{i,\alpha}$.
  It is then clear that $\tilde\pi^k_\alpha$ is a transport plan
  between the measures $\tilde\mu_{0,\alpha}$ and
  $\tilde\mu_{1,\alpha}$.
  Lemma~\ref{lem:monotonicity-of-sets} implies that
  $\tilde\pi^k_\alpha$ is concentrated on some monotone set
  $B_\alpha^k$, therefore both $\tilde\pi_\alpha^1$ and
  $\tilde\pi_\alpha^2$ are monotone rearrangements of
  $\tilde\mu_{0,\alpha}$ into $\mu_{1,\alpha}$, thus these two
  transport plans coincide.
  By integrating, we conclude that $\pi^1=\pi^2$.

  We now drop the assumption that a global acausal cross-section
  exists.
  In this case we use \Cref{P:decomposition}, obtaining
  the sequences $(w_j)_j$, $\mu_i^j$, $\nu^j_1$, and $\nu^j_2$.
  In particular $\nu=\sum_{j} w_j \nu_k^j$, $k=1,2$.
  If we define $\pi^k_j:=(e_0\otimes e_1)_{\sharp} \nu_j$, $k=1,2$, then
  both $\pi^1_j$ and $\pi^2_j$ are monotone optimal transport plans. Applying the previous part, we get that
  $\pi^1_j=\pi^2_j$, which implies $\pi^1=\pi^2$.
\end{proof}

We are now in position to prove the existence of a null-geodesic
dynamical transport plan and a corresponding monotone optimal coupling  between two probability measures null connected along $H$.

\begin{theorem}
  \label{T:good-representation2}
Let $(M,g)$ be a Lorentzian manifold.
  Let $H$ be a causal, null hypersurface such that there exists a countable
  dense family of local space-like and acausal cross-sections. 
  Let $\mu_0,\mu_1\in \Prob_{ac}(H)$ be two probability measures,
  absolutely continuous w.r.t.\ $H$, such that
  $\mu_0$ is null connected to $\mu_1$ along $H$.  Then 
  \begin{itemize}
 \item   There exists a unique null-geodesic dynamical transport plan
  $\nu\in\OptGeo^H(\mu_0,\mu_1)$, such that the optimal coupling $\pi:=(e_0\otimes e_1)_{\sharp}\nu$ is monotone.  
  \item Such an optimal coupling $\pi$ is induced by a map.
  \item There exists a unique future-directed vector field
  $V$, such that $\dot\gamma_0 = V(\gamma_0)$,
  for $\nu$-a.e.\ $\gamma$.  If $L$ is a global null-geodesic vector field as defined in~\eqref{E:globalL}, then
  $V(\gamma_0)=t_L(\gamma_0,\gamma_1)L$ for $\nu$-a.e.\ $\gamma$, where $t_{L}$ is defined by~\eqref{E:definitiontL};
  \item The measure $(e_t)_{\sharp}\nu$ is absolutely continuous
  w.r.t.\ $H$, for all $t\in [0,1]$.
  \end{itemize}
\end{theorem}

\begin{proof}
We consider the case when there exists a global acausal cross-section $S$.
Since it follows the line
  presented in the previous proofs (see \Cref{T:summary}),  we omit the proof of the general case.

  Let  $L$ be a global null-geodesic vector field.
  Let $\rho_i$ be the density of $\mu_i$ w.r.t.\ the measure
  $\vol_L$, $i=0,1$.
  First we claim that
  \begin{equation}\label{eq:intrho0=rho1}
    \int_{\R}
    \rho_0(\gflow_L(z,t))
    \,
    e^{W_{L}(z,t)}
    \,
    \de t
    =
    \int_{\R}
    \rho_1(\gflow_L(z,t))
    \,
    e^{W_{L}(z,t)}
    \,
    \de t
    ,
    \qquad
    \text{ for a.e.\ }
    z\in S
    .
  \end{equation}
  Indeed, fix $\varphi$ a transverse function and compute
  \begin{align*}
    &
    \int_S
    \varphi(z)
    \int_{\R}
    \rho_0(\gflow_L(z,t))
    \,
    e^{W_{L}(z,t)}
    \,
    \de t
    \,
    \haus^{d-2}(\de z)
    \\
    &\qquad
    =
    \int_S
    \int_{\R}
    \varphi(\gflow_L(z,t))
    \rho_0(\gflow_L(z,t))
    \,
    e^{W_{L}(z,t)}
    \,
    \de t
    \,
    \haus^{d-2}(\de z)
    \\
    &\qquad
      =
      \int_H
      \varphi(z)
      \rho_0(z)
      \,
      \vol_L(\de z)
      =
      \int_H
      \varphi(z)
      \rho_1(z)
      \,
      \vol_L(\de z)
    \\
    &\qquad
    =
    \int_S
    \int_{\R}
    \varphi(\gflow_L(z,t))
    \rho_1(\gflow_L(z,t))
    \,
    e^{W_{L}(z,t)}
    \,
    \de t
    \,
      \haus^{d-2}(\de z)
    \\
    &\qquad
      =
    \int_S
    \varphi(z)
    \int_{\R}
    \rho_0(\gflow_L(z,t))
    \,
    e^{W_{L}(z,t)}
    \,
    \de t
    \,
      \haus^{d-2}(\de z)
      .
  \end{align*}
  By arbitrariness of $\varphi$, we deduce the validity of~\eqref{eq:intrho0=rho1}.
  For $z\in S$, let $\tilde\mu_i^z\in\M^+(\R)$ be the measure given by
  \begin{equation}
    \tilde\mu_i^z
    :=
    \rho_i(\gflow_L(z,\dotargument))
    \,e^{W_{L}(z,\dotargument)}
    \,
    \mathcal{L}^{1}(\de t)
    ,
    \qquad
    i=0,1
    .
  \end{equation}
  Let $T(z,\dotargument):\R\to\R$ be the map transporting $\tilde\mu_0^z$
  to $\tilde\mu_1^z$ given by the monotone rearrangement, and let
  $R(z,t)=T(z,t)-t$.
  Define $V$ in the following way: if $z=\gflow_L(z',t)$, for some
  $z'\in S$, then $V(z)=R(z',t) L$.
  In other words, $V(z)=R(\QQ(z),t_L^S(z)) L$, where $\QQ:H\to S$ and
  $t_L^S:H\to \R$ are the maps uniquely defined by the relations
  \begin{equation}
    \QQ(z)\in S,
    \qquad
    (\QQ(z),z)\in\relation,
    \qquad
    \gflow_L(\QQ(z),t_L^S(z))=z
    .
  \end{equation}
  Let $\Exp(V):H\to C([0,1];H)$ be the function mapping  $z$ to the
  curve $s\mapsto \exp_z(sV(z))$.
  Define $$\nu:=(\Exp(V))_{\sharp}\mu_0.$$
  We now check that $\nu$ and $V$ satisfy the desired properties.
  First we notice that
  \begin{align*}
    \Exp(V)(\gflow_L(z,t))(s)
    &
      =
      \exp_{\gflow_L(z,t)}(sV)
      =
      \exp_{\gflow_L(z,t)}(sR(z,t) L)
    \\&
      =
      \gflow_L(z,t+s R(z,t))
      =
    \gflow_L(z,T(z,t,s))
    ,
  \end{align*}
  where $T(z,t,s)=t+sR(z,t)=(1-s)t + sT(z,t)$, or in other words
  $$(e_s\circ \Exp(V))(\gflow_L(z,t))=    \gflow_L(z,T(z,t,s)).$$
  It is clear that
  \begin{equation}
    \mu_i
    =
    \int_S
    \gflow_L(z,\dotargument)_{\sharp} \tilde\mu_i^z
    \,
    \haus^{n-2}(\de z)
    ,
    \qquad
    i=0,1
    .
  \end{equation}
  We can therefore compute
  \begin{equation}
    \label{eq:monotone-interpolation}
    \begin{aligned}
    (e_s)_{\sharp}\nu
    &
    =
      (e_s\circ\Exp(V))_{\sharp}
      \mu_0
    =
    \int_S
    (e_s\circ\Exp(V))_{\sharp} \,\gflow_L(z,\dotargument)_{\sharp} \tilde\mu_0^z
    \,
    \haus^{n-2}(\de z)
    \\
    &
          =
    \int_S
      (\gflow_L(z,T(z,\dotargument,s)))_\sharp
      \tilde\mu_0^z
    \,
      \haus^{n-2}(\de z)
           =
    \int_S
      \gflow_L(z,\dotargument)_{\sharp} T(z,\dotargument,s)_\sharp
      \tilde\mu_0^z
    \,
      \haus^{n-2}(\de z)
      .
  \end{aligned}
  \end{equation}
  The formula above has two consequences.
  If we specialize to the case $s=0,1$, we see that $(e_s)_{\sharp}\nu=\mu_s$,
  and therefore $\nu\in\OptGeo^H(\mu_0,\mu_1)$ (the fact that
  $\nu$ is optimal is trivial).
  Moreover, if $s\in[0,1]$, the formula above implies that
  $(e_s)_{\sharp}\nu$ is absolutely continuous.
  Following similar computations, one can see that the plan
  $\pi=(e_0\otimes e_1)_{\sharp}\nu$ can be represented as
  \begin{align*}
    \pi
    &
    =
    \int_S
      \gflow_L(z,\dotargument)_{\sharp} (\id\otimes T(z,\dotargument))_\sharp
      \tilde\mu_0^z
    \,
      \haus^{n-2}(\de z)
      .
  \end{align*}
  Therefore, Lemma~\ref{lem:monotonicity-of-sets} implies that $\pi$
  is monotone. It is clear from the construction that $\pi$ is induced by a map.
  The fact that $V(\gamma_0)=t_L(\gamma_0,\gamma_1)L=\dot\gamma_0$ for
  $\nu$-a.e.\ $\gamma$ follows almost by definition.
  The fact that $\mu_t$ is absolutely continuous is a
  consequence of~\eqref{eq:monotone-interpolation} and the fact that
  monotone interpolation in $\R$ preserves the absolute continuity
  w.r.t.\ the Lebesgue measure.
  
  Finally, we check that $V$ is future directed.
  Clearly, it is sufficient to check that it is future-directed
  $\mu_0$-a.e., otherwise, one can change $V$ on a negligible set, preserving the already-proven properties.
  Consider a dynamical transport plan
  $\widehat\nu\in  \OptCaus^H(\widehat\mu_0,\widehat\mu_1)$.
  Define $\widehat\pi:=(e_0\otimes e_1)_{\sharp}\widehat\nu$.
  By disintegrating the measure $\pi$  with respect to the projection on the first factor $P_1:H\times H\to H$, following the argument in the
  proof of \Cref{P:uniqueness-monotone}, one can produce a family of coupling $\widehat \pi_z$, between their marginals $\tilde\mu_0^z$ and
  $\tilde\mu_1^z$.
  Since $\pi$ gives full measure to $\Gamma$, it follows that for  for $\mu_0$-a.e. $z$, $\pi_z$ gives full measure to the set $\{(x,y) \colon y\geq
  x\}$.
  Therefore, \Cref{lem:monotone} implies that the monotone
  rearrangement of $\tilde\mu_0^z$ into $\tilde\mu_1^z$, given by the
  function $T(z,\dotargument)$, satisfies $T(z,t)\geq t$.
  In other words $R(z,t)\geq 0$, and since $V$ is obtained by
  multiplying $R$ with $L$, we conclude that $V$ is future directed.
\end{proof}

\begin{definition}
 Let $H$ be a causal, null hypersurface such that there exists a countable
  dense family of local space-like and acausal cross-sections. 
  Let $\mu_0,\mu_1\in \Prob_{ac}(H)$ be two probability measures,
  absolutely continuous w.r.t.\ $H$, such that
  $\mu_0$ is null connected to $\mu_1$ along $H$.  
  
  The unique null-geodesic dynamical transport plan
  $\nu\in\OptGeo^H(\mu_0,\mu_1)$, such that the optimal coupling $\pi:=(e_0\otimes e_1)_{\sharp}\nu$ is monotone given in Theorem \ref{T:good-representation2}, will be called \emph{monotone null-geodesic dynamical transport plan.}
\end{definition}

  \begin{remark}
    Under a finer analysis, one can see that the results of this
    section are available under the weaker assumption $g\in C^{1,1}_{loc}$ and
    $H$ of class $C^1$.
  \end{remark}

\section{Null energy condition as displacement convexity}\label{sec:NCe}

Building on top of the results about optimal transport along a null hypersurface obtained in Section~\ref{sec:OTinsideH}, in the present section we will provide a synthetic characterization of the null energy condition based on displacement convexity of the Boltzmann--Shannon entropy relative to the rigged measure along monotone null-geodesic optimal transport plans (cf.\ \Cref{T:good-representation2}) contained in null hypersurfaces.
Recall that the Boltzmann--Shannon entropy $\Ent(\mu|\mm)$ of a probability measure $\mu\in \Prob(M)$ with respect to the reference measure $\mm$ is defined by
\begin{equation}\label{eq:defSBEnt}
\Ent(\mu|\mm)=\int \rho \log \rho \, \de\mm,
\end{equation}
if $\mu=\rho\,\mm$ is absolutely continuous w.r.t.\ $\mm$ and $\left(\rho \, \log \rho \right)_+$ is $\mm$-integrable; otherwise, we adopt the convention that $\Ent(\mu|\mm)=+\infty$.

Firstly we notice that the convexity properties of the entropy along monotone null-geodesic optimal transport plans do not depend on the choice of null-geodesic vector field $L$, used to define the volume meausure $\vol_L$ on $H$ (cf.\ \Cref{T:summary}).

\begin{proposition}
    \label{P:entropy-convexity}
    Let $(M,g)$ be a Lorentzian manifold, and let $H$ be a null hypersurface.
  Let $L_1$, $L_2$ be two null-geodesic vector fields and let $\vol_{L_1}, \vol_{L_2}$ be the associated volume measures on $H$, as in \Cref{T:summary}. Let $\Phi:H\to \R$ be a $C^0$ function and let $\mm_{L_1}:=e^{\Phi}\vol_{L_1},\; \mm_{L_2}:=e^{\Phi}\vol_{L_2}$ be the corresponding weighted measures on $H$.
  Let $\mu_0,\mu_1\in \Prob(H)$ be two probability measures null connected
  along $H$ and let $\mu_t:=(e_t)_{\sharp}\nu$, $t\in [0,1]$,
  for some  $\nu\in\OptCaus^H(\mu_0, \mu_1)$.
  Let $\ee_k(t):=\Ent(\mu_t|\mm_{L_k})$, $k=1,2$, and assume that both
  functions are well-defined on $[0,1]$.

  Then the function $\ee_1-\ee_2$ is constant on $[0,1]$.
  In particular, $\ee_1$ and $\ee_2$ enjoy the same convexity
  properties.
\end{proposition}
\begin{proof}
Since $L_1$ and $L_2$ are null-geodesic vector fields, we can assume without any loss of generality that there exists
 a positive transverse function $\varphi$, such that $L_2=\varphi L_1$ on the union of the supports of the measures $\mu_{t}, \;t\in [0,1]$.
  Let $\rho_{k,t}$ be the density of $\mu_t$ w.r.t. the measure
  $\mm_{L_k}$, $k=1,2$, $t\in [0,1]$.
  Since $\rho_{2,t}=\rho_{1,t}/\varphi$, a direct
  computation gives
  \begin{align*}
    \ee_{2}(t)
    &
      =
      \Ent(\mu_t|\mm_{L_2})
    =
    \int_H
    \log(\rho_{2,t})
    \,
    \de\mu_t
    =
    \int_H
    \log(\rho_{1,t})
    \,
    \de\mu_t
    -
      \int_H
    \log \varphi
    \,
      \de\mu_t
    \\
    &
      =
      \Ent(\mu_t|\mm_{L_1})
    -
      \int_H
    \log \varphi
    \,
    \de\mu_t
      \stackrel{\text{\eqref{eq:integral-transverse}}}{=}
      \Ent(\mu_t|\mm_{L_1})
    -
      \int_H
    \log \varphi
    \,
    \de\mu_0
      =
      \ee_1(t)
    -
      \int_H
    \log \varphi
    \,
      \de\mu_0
      .
      \qedhere
  \end{align*}
\end{proof}

We now recall a classical one-dimensional result on the displacement convexity of the entropy functional.
We include a quick proof for readers' convenience. 

\begin{lemma}
  \label{lem:entropy-1d}
  Let $\mu_0\in \Prob_{ac}(\R)$ be an absolutely continuous
  probability measure, and let $T:\R\to\R$ be a non-decreasing function.
  Let $\mu_t :  =((1-t)\id+t T)_{\sharp}(\mu_0)$, 
  for all $t\in (0,1]$.
  Then $\mu_t$ is absolutely continuous for all $t\in(0,1)$.

  Moreover,
  if $\Phi:\R\to\R$ is a locally-Lipschitz function  satisfying
  \begin{equation}\label{E:CDlog}
    \Phi''
    +\frac{(\Phi')^2}{N-1}
    \leq 
    - K
    \qquad
    \text{ in the sense of distributions}
    ,
  \end{equation}
  then, having defined $\ee(t):=\Ent(\mu_t|e^{\Phi}\L^1)$, it holds
  that $\ee(\dotargument)$ is locally-Lipschitz and it satisfies
  \begin{equation}\label{E:CDentropic1d}
    \ee''
    -\frac{(\ee')^2}{N}
    \geq
    K
    \int_{\R}
    |x-T(x)|^2
    \,
    \mu_0(\de x)
    ,
    \qquad
    \text{ in the sense of distributions.}
  \end{equation}
  Conversely, if for any $(\mu_t)$ as above the entropic inequality~\eqref{E:CDentropic1d} holds, then $\Phi$ is locally-Lipschitz and satisfies~\eqref{E:CDlog}.
\end{lemma}

\begin{proof}
The first claim follows from the fact that 
$(\mu_t)_{t \in [0,1]}$ is a $W_2$-geodesic along 
$\mathcal{P}_2(\R)$ (the space of probability measures with finite second moment over $\R$) as soon as 
$\mu_0, \mu_1 \in \mathcal{P}_2(\R)$; since the claim is about a local property, it is not restrictive to 
assume $\mu_0, \mu_1 \in \mathcal{P}_2(\R)$.
Hence the first claims follows, see~\cite{villani:oldandnew}.

 The equivalence of~\eqref{E:CDlog} and~\eqref{E:CDentropic1d} can be checked by a direct computation (since it involves optimal transport on the real line with a weight). For sake of brevity, we refer to existing literature.
The inequality
\eqref{E:CDlog}
is equivalent to asking that $(\R, |\dotargument|, e^{\Phi}\mathcal{L}^1)$ satisfies the $\CD(K,N)$ condition. From 
\cite[Th.~3.12]{EKS}, this is equivalent to 
$(\R, |\dotargument|, e^{\Phi}\mathcal{L}^1)$ satisfying the 
entropic $\CD^{e}(K,N)$ (see~\cite[Definition~3.1]{EKS}) that by~\cite[Lemma~2.8,~Lemma~2.2]{EKS} is equivalent to~\eqref{E:CDentropic1d}. 
\end{proof}

We are now in position to prove 
 the convexity of the entropy along a Wasserstein geodesic
induced by a monotone null-geodesic dynamical transport
plan, provided a null Ricci lower bound holds true. 

\begin{theorem}
  \label{T:convexity-of-entropy-null}
  Let $(M,g)$ be a Lorentzian manifold of
  dimension $n$, with $g\in C^2$,  and   let $H\subset M$ be a causal null $C^2$-hypersurface that
  admits a dense  sequence $(S_i)_i$ of local space-like and
  acausal cross-sections of class $C^2$.
  Consider $\Phi:H\to \R$ a $C^{0}$
  weight function and  assume  that $H$ satisfies the
    $\NC^1(N)$ condition (see \Cref{def:NC1Quadruples}).

  Let $\mu_0,\mu_1\in \Prob_{ac}(H)$ be two probability measures
  such that $\mu_0$ is null connected to $\mu_1$ along $H$
  and let $\nu\in\OptGeo^H(\mu_0,\mu_1)$ be the unique monotone null-geodesic
  dynamical transport plan (see \Cref{T:good-representation2}).

Let $L$ be a null-geodesic vector field of class $C^1$  and let $\vol_L$ be the associated volume measure on $H$, as in \Cref{T:summary}. Set $\mm_L:=e^{\Phi} \vol_L$. Denote by $\mu_t:=(e_t)_{\sharp}\nu$, $t\in [0,1]$, and let $\ee(t):=\Ent(\mu_t|\mm_L)$.

Then  the function $[0,1]\ni t \mapsto \ee(t)$ is convex (thus, in particular, locally-Lipschitz on $(0,1)$) and it satisfies
  \begin{equation}
    \label{eq:convexity-entropy-0}
    \ee''
    -
    \frac{    (\ee')^2 }{N-1}
    \geq
    0
    ,
    \qquad
    \text{ 
  in the sense of distributions on $(0,1)$.}
  \end{equation}
\end{theorem}

The l.h.s.\ in \Cref{eq:convexity-entropy-0} does not depend on the choice of $L$:
 \Cref{P:entropy-convexity} guarantees that the
convexity properties of the entropy along Wasserstein geodesics do not
depend on the null-geodesic vector field $L$.

\begin{proof}
Following 
\Cref{P:decomposition}, we first assume the existence a global space-like and acausal cross-section $S$ for $H$.
 For each $t\in [0,1]$,  let $\rho_{t}$ be the Radon--Nikodym derivative of $\mu_t$ w.r.t. $\mm_L$.
The entropy is given by
  \begin{equation}
    \label{eq:global-entropy}
    \begin{aligned}
      \ee(t)
      &
    =
    \Ent(\mu_t|\mm_L)
    =
    \int_H
    \rho_{t}
    \log \rho_{t}
    \,
        \de\mm_L
      \\
      &
        =
        \int_S\int_{\R}
        \rho_t(\gflow_L(z,s))
        \log (\rho_t(\gflow_L(z,s)))
        \,
        e^{W_{L}(z,s)+\Phi(\gflow_L(z,s))}
        \,
        \de s
        \,
        \haus^{n-2}(\de z)
        .
        \end{aligned}
  \end{equation}
Let  $\tilde \mm_L^z :=e^{W_{L}(z,\dotargument)+\Phi(\gflow_L(z,\dotargument))} \mathcal{L}^{1}(\de s)$ and 
$\tilde  \mu_{t,z} := \rho_t(\gflow_L(z,\dotargument)) \tilde \mm_L^z $.  
  If we define
$$  
\ee_z(t)
    :=
    \Ent(\tilde\mu_{t,z}|\tilde\mm_L^z)
    =
    \int_{\R}
    \rho_t(\gflow_L(z,s))
        \log (\rho_t(\gflow_L(z,s)))
        \,
        e^{W_{L}(z,s)+\Phi(\gflow_L(z,s))}
        \,
        \de s
    ,  
$$
then~\eqref{eq:global-entropy} can be rewritten as
 $$
    \ee(t)
    =
    \int_S
    \ee_z(t)
    \,
    \haus^{n-2}(\de z).
$$
  Let $V$ be the vector field given by
\Cref{T:good-representation2}.
  Let $h :H\to[0,\infty)$ be the function such that $V=h L$ and let
  $R(z,s)=h(\gflow_L(z,s))$, for $z\in S$.
  We claim that  the function $[0,1]\ni t \to \ee_z(t)$ is convex (thus, in particular, locally-Lipschitz on $(0,1)$) and it satisfies
  \begin{equation}
    \label{eq:entropy-rays}
    \ee_z''
    -
    \frac{(\ee_z')^2}{\tilde\mu_{0,z}(\R)(N-1)}
    \geq
    0  \quad \text{ 
  in the sense of distributions on $(0,1)$.}
  \end{equation}
  Let $T(z,s)=s+R(z,s)$.
  The map $T(z,\dotargument)$ is monotone by construction and one can see that
  $\tilde\mu_{t,z}=((1-t)\id+t T(z,\dotargument))_{\sharp}(\tilde\mu_{0,z})$ (this
  computation was indeed done in the proof of
\Cref{T:good-representation2}).
  Therefore the Ricci lower bound, 
  \Cref{lem:entropy-1d} and 
  \Cref{eq:riccati-for-jacobi}
  imply  the convexity of $\ee_z(\dotargument)$ and the validity of~\eqref{eq:entropy-rays}.
  In case of a global cross-section, the thesis follows by combining~\eqref{eq:entropy-rays} with \Cref{eq:family-of-riccati} below.

In the general case,
\Cref{P:decomposition} produces a sequence of mutually
  orthogonal probability measures $(\mu_i^k)_k$, $i=0,1$, such that
  $\mu_i=\sum_{k=1}^{\infty} w_k \mu_i^k$, for some sequence $\mu_k$.
  Each measure $\mu_0^k$ is null connected to $\mu_1^k$ along
  $H$ by the dynamical transport plan
  $\nu_k\in\OptGeo^H(\mu_0^k,\mu_1^k)$ and $\sum_k
  w_k\nu_k=\nu$.
  Let $\mu_t^k:=(e_t)_{\sharp}\nu$ and notice that for all $t\in [0,1]$ the
  measure $\mu_t^k$ and $\mu_t^h$ are mutually orthogonal.
  Define $\pi^k=(e_0\otimes e_1)_{\sharp} \nu^k$ and notice that
  $\pi=\sum_k w_k\pi^k$.
  If we define $\ee_k(t):=\Ent(\mu_t^k|\mm_L)$, then the first part of the proof implies that
  \begin{equation}
    \ee_k''
    -
    \frac{(\ee_k')^2}{N-1}
    \geq
    0    .
  \end{equation}
  Let $\tilde \ee_k(t):=\Ent(w_k \mu_t^k|\mm_L)=w_k \ee_k(t)+w_k\log w_k$, and notice that
  \begin{equation}
    \tilde \ee_k''
    -
    \frac{(\ee_k')^2}{w_k(N-1)}
    \geq 0
    .
  \end{equation}
  Since the measures $\mu_{t}^k$ are mutually orthogonal, then
  $\ee(t)=\sum_{k=1}^{\infty} \tilde \ee_k(t)$, and    we conclude the proof by invoking \Cref{eq:family-of-riccati} below.
\end{proof}

The following are two technical lemmas used in the proof of \Cref{T:convexity-of-entropy-null}.

\begin{lemma}
  \label{lem:approximate-riccati}
  Let $\ee:(0,1)\to\R$ be a continuous function; let $a\geq 0$ and $b\in\R$.
  Let $\ee_\epsilon:=\ee*\rho_\epsilon$, where $\rho_\epsilon$ is a
  mollifier with support in $[-\epsilon,\epsilon]$.
  Then 
  \begin{equation}\label{eq:e''ab}
  \text{$\ee$ is locally-Lipschitz\qquad and\qquad }
    \ee''
    \geq a (\ee')^2
    + b
    ,
    \quad
    \text{ in the sense of distributions}
  \end{equation}
  if and only if
  \begin{equation}
    \label{eq:mollified-riccati}
    \ee_\epsilon''(t)
    \geq a (\ee_\epsilon'(t))^2
    + b
    ,
    \qquad
    \text{ for all $t\in(\epsilon,1-\epsilon)$, for all }\epsilon>0
    .
  \end{equation}
\end{lemma}
\begin{proof}
  ``If'' part.\quad  
  Clearly, $\ee_\epsilon\to \ee$ pointwise.
  Since $a\geq 0$, then each $\ee_\epsilon$ is $b$-convex. 
  Therefore, on every compact sub-interval of $(0,1)$, the functions
  $\ee_\epsilon$ are equi-Lipschitz, 
  for every $\epsilon>0$.
  As a consequence $\ee$ is locally-Lipschitz, the convergence 
  is locally uniform, and $\ee_\epsilon'=\ee'\star\rho_\epsilon$.
  Regarding the derivative, $u_\epsilon'$ is in $L^\infty_{loc}$, thus in
  $L^2_{loc}$ and $\ee_\epsilon'\to \ee'$ in $L^2_{loc}$.
  Therefore
  $\int_0^1 (\ee_\epsilon'(t))^2 \psi \to\int_0^1 (\ee'(t))^2 \psi$,
  $\forall \psi\in C^\infty_c((0,1))$, i.e.,
  $(\ee_\epsilon')^2\to (\ee')^2$ in the sense of distributions.
  Thus we can take the limit in~\eqref{eq:mollified-riccati} and obtain~\eqref{eq:e''ab}.
  
``Only if'' part.
  \quad
  Fix $\psi\in C^{\infty}_c((0,1))$ and $\epsilon>0$.
  A direct computation gives
  \begin{align*}
    \int
    \ee_\epsilon''(t)
    \psi(t)
    \,
    \de t
    &
    \geq
      \int
      (
      a[(\ee')^2]_{\epsilon}
      +b
      )
      \psi(t)
    \,
    \de t
      =
    \int
      a
      \int (\ee'(t))^2
      \rho_{\epsilon}(s-t)
      \,\de s\, 
      \psi(t)
      \,\de t
      +b
      \int
    \psi(t)
    \,
      \de t
    \\
    &
      \geq
    \int
      a
      \left(
      \int \ee'(t)
      \rho_{\epsilon}(s-t)
      \,\de s
      \right)^2
      \psi(t)
      \,\de t
      +b
      \int
    \psi(t)
    \,
      \de t
    \\
    &
      =
      \int
      (
      a
      \left(
      \ee_\epsilon'(t)
      \right)^2
      +b)
      \psi(t)
    \,
      \de t
      ,
  \end{align*}
  having used the convention that $[(\ee')^2]_\epsilon$ is the
  mollification of $(\ee')^2$, the fact that $\ee'$ is an $L^1_{loc}$
  function, and Jensen inequality.
\end{proof}
\begin{lemma}
  \label{eq:family-of-riccati}
  Let $(Q, \mathcal{Q}, \qq)$ be a measure space and let $\ee:Q\times (0,1)\to\R$, be a function satisfying the following assumptions:
  \begin{itemize}
\item For $\mathcal{L}^1$-a.e.\ $t\in \R$, the map $Q\ni \alpha\mapsto \ee_\alpha(t)$ is $\qq$-integrable;
\item For $\qq$-a.e.\ $\alpha\in Q$, the map $t\mapsto \ee_\alpha(t)$ is locally-Lipschitz on $(0,1)$ and it satisfies
 \begin{equation}
    \ee''_\alpha\geq a\frac{(\ee_\alpha')^2}{c(\alpha)}
    +b(\alpha)
    ,
    \quad
    \text{ in the sense of distributions,}
  \end{equation}
  where $a\geq 0$, $b\in L^1(\qq), \; b\geq 0$ $\qq$-a.e., $\int_Q c(\alpha)
  \,\qq(\de\alpha)=1$ with $c>0$ $\qq$-a.e..
\end{itemize}
  Denote by
  $\ee(t):=\int_Q \ee_\alpha(t)\,\qq(\de\alpha)$ and $B:=\int_Q b(\alpha)\,\qq(\de\alpha)$. Then the map $t\mapsto \ee(t)$ is  locally-Lipschitz on $(0,1)$ and it satisfies
  \begin{equation}
    \ee''\geq a(\ee')^2
    +B
    ,
    \quad
    \text{ in the sense of distributions.}
  \end{equation}
\end{lemma}
\begin{proof}
  Since $a,b,c$ are non-negative, then each $\ee_\alpha$ is convex on $[0,1]$, hence
  $\ee$ is convex on $[0,1]$, thus locally-Lipschitz on $(0,1)$. Let $\rho_\epsilon$ be a mollifier and define
  \begin{equation*}
    \ee_{\alpha,\epsilon}(t)
    :=\ee_\alpha*\rho_\epsilon(t)
    =
    \int_{\R}
    \ee_\alpha(t-s)
    \rho_\epsilon(s)
    \,\de s,
    \qquad
    \ee_{\epsilon}(t)
    :=\ee*\rho_\epsilon(t)
    =
    \int_{\R}
    \int_Q
    \ee_\alpha(t-s)
    \rho_\epsilon(s)
    \,\qq(\de\alpha)
    \,\de s.
  \end{equation*}
  We claim that $\ee_\epsilon(t)=\int_Q
  \ee_{\alpha,\epsilon}(t)\,\qq(\de\alpha)$.
  Indeed, it is sufficient to check that $e$ is integrable in
  $Q\times[t-\epsilon,t+\epsilon]$, and then apply Fubini theorem.
  This fact is a consequence of the following two bounds given by the
  convexity in the second variable:
  \begin{align}
    &
    \ee_\alpha(s)
    \geq
    \ee_\alpha(t-\epsilon)
    +
    (s-t+\epsilon)\frac{\ee_\alpha(t-\epsilon)-\ee_\alpha(t-2\epsilon)}{2\epsilon}
      ,
      \quad
      \forall s\geq t-\epsilon
      ,
    \\
    &
      \ee_\alpha(s)
      \leq
      \ee_\alpha(t-\epsilon)
      \frac{t+\epsilon-s}{2\epsilon}
      +
      \ee_\alpha(t+\epsilon)
      \frac{s-t+\epsilon}{2\epsilon}
      ,
      \quad
      \forall s\in [t-\epsilon,t+\epsilon]
      .
  \end{align}

  Lemma~\ref{lem:approximate-riccati} gives that $\ee_{\alpha,\epsilon}$
  solves the equation
  \begin{equation}
    \ee''_{\alpha,\epsilon}(t)\geq a\frac{(\ee_{\alpha,\epsilon}'(t))^2}{c(\alpha)}
    +b(\alpha)
    .
  \end{equation}
  We integrate over $Q$, obtaining
  \begin{align*}
    e_{\epsilon}''(t)
    &
      =
    \int_Q \ee''_{\alpha,\epsilon}(t)
    \,\qq(\de\alpha)
      \geq
      a
      \int_Q \frac{(\ee_{\alpha,\epsilon}'(t))^2}{c(\alpha)}
      \,
      \qq(\de\alpha)
      +\int_Q b(\alpha)
      \,\qq(\de\alpha)\nonumber\\
     & =
      a
      \int_Q
      \left(
      \frac{\ee_{\alpha,\epsilon}'(t)}{c(\alpha)}
      \right)^2
      c(\alpha)
      \,
      \qq(\de\alpha)
      +
      B
    \\
    &
      \geq
        a
      \left(
      \int_Q
      \frac{\ee_{\alpha,\epsilon}'(t)}{c(\alpha)}
      c(\alpha)
      \,
      \qq(\de\alpha)
      \right)^2
      +
      B
      =
      a
      (\ee_{\epsilon}'(t))^2
      +B
      ,
  \end{align*}
  where we have used the Jensen inequality.
  We conclude by taking the limit as $\epsilon\to0$ using
  Lemma~\ref{lem:approximate-riccati}.
\end{proof}

\Cref{T:convexity-of-entropy-null} suggests the following two definitions, one
concerning null hypersurfaces and one concerning the ambient manifold. In order to state them in a more concise form, let us  consider the following dimensional variant of the Boltzmann--Shannon entropy
\begin{equation}\label{eq:defSPEnt}
{\rm U}_N(\mu|\mm):=\exp\left(-\frac{1}{N}\Ent(\mu|\mm) \right),
\end{equation}
which was studied in~\cite{EKS} in connection to Ricci bounds. This kind of functional is well-known in information theory as the
Shannon entropy power, see e.g.~\cite{DCT-1991}. For instance, the entropy power on $\R^N$ is the functional ${\rm U}_{N/2}$.

\begin{definition}\label{D:entropynullCD}
  Let $(M,g)$ be a Lorentzian manifold with $g\in C^2$
  and $H\subset M$ be a causal null $C^2$-hypersurface
  admitting a dense sequence of local space-like and
  acausal cross-sections of class $C^2$. Let $\vol_L$ be as in  \Cref{T:summary}. 
  Let $\Phi: H \to\R$  be a $C^0$ function
  and define $\mm_L:=  e^{\Phi}\vol_L$.

  We say that the quadruple $(M,g,H, \Phi)$ satisfies the null energy condition
  $\NC^e(N)$ if
  and only if the following holds:

  For every $\mu_0,\mu_1\in \Prob_{ac}(H)$ probability measures
  such that $\mu_0$ is null connected to $\mu_1$ along $H$,
  there exists  a null-geodesic
  dynamical transport plan $\nu\in\OptGeo^H(\mu_0,\mu_1)$,
  such that the function  ${\rm u}_{N-1}(t):={\rm U}_{N-1} ((e_t)_\sharp
    \nu|\mm_L)$ satisfies 
  \begin{equation}\label{eq:defNCE}
    {\rm u}_{N-1}(t)\geq
    (1-t)
    {\rm u}_{N-1}(0)
    +
    t\, 
    {\rm u}_{N-1}(1)
    ,
    \qquad
    \forall t\in[0,1]
    .
  \end{equation}
  \end{definition}

\begin{definition}\label{def:NCEspace}
  Let $(M,g)$ be a Lorentzian manifold with $g\in C^2$ and $\Phi:M\to\R$ a $C^0$ function.
  We say that the triple $(M,g,\Phi)$ satisfies the null energy
  condition $\NC^e(N)$ if and only if for any null causal $C^2$-hypersurface
  $H\subset M$ admitting a dense sequence of local space-like and
  acausal cross-sections of class $C^2$, the quadruple $(M,g,H, \Phi)$ satisfies the null energy condition $\NC^e(N)$.
\end{definition}

\begin{remark}\label{rem:NCeNoL}
\Cref{D:entropynullCD} is well-posed, i.e. it does not depend on the
particular choice of the global Borel null vector field $L$. Indeed,
\Cref{P:entropy-convexity} implies that on each 
saturated set $H_{S_k}$ the convexity properties of the entropy along $\mu_t$ do not depend on $L$. 
To obtain the full claim, one can argue as in the proof of \Cref{T:convexity-of-entropy-null}.
\end{remark}

\begin{remark}\label{R:implications}
With the just-introduced definitions,
\Cref{T:convexity-of-entropy-null} can be stated as $\NC^1(N)$ implies
$\NC^e(N)$ for null hypersurfaces, and therefore for spaces.
In next section we will prove the converse implication.
\end{remark}

\begin{remark}
  Note that \Cref{D:entropynullCD} is slightly different from
  the thesis of \Cref{T:convexity-of-entropy-null}, in three
  aspects. The first one is merely cosmetic: \Cref{T:convexity-of-entropy-null} is phrased as a convexity property of the Boltzmann--Shannon entropy $\Ent$, see~\eqref{eq:defSBEnt}, while \Cref{D:entropynullCD} writes as a concavity condition on the Shannon entropy power ${\rm U}_N$,  see~\eqref{eq:defSPEnt}. However one can easily pass from one formulation to the other via a change of variables. The two slightly more substantial differences are the following.
  On the one hand, \Cref{T:convexity-of-entropy-null} deals with monotone plans, whereas \Cref{D:entropynullCD} does not require plans to be monotone. In this regard, \Cref{D:entropynullCD} is less restrictive on the allowed transports.
  On the other hand, \Cref{T:convexity-of-entropy-null} guarantees a full convexity of the
  entropy along monotone transports (which always exist, thanks to \Cref{T:good-representation2}), whereas the condition requested in \Cref{D:entropynullCD}  is a concavity-type inequality only at
 the end-points.

The choice
  for \Cref{D:entropynullCD} is motivated by the subsequent
  \Cref{th:entropic-to-distributional} showing that checking convexity
  along a causal plan is enough to obtain the $\NC^{1}$ condition 
  \end{remark}

\section{Equivalence of the various formulations}\label{S:converse}

\subsection{Equivalence of the synthetic energy conditions $\NC^e(N)$ and $\NC^1(N)$}

In the previous section (see \Cref{T:convexity-of-entropy-null} and \Cref{R:implications}), we proved that $\NC^1(N)$ implies $\NC^e(N)$. In the present section, we prove the converse implication  
$\NC^e(N)\implies \NC^1(N)$, under 
the  assumptions 
$\Phi\in C^0$ and $g\in C^2$.
We start with a few technical lemmas.

\begin{lemma}
  \label{lem:mixture-riccati}
  Let $(Q, \mathcal{Q}, \qq)$ be a measure space, $t\in (0,1)$, and
  $\ee:Q\times \{0,t,1\}\to \R$ a measurable function.
  Fix $G:\overline{\R}\to\overline{\R}$ a continuous, strictly
  decreasing function.
  Let $F=\{w\in L^\infty(Q,\q): w\geq 0, \int_Q w\,\de\q=1\}$.
  If $w\in F$, let $\ee_w(t)=\int_Q w(\alpha)\,
  \ee(\alpha, t)\,\q(\de\alpha)$.
  Define
  \begin{equation}
    u_w= G(\ee_w)
    \qquad
    \text{ and }
    \qquad
    u_\alpha=
    G(\ee_\alpha)
    .
  \end{equation}
  Assume that for all
  $w\in F$ satisfies
  \begin{equation}
    u_w(t)
    \geq
    (1-t)
    u_w(0)
    +
    t
    u_w(1)    .
  \end{equation}
  Then, for $\q$-a.e.\ $\alpha\in Q$, it holds that
    \begin{equation}
    u_\alpha(t)
    \geq
    (1-t)
    u_\alpha(0)
    +
    t
    u_\alpha(1)    .
  \end{equation}  
\end{lemma}

\begin{proof}
  Fix $h,k\in G(\overline{\R})$ and let
  \begin{equation*}
    A_{h,k}
    =
    \{
    \alpha\in Q
    :
    u_\alpha(0)\leq h
    ,
    u_\alpha(1)\leq k
    \}
    =
    \{
    \alpha\in Q
    :
    \ee_\alpha(0)\geq G^{-1}(h)
    ,
    \ee_\alpha(1)\geq G^{-1}(k)
    \}
    .
  \end{equation*}
  Fix $B\subset A_{h,k}$ and take $w=\frac{1}{\q(B)}\indicator_B$.
  A direct computation gives
  \begin{align*}
    u_w(t)
    &
      \leq
      (1-t) u_w(0) + tu_w(0)
    \\
    &
      =
      (1-t)G
      \left(
      \frac{1}{\q(B)}
      \int_{B} \ee_\alpha(0)\,\q(\de\alpha)
      \right)
      +
      tG
      \left(
      \frac{1}{\q(B)}
      \int_{B} \ee_\alpha(1)\,\q(\de\alpha)
      \right)
    \\
    &
      \leq
      (1-t) h+ t k
      ,
  \end{align*}
thus
$$
    \frac{1}{\q(B)}
    \int_B \ee_\alpha(t)
    \geq
    G^{-1}
    \left(
      (1-t) h +t k
    \right)
    .
$$
By arbitrariness of $B$, we deduce
$$
    \ee_\alpha(t)
    \geq
    G^{-1}
    \left(
      (1-t) h +t k
    \right)
    ,
    \qquad
    \text{ for $\q$-a.e.\ }
    \alpha\in A_{h,k}
    .
$$
Hence the set
$$
    \tilde A_{h,k}
    :=
    \{
    \alpha\in Q:
    \alpha \in A_{h,k}
    \implies
    u_\alpha(z)\leq (1-t) h+t k
    \}
$$
has full measure in $Q$.
  Define the full-measure set
  $\tilde A:=\bigcap_{h,k\in G(\overline{\R})\cap \Q} \tilde A_{h,k}$.
  Assume by contradiction that for some $\epsilon>0$, for some
  $\alpha\in \tilde A$ it holds that
  \begin{equation}
    \label{eq:mixture-contradiction}
    u_\alpha(t) \geq
    (1-t) u_\alpha(0)
    +
    t u_\alpha(1)
    +2\epsilon
    .
  \end{equation}
  In this case, choose $h\in [u_\alpha(0),u_\alpha(0)+\epsilon]\cap \Q
  \subset G(\overline{\R})$ and
  $k\in [u_\alpha(1),u_\alpha(1)+\epsilon]\cap \Q \subset G(\overline{\R})$.
  By construction, $\alpha\in \tilde A_{h,k}$, therefore
$$
    u_\alpha(t)\leq (1-t) h+t k
    \leq
    (1-t) (u_\alpha(0)+\epsilon)+t (e_\alpha(y)+\epsilon)
    =
    (1-t) u_\alpha(0)+t e_\alpha(y)+\epsilon
    ,
$$
which contradicts 
\eqref{eq:mixture-contradiction}.
\end{proof}

The next lemma is a sort of ``Brunn--Minkowski inequality implies
log-concavity'' that suits our setting.

\begin{lemma}
  \label{lem:brunn-minkowski-to-cd}
  Fix $N>1$.
  Let $I\subset \R$ be an interval and $\Phi:I\to \R$ be continuous
  function and let $\mm:=e^{-\Phi}\L^1$.
  Given $x_0 < x_1 $ in $I\cap \Q$, define $x_t:= (1-t)x_0 + t x_1$.
  Similarly, for $\epsilon_0,\epsilon_1>0$ in $\Q$, define
  $\epsilon_t=(1-t) \epsilon_0+t \epsilon_1$.
  Assume that for all $x_i$ and $\epsilon_i$ as above, $i=0,1$, it
  holds that
  \begin{equation}
    \mm([x_t,x_t+\epsilon_t])^{\frac{1}{N}}
    \geq
    (1-t)\mm([x_0,x_0+\epsilon_0])^{\frac{1}{N}}
    +
    t\mm([x_1,x_1+\epsilon_1])^{\frac{1}{N}}
    ,
    \qquad
    \forall t\in [0,1]
    \cap \Q
    .
  \end{equation}

  Then $\Phi$ is locally-Lipschitz and satisfies
  \begin{equation}
    \Phi''
    +
    \frac{(\Phi')^2}{N-1}
    \leq
    0
    ,
    \qquad
    \text{ in the sense of distributions}.
  \end{equation}
  \end{lemma}

  \begin{proof}
  Fix $x_0<x_1$ as in the hypothesis.
  Let $\lambda=(\lambda_0,\lambda_1)$, such that
  $\lambda_0,\lambda_1>0$ to be fixed later.
  Let $(\lambda^n)_n$ be a sequence in $\Q^2$ converging to $\lambda$
  and let $\lambda^n_t=(1-t)\lambda^n_0+t \lambda^n_1$.
  Let $(\delta_n)_n$ be an infinitesimal sequence.
  We define
  \begin{equation}
    u_n(t)
    :=
    \delta_n^{-\frac{1}{N}}
    \mm([x_t,x_t+\delta_n\lambda^n_t])^{\frac{1}{N}}
    .
  \end{equation}
  We can easily compute the limit
  \begin{align*}
    u(t)
    :=
    \lim_{n\to\infty}
    u_n(t)
    =
    \left(
    \lim_{n\to\infty}
    \delta_n^{-1}
    \int_{x_t}^{x_t+\delta_n\lambda^n_t}
    \Phi(y)
    \,
    \de y
    \right)^{\frac{1}{N}}
      =
      (\lambda_t
      e^{\Phi(x_t)}
      )^{\frac{1}{N}}
      .
  \end{align*}
  Since the concavity passes to pointwise limit, we have that
  \begin{align*}
    e^{\Phi(x_t)/N}
    \geq
    \lambda_t^{-\frac{1}{N}}
    \left(
    (1-t)
    (
    \lambda_0
    e^{\Phi(x_0)}
    )^{\frac{1}{N}}
    +
    t
    (
    \lambda_1
    e^{\Phi(x_1)}
    )^{\frac{1}{N}}
    \right)
    .
  \end{align*}
  The choice $\lambda_i=e^{\Phi(x_i)/(N-1)}$, for $i=0,1$, yields
  \begin{align*}
        e^{\Phi(x_t)/N}
    &
      \geq
    \lambda_t^{-\frac{1}{N}}
    \left(
    (1-t)
      (e^{\Phi(x_0)})^{
      (\frac{1}{N-1}+1)
      \frac{1}{N}
      }
    +
    t
      (e^{\Phi(x_1)})^{
      (\frac{1}{N-1}+1)
      \frac{1}{N}
      }
    \right)
    \\
    &
      =
    \lambda_t^{-\frac{1}{N}}
    \left(
    (1-t)
      e^{\Phi(x_0)/(N-1)}
    +
    t
      e^{\Phi(x_1)/(N-1)}
      \right)
    \\
    &
      =
    \lambda_t^{-\frac{1}{N}}
    \left(
    (1-t)
      \lambda_0
    +
    t
      \lambda_1
      \right)
      =
      \lambda_t^{\frac{N-1}{N}}
      =
          \left(
    (1-t)
      e^{\Phi(x_0)/(N-1)}
    +
    t
      e^{\Phi(x_1)/(N-1)}
      \right)^{\frac{N-1}{N}}
      .
  \end{align*}
  By arbitrariness of $x_0$ and $x_1$, we deduce
  \begin{equation*}
    e^{\Phi((1-t)x_0+t x_1)/(N-1)}
    \geq
    (1-t)
    e^{\Phi(x_0)/(N-1)}
    +
    t
    e^{\Phi(x_1)/(N-1)}
    ,
    \qquad
    \forall
    x_0,x_1\in I\cap \Q
    ,
    \forall t\in [0,1]\cap \Q
    .
  \end{equation*}
  The continuity of $\Phi$ guarantees that the constraint of
  $x_0,x_1,t$ to be rational can be dropped, or in other words
  $e^{\Phi/(N-1)}$ is concave, which is the thesis.
\end{proof}

\begin{lemma}
\label{lem:wasserstein-mixture}
Let $(Q,\mathcal{Q},\q)$ be a measure space, $I\subset \R$ an 
interval, and
$\Phi:Q\times I\to \R$ a function, measurable in the first
variable and continuous in the second variable.
Let 
$\mm:=e^{\Phi}\,\q\otimes\L^1\llcorner_{I}\in\M^{+}(Q\times
  I)$ and fix $N>0$. 
  Define $F_t:Q\times I\times I\to Q\times I$ as
  $F_t(\alpha,x,y)=(\alpha,(1-t)x+t y)$.
  For a given measure $\pi\in\Prob(Q\times I\times I)$, define
  \begin{align*}
    &
      \ee_\pi(t)
      :=
      \Ent((F_t)_{\sharp}\pi|\mm)
      \in[-\infty,\infty]
      ,
      \qquad
      u_\pi(t)
      :=
      \exp
      \left(
      -\frac{\ee_\pi(t)}{N}
      \right)
      .
  \end{align*}
  Assume that, for every $w\in L^\infty(Q)$, with $\int_Q w\,\de\q=1$ and
  $w\geq 0$,
  and for every $\hat \mu_0,\hat \mu_1\in\Prob(I)$, such that $\supp
  \mu_0\subset (-\infty,a)$ and $\supp \mu_1\subset (a,+\infty)$, for
  some $a\in I$,
  there exists
  $\pi\in\Prob(Q\times I\times I)$, such that
  \begin{equation}
    (P_1, P_2)_{\sharp}\pi
    =
    w\q\otimes\hat\mu_0
    \qquad
    \text{ and }
    \qquad
    (P_1, P_3)_{\sharp}\pi
    =
    w\q\otimes\hat\mu_1
  \end{equation}
  and that the function $u_\pi$ is satisfies
  \begin{equation}\label{E:convexitylemma}
    u_\pi(t)
    \geq
    (1-t)
    u_\pi(0)
    +
    t
    u_\pi(1)
    ,
    \qquad
    \forall t\in[0,1]
    .
  \end{equation}

Then, for $\qq$-a.e.\ $\alpha\in Q$, it holds that  the function $t\mapsto \Phi(\alpha,t)$ is locally-Lipscitz and it satisfies
\begin{equation}
    \label{eq:riccati-for-fibers}
\Phi(\alpha,\dotargument)''
    +
\frac{(\Phi(\alpha,\dotargument)')^2}{N-1}
    \leq
    0
    ,
    \qquad
    \text{ in the sense of distributions.}
\end{equation}
\end{lemma}

\begin{proof}
  Fix $x=(x_0,x_1)\in(\Q\cap I)^2$, with $x_0<x_1$ and
  $\epsilon=(\epsilon_0,\epsilon_1)\in \Q^2$ with
  $\epsilon_0,\epsilon_1>0$ small enough; define  $x_t=(1-t)x_0+t x_1$ and
  $\epsilon_t=(1-t)\epsilon_0+t\epsilon_1$.
  Fix also $w\in L^\infty(Q)$, as above.
  We define
  \begin{equation*}
    \hat\mu_i^{x,\epsilon}
    =
    \frac{1}{\epsilon_i}
    \L^1\llcorner_{[x_i,x_i+\epsilon_i]},
    \qquad
    \mu_i^{x,\epsilon,w}
    :=
    w\q\otimes
    \hat\mu_i^{x,\epsilon}
    ,
    \qquad
    i=0,1
    .
  \end{equation*}
  The measures $\hat\mu_i^{x,\epsilon,w}$, $i=0,1$, satisfy the
  assumption, thus there exists $\pi_w$, as above.
  In particular, since $\mu_i^{x,\epsilon,w}$, $i=0,1$, are absolutely
  continuous, then inequality~\eqref{E:convexitylemma} implies that
  also $\mu_t^{x,\epsilon,w}:=(F_t)_{\sharp}\pi_w$,
  $t\in[0,1]$, is absolutely continuous w.r.t.\ $\mm$.
  We can therefore write
  \begin{equation}
    \mu_t^{x,\epsilon,w}
    =
    w(\alpha)
    \rho^{x,\epsilon,w}_{t,\alpha}
    \q
    \otimes \Leb^1
    ,
  \end{equation}
  for some function $ \rho^{x,\epsilon,w}_{t,\alpha}$ which satisfies
  the following immediate properties
  \begin{align}
    &
      \int_{x_t}^{x_t+\epsilon_t}
          \rho^{x,\epsilon,w}_{t,\alpha}(y)
      \,
      \de y
      =1
      \qquad
      \text{ and }
      \qquad
      \supp     \rho^{x,\epsilon,w}_{t,\alpha}
      \subset
      [x_t,x_t+\epsilon_t]
      .
  \end{align}
  We now compute the entropy
  \begin{align*}
    \ee^{x,\epsilon,w}(t)
    :=&
        \Ent(\mu_t^{x,\epsilon,w}|\mm)
        =
        \int_Q
        \int_{x_t}^{x_t+\epsilon_t}
         w(\alpha)
         \rho^{x,\epsilon,w}_{t,\alpha}(y)
    \log
    \big(
        w(\alpha)
        \rho^{x,\epsilon,w}_{t,\alpha}(y)
        e^{-\Phi(\alpha,y)}
        \big)
        \,
        \de y
        \,
        \de\q
    \\
    =
      &
      \int_Q
      w(\alpha)
      \log
      w(\alpha)
      \,
      \q(\de\alpha)
      +
        \int_Q
        \int_{x_t}^{x_t+\epsilon_t}
         w(\alpha)
         \rho^{x,\epsilon,w}_{t,\alpha}(y)
    \log
    \big(
        \rho^{x,\epsilon,w}_{t,\alpha}(y)
        e^{-\Phi(\alpha,y)}
        \big)
        \,
        \de y
        \,
        \de\q
    \\
    =&
      \Ent(w\q|\q)
      +
        \int_Q
         w(\alpha)
        \int_{x_t}^{x_t+\epsilon_t}
         \rho^{x,\epsilon,w}_{t,\alpha}(y)
    \log
    \big(
        \rho^{x,\epsilon,w}_{t,\alpha}(y)
        e^{-\Phi(\alpha,y)}
        \big)
        \,
        \de y
        \,
        \de\q
       .
  \end{align*}
  Next, we define
  $\mm_\alpha:=e^{\Phi(\alpha,\dotargument)}\,\L^1\llcorner_{I}$ and
  compute (using Jensen's inequality for the function $\rho\log\rho$,
  w.r.t.\ the measure $\mm_\alpha$)
  \begin{align*}
    \ee_\alpha^{x,\epsilon}(t)
    :&=
      -
      \log(
      \mm_\alpha([x_t,x_t+\epsilon_t])
      )
    \\
    &
      =
      \int_{x_t}^{x_t+\epsilon_t}
      \rho^{x,\epsilon,w}_{t,\alpha}(y)
      e^{-\Phi(\alpha,y)}
      \,
      \mm_\alpha(\de y)
      \log
      \left(
      \frac{
      \int_{x_t}^{x_t+\epsilon_t}
      \rho^{x,\epsilon,w}_{t,\alpha}(y)
      e^{-\Phi(\alpha,y)}
      \,
      \mm_\alpha(\de y)
      }{\mm_\alpha([x_t,x_t+\epsilon_t])}
      \right)
          \\
    &
\leq
      \int_{x_t}^{x_t+\epsilon_t}
      \rho^{x,\epsilon,w}_{t,\alpha}
      e^{-\Phi(\alpha,y)}
      \log
      \big(
      \rho^{x,\epsilon,w}_{t,\alpha}
      e^{-\Phi(\alpha,y)}
      \big)
      \,
      \mm_\alpha(\de y)
    \\
    &
      =
      \int_{x_t}^{x_t+\epsilon_t}
      \rho^{x,\epsilon,w}_{t,\alpha}
      \log
      \big(
      \rho^{x,\epsilon,w}_{t,\alpha}
      e^{-\Phi(\alpha,y)}
      \big)
      \,
      \de y
    ,
  \end{align*}
  with equality for $t=0,1$.
  Therefore, it holds that
  \begin{equation}
    \ee^{x,\epsilon,w}(t)
    \geq
    \Ent(w\q|q)
    +
    \int_Q
    w(\alpha)
    \ee_\alpha^{x,\epsilon}
    (t)
    \,
    \q(\de\alpha)
    ,
  \end{equation}
  with equality at the endpoints.
  Let
  \begin{equation*}
    u^{x,\epsilon,w}(t)
    :=
    \exp
    \left(
      -
      \frac{\ee^{x,\epsilon,w}(t)}{N}
    \right)
    ,
    \qquad
    u^{x,\epsilon,\alpha}(t)
    :=
    \exp
    \left(
      -
      \frac{\ee^{x,\epsilon,\alpha}(t)}{N}
    \right)
    =
    \mm_\alpha([x_t,x_t+\epsilon_t])
    .
  \end{equation*}
  Fix $t\in(0,1)$.
  By assumption, $u^{x,\epsilon,w}$ verifies
 ~\eqref{E:convexitylemma}, thus we can apply
  \Cref{lem:mixture-riccati} (with $G(b)=\exp(-b/(N-1))$ and deduce  that for $\q$-a.e.\ $\alpha$ it
  holds that
  \begin{equation*}
    u^{x,\epsilon,\alpha}(t)
    \geq
    (1-t)
    u^{x,\epsilon,\alpha}(0)
    +
    t
    u^{x,\epsilon,\alpha}(1)
    .
  \end{equation*}
  By taking a countable intersection, we deduce that the inequality
  above holds true for all $t\in[0,1]\cap \Q$.
  We can thus conclude by applying \Cref{lem:brunn-minkowski-to-cd}.
\end{proof}

We are now in position to prove that the displacement convexity of the entropy implies the $\NC^1(N)$ condition.

\begin{theorem}[$\NC^e(N)\Rightarrow\NC^1(N)$]
  \label{th:entropic-to-distributional}
  Let $(M,g)$ be a Lorentzian manifold, let $H$ be a causal null-hypersurface and $(S_k)_k$
be a dense sequence of local space-like and acausal cross-sections   
for $H$. Assume $g,H,S_k$ to be of class $C^2$.
  Consider $\Phi:H \to \R$ a $C^0$
  weight function and define $\mm_L : =e^{\Phi}\vol_L$.

If $(M,g,H, \Phi)$ satisfies the null energy condition $\NC^e(N)$,
then, for every $z\in S$, the function
$a_z(t):=W_{L}(z,t)+\Phi(\gflow_L(z,t))$  is locally-Lipschitz and it satisfies
\begin{equation}
  \label{eq:convexity-density}
  a_z''
  +
  \frac{(a_z')^2}{N-2}
  \leq 0, \qquad \text{in the sense of distributions on $(0,1)$.}
\end{equation}
\end{theorem}

\begin{proof}
We will consider only the case for $t$
  in a neighborhood of $0$: if $t$ is far away from $0$, we can take
  an auxiliary cross section and apply \Cref{L:domainJ}.
  We can also restrict $H$ so that $S$ is a global space-like and acausal cross-section (\Cref{T:summary}).
  Up to multiplying $L$ by a transverse function (no change in the derivatives of the densities of the measures) 
  and further restricting $H$, we can also assume
  that $\gflow_L(S\times (-1,1)) = H$. 

  Our aim is now to apply \Cref{lem:wasserstein-mixture}.
  Take $I=(-1,1)$, $Q=S$, $\q=\haus^{d-2}$,
  and $\hat\mm=a_z(t)\q(\de z)\otimes \de t\in\M^{+}(Q\times I)$.
  Fix $\hat \mu_i\in \Prob(I)$, $i=0,1$, absolutely continuous, such that $\supp
  \hat \mu_0\subset (-\infty,b)$ and $\hat \mu_1\subset (b,+\infty)$,
  for some $b\in I$.
  Fix $w\in L^\infty(Q)$ non-negative, with $\int_Q w=1$.
  Define $ \mu_i:=(\gflow_L)_{\sharp}(w \q\otimes\hat \mu_i)$, $i=0,1$.
  Let $\nu\in \OptGeo^H(\mu_0,\mu_1)$ by null-geodesics dynamical
  transport plan given by the definition of $\NC^e(N)$; in other
  words, $\nu$ satisfies
  \begin{equation}
    \label{eq:convexity-endpoints-again}
    u(t)\geq
    (1-t)
    u(0)
    +
    t
    u(1),
    \qquad
    \text{ where }
    u(t):=
    \exp
    \left(
      -
      \frac{\Ent((e_t)_{\sharp}\nu|\mm)}{N-1}
    \right)
    .
  \end{equation}
  
  We now construct $\pi$.
  Let $P_S:H\to S$, given by $P_S(x)=\alpha$, if
  $(x,\alpha)\in \relation$; let $G_L:H\to \R$, given by
  $\Psi_L(P_S(x),G_L(x))=x$ (in other words, $(P_S,G_L)$ is the
  inverse of $\Psi_L$).
  Let $\Gamma$ a map that given a causal curve $\gamma$
  it returns
  $\Gamma(\gamma)=(P_S(\gamma_0),G_L(\gamma_0),G_L(\gamma_1))$.
  We can thus define $\pi:=\Gamma_{\sharp}\gamma$.

  We need to check that $\pi$ satisfies the hypothesis of
  \Cref{lem:wasserstein-mixture}.
  Let $F_t$ as in the hypothesis of \Cref{lem:wasserstein-mixture}.
  Since $F_t\circ \Gamma=(P_S,G_L)\circ e_t$, it holds that
  \begin{align*}
    \Ent((F_t)_{\sharp}\pi | \hat \mm)
    &
    =
    \Ent((F_t\circ \Gamma)_{\sharp}\nu | \hat\mm)
    =
    \Ent(((P_S,G_L)\circ e_t)_{\sharp}\nu | \hat \mm)
    \\
    &
      =
    \Ent(((P_S,G_L)\circ e_t)_{\sharp}\nu | (P_S,G_L)_{\sharp} \mm_L)
    =
    \Ent(( e_t)_{\sharp}\nu | \mm_L),
  \end{align*}
  having used the fact that $\hat\mm=(P_S,G_L)_{\sharp}\mm_L$.
  We can thus deduce from~\eqref{eq:convexity-endpoints-again} the
  assumption~\eqref{E:convexitylemma} of
  \Cref{lem:wasserstein-mixture}.

  Applying this Lemma, we deduce that for $\q$-a.e.\
  $z$,  the function $I \ni t\mapsto W_{L}(z,t)+\Phi(\gflow_L(z,t))$ is 
  locally-Lipschitz and it satisfies~\eqref{eq:convexity-density}.
  Finally, the ``for $\q$-a.e.\ $z\in S$'' improves to ``for all $z\in
  S$'' thanks to the continuity of both $W_{L}$ and $\Phi$.
\end{proof}


\subsection{From 
displacement convexity to classical lower Ricci bounds.}
\label{Ss:futurecone}

Let us start by reviewing the implications proved so far.
Under the regularity assumptions
  $\Phi\in C^0$ and $g\in C^2$, we showed that
  the conditions $\NC^e(N)$ and $\NC^1(N)$ are equivalent,
  for a fixed weighted null hypersurface (see \Cref{T:convexity-of-entropy-null} and \Cref{th:entropic-to-distributional}); as a consequence, $\NC^e(N)$ and $\NC^1(N)$ are two equivalent
  conditions for spaces.

  Both conditions, under the regularity assumptions 
  $\Phi\in C^2$
  and $g\in C^2$ are implied by $\Ric^{g,\Phi,N}(v)\geq 0$, for every
  light-like vector $v$.

  We now close this circle of implications, by showing that if a
Lorentzian manifold enjoys the convexity of the entropy in the way
specified above, then $\Ric(v,v)\geq0$ 
in the null directions.

To obtain the classical null-energy condition, we will merely use the convexity of the entropy inside 
future light-cones.

We fix few notations that will be used in the sequel.
Given a Lorenztian manifold $M$ and a point $p\in M$, we
consider the set
\begin{equation}
  \hat H
  :
  =\{v\in T_pM: v
  \text{ is future-directed and }
  g(v,v)=0
  \},
\end{equation}
i.e.
the future light-cone in the tangent space without the tip.
Let $U\subset T_pM$ be a neighborhood of the origin such that
$\exp_p|_{U}$ is a diffeomorphism on its image.
We define $H:=\exp_p(\hat H\cap U)$.
It is clear that $H$ is a null hypersurface.
Any null hypersurface constructed in this way will be called \emph{local future
light-cone}.

\begin{theorem}[$\NC^1(n)$ on light-cones $\Rightarrow$ NEC]\label{thm:NC1toNEC}
  Let $(M,g)$ be a Lorentzian manifold of dimension $n$, with $g$ of
  class $C^2$.
  Assume that every local future light-cone $H$ satisfies the
   null energy condition $\NC^1(n)$.

  Then $\Ric_g(v,v)\geq 0$ for any $v\in TM$, such that $g(v,v)=0$.
\end{theorem}

\begin{proof}
By the the definition of $W_{L}$, the hypothesis coincides with the convexity of the function
\begin{equation*}
  t\mapsto \exp(W_{L}(z,t)/(n-2))=(\det(J(t)))^{1/(n-2)}
  .
\end{equation*}
Recall also that from~\eqref{eq:J-appearence} we have
  $\det J(t)=\det \bar J(t)$, where $\bar J$ is the upper-left
  $(n-2)\times(n-2)$-minor of $J$.

  Fix $p\in M$ and assume by contradiction that there exists  $v\in T_pM$ such that  $g(v,v)=0$ and $\Ric(v,v)<0$.
  Let $(e_i)_i$ be a basis of $T_pM$, such that:
  \begin{align}
    &
      g(e_i,e_j)=\delta_{ij}
      ,
      \qquad
      i=1,\dots,n-2, \,\, j=1,\dots,n
    \\
    &
      e_{n-1}=v
      ,
      \qquad
      g(e_{n-1},e_{n})=-1
      ,
      \qquad
      g(e_{n},e_{n})=0
      .
  \end{align}
  We endow $T_pM$ with a (Riemannian) scalar product $\hat g$, such that $(e_i)_i$ is
  an orthonormal basis for $\hat g$; in particular $\hat g(v,v)=1$.
  Up to rescaling $v$ (and the corresponding $(e_i)_i$), we can assume that
  \begin{equation*}
    \hat S
    :=
    \{
    w\in
    \hat H
    :
    \hat g(w,w)=1
    \}
    \cap
    B_{\epsilon}^{\hat{g}}(v)
    \subset U
  \quad
  \text{ and }
  \quad
  S:=\exp_p(\hat S)
  \subset H
    ,
  \end{equation*}
  for some $\epsilon>0$ small enough.
  Clearly, $S$ is a global space-like cross-section for $H$.
  We define the map $\hat L: \hat H\to T_pM$ as $\hat
  L(w)= w/\sqrt{\hat g(w,w)}$.
  Using the standard identification of $T_pM$ with $T_w(T_pM)$, we can
  define $L(\exp_p(w)):=\de (\exp_p)_w[\hat L]$.
  It is clear that $L$ is a null-geodesic vector field on $H$.
  Let $\gamma_t=\exp_p(t v)$, and let
  $z=\gamma_1=\exp_p( v)\in S$.

  With a slight abuse of notation, we still denote with $(e_i)_i$ the
  basis of $T_zS$ given by
  $e_i=\de(\exp_p)_v[e_i]$, $i=1,\dots, n$;
  in other words $(e_i)_i$  is obtained via parallel
  transport along the curve $\gamma$.
  Notice that $e_{n-1}=L$ and $e_n=\overline{L^S}$.
  We now consider the Jacobi fields $J_i=J_{e_i}$ (recall the notation
  of Section~\ref{S:jacobi}).

It is a standard fact of Lorentzian (and also Riemannian) geometry
  (see, e.g.,~\cite[Sec.~3.C.3]{GallotHulinLafontaine04}) that $J_{e_i}$ can be
  characterized as 
  \begin{equation}\label{eq:JacExp}
    J_{e_i}(t-1)     =
    \de (\exp_p)_{tv}[te_i]
    ,
    \qquad
    i=1,\dots, n-2
    .
  \end{equation}
  Moreover, for $t=0$, it holds that
  \begin{equation}\label{eq:JIC}
    J_{e_i}(-1) =0
    \qquad
    \text{ and }
    \qquad
   J_{e_i}'(-1) =e_i
    ,
    \qquad
    i=1,\dots, n-2
    .
  \end{equation}
  Using the coordinate system introduced in Section~\ref{S:jacobi} (we
  recall that $R$ represents the Riemann curvature tensor, $J$
  represents the Jacobi flow, and $J_i$ is the $i$-th row of $J$), we
  can write
  \begin{align*}
    &
      J_i''(t)=-J_i(t)R(t)
      ,
      \qquad
      J_i(-1)=0
      ,
      \qquad
      J_i'(-1)
      =\mathbf{f}_i^T
      ,
      \qquad
      i=1,\dots,n-2
      ,
  \end{align*}
  and thus $J_i''(-1)=0$.
  The $C^2$-regularity of the metric guarantees that $R$ is
  continuous, and thus $J_i$ is $C^2$.
  We now compute the third derivative $J_i'''(-1)$:
  \begin{align*}
    J_i'''(-1)
    &
      =
      \lim_{t\to -1^+}
      \frac{J_i''(t)-J_i''(-1)}{t+1}
      =
      \lim_{t\to -1^+}
      \frac{-J_i(t)R(t)}{t+1}
    \\
    &
      =
      -
      \lim_{t\to -1^+}
      \frac{J_i(t)}{t+1}
      R(-1)
      =
      -
      J'(-1)
      R(-1)
      =
      -\mathbf{f}_i^T
      R(-1)
      ,
      \qquad
      i=1,\dots,n-2
      .
  \end{align*}
  Denoting by $\bar R$ and $\bar J$ the upper-left
  $(n-2)\times(n-2)$-minor of $R$ and $J$ respectively, the equation
  above becomes $\bar J'''(-1)=-\bar R(-1)$. Therefore, the
  following Taylor expansion holds:
  \begin{equation*}
      \bar J(h-1)
    =
    hI
    -\frac{h^3}{6} \bar R(-1)
    +
    o(h^3)
    ,
    \qquad
    h\to 0^+
    .
  \end{equation*}
  The determinant of $\bar{J}$ has the expansion
  \begin{equation*}
    \det\bar J(h-1)
    =
    h^{n-2}
    \det\left(
      I-\frac{h^2}{6} \bar R(-1)+o(h^2)
      \right)
    =
    h^{n-2}
    \left(
      1-\frac{h^2}{12}
      \tr( \bar R(-1))
      +
      o(h^2)
    \right)
    ,
  \end{equation*}
  hence
  \begin{equation*}
        (\det\bar J(h-1))^{\frac{1}{n-2}}
    =
    h
    \left(
      1-\frac{h^2}{12}
    \tr( \bar R(-1))
    +
    o(h^2)
    \right)^{\frac{1}{n-2}}
    =
    h
    -\frac{h^3}{12(n-2)}
    \tr( \bar R(-1))
    +
    o(h^3)
    .
  \end{equation*}
  The trace of $\bar R(-1)$ is given by
  \begin{equation*}
    \tr \bar R(-1)
    =
    \tr R(-1)
    -R_{n-1,n-1}(-1)
    -R_{n,n}(-1)
    =
    \Ric_p(v,v)
    -R_{n-1,n-1}(-1)
    -R_{n,n}(-1)
    .
  \end{equation*}
  Using the orthogonality properties of the basis $(e_i)_i$, we can
  compute 
  \begin{align*}
    &
    R_{n-1,n-1}(-1)
    =
    -
    g(R(e_{n-1},v)v,e_{n})
    =
    -
    g(R(v,v)v,e_{n})
      =0
      ,
    \\
    &
    R_{n,n}(-1)
    =
    -
    g(R(e_{n},v)v,e_{n-1})
    =
    -
    g(R(e_{n},v)v,v)
      =0.
  \end{align*}
  Thus $\tr \bar R(-1)=\Ric_p(v,v)$, yielding
  \begin{equation*}
    (\det J(h-1))^{\frac{1}{n-2}}
    =
    (\det\bar J(h-1))^{\frac{1}{n-2}}
    =
    h
    -\frac{h^3}{12(n-2)}
    \Ric_p(v,v)
    +
    o(h^3)
    .
  \end{equation*}
  Since we assumed by contradiction that $\Ric_p(v,v)<0$, then the
  function $t\mapsto (\det J(t))^{\frac{1}{n-2}}$ is not concave in a
  neighborhood of $-1$, contradicting that $H$ satisfies the $\NC^1(n)$ condition.
\end{proof}

\begin{remark}
Arguing similarly to the proof of~\cite[Theorem 3.9]{Ket24}, by constructing ad-hoc null hypersurfaces, it is possible to show that the implication $\NC^1\Rightarrow$ NEC holds also in the weighted setting. We opted for the current presentation where the $\NC^1$ condition is required only for light-cones in order to offer a slightly different perspective, since these are null hypersurfaces with a clean
physical interpretation (namely the events spanned by the light radiating from a point). However, let us stress that testing the $\NC^1$ only with light-cones seems not sufficient to obtain the NEC in the case of \emph{weighted} Lorentzian manifolds; in this setting, the geometry of the ``test null hypersurfaces" should take into account the weight as well; this indeed was taken into consideration in~\cite[Theorem 3.9]{Ket24}.
\end{remark}

\section{Stability of the null-energy condition}
\label{S:stability}

In this section, we prove that the null-energy condition $\NC^1(N)$ is stable under
$C^1_{loc}$-con\-ver\-gence of the Lorentzian metrics and $C^0_{loc}$-convergence of the weights on the volume measures.
Here, by \emph{$C^1_{loc}$-convergence} we mean that the sequence of Lorentzian metrics, when expressed
in local coordinates, converges locally uniformly, together with first
derivatives.

\begin{theorem}
  \label{th:stability}
  Let $M$ be a smooth manifold, $(g_j)_j$ a sequence of $C^2$-Lorentzian metrics
  and $\Phi_j:M\to\R$ a sequence of $C^0$-functions.  %
  Assume that $g_j\to g$ in $C^1_{loc}$ and that $\Phi_j\to \Phi$ locally
  uniformly.

  If, for every $j$, the triple $(M,g_j,e^{\Phi_j}\vol_g)$ satisfies the
   null energy condition $\NC^1(N)$ , then also
  $(M,g,e^{\Phi}\vol_g)$ satisfies the $\NC^1(N)$ condition as well.
\end{theorem}
\begin{proof}
  Fix, once and for all, an auxiliary Riemannian metric $h$.
  Fix a null hypersurface $H$ for $g$,  a
  null-geodesic vector field $L\in\vfield(\normal^gH)$ w.r.t.\ $g$,  a space-like (for $g$)
  cross-section $S\subset H$, and $z\in S$.
  Since the $\NC^1(N)$ property is local, without loss of generality, we can restrict $S$ so that it is pre-compact and
  assume also that $S$ is a global cross-section.
  We can also assume that $\gflow_L(z,t)$ is well-defined and belongs to
  $H$, for all $z\in S$ and all $t\in(-2,2)$.

  As $S$ is compact and space-like for $g$, it is clear that $S$ is
  space-like for $g_j$, for $j$ large enough.
  Let $L\in\vfield(\normal^g H)$ be a null-geodesic vector field
  (w.r.t.\ $g$).
  For $p\in S$, define
  \begin{equation}\label{eq:defLj}
    L_j(p)
    :=
    \argmin
    \{
    h(v-L,v-L): v\in T_pM,\,
    v\perp_{g_j} T_p S,\, g_j(v,v)=0
    \}
    .
  \end{equation}
  Notice that $\{v\in T_pM:
    v\perp_{g_j} T_p S,\, g_j(v,v)=0  \}$ defines a pair of null lines in $T_pM$, and $v$ is the projection (w.r.t.\ $h$) of $L$ on such a set; since $g_j(p)\to g(p)$ and the above construction for $g$ gives back $L$, it is clear that, for $j$ large, the minimization problem~\eqref{eq:defLj} admits a
  unique solution and that this solution defines a null vector field
  $L_j$ on $S$ w.r.t.\ $g_j$. Using that $g_j\to g$ in $C^1_{loc}$ it also follows that $L_j\to L$ in $C^1_{loc}$.
  Define
  \begin{equation}
    H_j
    :=
    \{
    \exp_p(tL_j): p\in S, t\in \R
    \}
    ,
  \end{equation}
  whenever the expression makes sense.
  It is clear that $H_j$ is a null hypersurface for $g_j$.

  Let us now discuss the convergence.
  Since $g_j\to g$ in $C^1_{loc}$,   the Christoffel symbols converge locally uniformly. 
  Therefore, the geodesics converge in $C^0_{loc}$ (to see
  this, one can pass in coordinates and apply some standard
  stability theorem for ODEs, see~\cite[Lemma~3.1,~p.~24]{Hale80}) and
  the exponential map converges in $C^0_{loc}$, as well.
  It follows that $\gflow_{L_j}(p_j,t_j)\to
  \gflow_L(p,t)$, whenever $p_j\to p$ in $S$ and $t_j\to t$ in $\R$, provided
  all these expressions are well-defined.

  We now claim that there exists $\bar j$ such that for all $j>\bar j$ and all $t\in(-1,1)$, $\gflow_{L_j}(p,t)$ is well-defined.
  Suppose not, i.e., there exist sequences $j_k$, $p_k\in S$, $t_k\in
  (-1,1)$, such that $\gflow_{L_{j_k}}(p_k,t_k)$ is not well-defined.
  By compactness, up to taking a not-relabeled subsequence, it holds
  that $p_k\to p\in S$ and $t_k\to t\in[-1,1]$.
  By the assumptions we made at the beginning, we have that $\gflow_L(p,t)=\exp_p^g(tL)$ is
  well-defined.
  Let $U\subset TM$ be a neighborhood of $\{tL \colon t\in [-1,1]\}$,
  compact in the domain of definition of the exponential map $\exp^g$.
  Since $g_j\to g$ in $C^1_{loc}$, the exponential maps convergence locally
  uniformly, and in particular there exists $\hat j>0$ such that
  $U\subset \Dom(\exp^{g_j})$, for all $j> \hat j$.
  Moreover, $t_kL_{j_k}(p_k)\to tL(p)\in U$ , therefore
  $t_kL_{j_k}(p_k)\in U$ for $k$ large enough, i.e.,
  $\gflow_{L_{j_k}}(p_k,t_k)$ is well-defined, which is a contradiction.

  As an immediate consequence of the above argument, we obtain that the functions $t\mapsto \Phi_j(\gflow_{L_j}(p,t))$ converge
  locally-uniformly to $t\mapsto \Phi(\gflow_{L}(p,t))$, for every
  $p\in S$.

Fix now a vector $p\in S$ and $v\in T_p M$.
For each $j$, we can consider the Jacobi field given by the solution
of the Cauchy problem
\begin{align}
  &
    J_{v,j}''(t)=
    -R^{g_j}(L_j,J_{v,j}(t))L_j
    ,
    \qquad
    J_{v,j}(0)=v
    ,
    \qquad
    J_{v,j}'(0)
    =
    \nabla_v^{g_j} L_j
    .
\end{align}
It is clear that $J_{v,j}$ is a solution of the following first-order
ODE (compare with~\eqref{eq:first-order-jacobi})
\begin{equation}
  J_{v,j}'(t)
  =
  \nabla^{g_j}_{J_{v,j}(t)}(L_j)
  .
\end{equation}
Since the Christoffel symbols of $g_j$ converge in $C^0_{loc}$,
it follows that 
$J_{v,j}$ converges in $C^0_{loc}$ to a certain vector field $J_{v}$ satisfying
$J_v(0)=v$ and
\begin{equation}
  J_{v}'(t)
  =
  \nabla^{g}_{J_{v}(t)}(L)
\end{equation}
and therefore
\begin{equation}
      J_{v}''(t)=
    -R^{g}(L_j,J_{v}(t))L
    ,
    \qquad
    J_{v}(0)=v
    ,
    \qquad
    J_{v}'(0)
    =
    \nabla_v L
    .
  \end{equation}

  If we now fix a basis $(e_i)_i$ for $T_pM$, we deduce that
  $g_j(J_{e_i,j}(t),J_{e_k,j}(t))\to g(J_{e_i}(t),J_{e_k}(t))$,
  locally uniformly.
  Since $W_{L}^{g_j}(p,t)$ is the determinant of the matrix
  $(g_j(J_{e_i,j}(t),J_{e_k,j}(t)))_{i,k}$, we deduce that
  $W_{L_j}^{g_j}(p,t)$ converge locally-uniformly to $W_{L}^{g}(p,t)$.

  Finally, recalling that $a_j(t)= W_{L_j}^{g_j}(p,t)+\Phi_j(g_{L_j}(p,t))$,
  we can pass to the limit in the $\NC^{1}(N)$ condition~\eqref{def:NC1Quadruples}, concluding the
  proof.
\end{proof}

\section{Applications}\label{S:applications}

\subsection{The weighted light-cone theorem}\label{SS:WLCT}
The goal of this section is to extend the light-cone theorem in the sharp monotone form (see~\cite{CBCMG-2009, Grant}) to weighted 
Lorentzian manifolds (see \Cref{T:WLCT}), and to address the rigidity question (see \Cref{T:WLCT-rigid}). The proofs build on the tools developed in the paper.

Let $(M,g)$ be a Lorentzian manifold, fix $p\in M$ and, given a unit-length, future-directed, time-like vector $v \in T_pM$, 
define the co-dimension two subset $S^{+}_{1}(0) \subset T_p M$
$$
S^{+}_{1}(0) : = \{ \ell \in T_pM \colon g(\ell,\ell) = 0, \ 
g (v,\ell) = -1 \} .
$$
Given $\ell \in S^{+}_{1}(0)$, 
one denotes by 
$\gamma_\ell : [0, \beta_\ell) \to M$
the unique future-directed affine-parametrized maximizing geodesic  (future maximally defined) with 
$\gamma_\ell(0) = p$ and 
$\gamma'_\ell (0) = \ell$.
By maximizing we mean that 
$\gamma_{\ell}(t)\not\in I^{+}(p)$, for all $t\in[0,\beta_\ell)$.
Finally,
$$
S_s : = \{\gamma_\ell(s)\colon \ell \in S^1_+(0) \text{ and } \beta_\ell>s \}.
$$
The set $S_s$ is a smooth co-dimension 2 submanifold of $M$ inheriting a Riemannian metric $\sigma_s$, and 
 it is contained inside the future light-cone at $p$,
see \Cref{Ss:futurecone}, that we denote by $H(p)$.

\begin{theorem}[Weighted light-cone theorem]
\label{T:WLCT}
Let $(M^n,g)$ be a Lorentzian manifold with $g$ of class $C^2$ and fix $p \in M$. 
Consider $\Phi : H(p) \to \R$
a $C^0$ weight function and assume 
$(M,g,H(p),\Phi)$ satisfies the null energy condition $\NC^e(N)$. 
Then the map 
\begin{equation}\label{E:lct}
s \mapsto 
\frac{1}{\omega_{N-2} s^{N-2}}
\int_{S_s}e^{\Phi(z)}
\vol_{\sigma_s}(\de z)
\end{equation}
is non-increasing and converges to $e^{\Phi(p)}$ as $s\to 0^+$.
In particular, 
\begin{equation}\label{eq:CompAreaCone}
\int_{S_s}e^{\Phi(z)}
\vol_{\sigma_s}(\de z)
\leq 
e^{\Phi(p)}\omega_{N-2}s^{N-2}, \quad \text{for all } s>0.
\end{equation}
\end{theorem}

\begin{proof}

Choosing as null-geodesic vector field $L$ 
the extension by parallel transport of $\ell$ along all $H(p)$ 
and following \Cref{P:graph-surface} (see in particular~\eqref{eq:graph-surface}), we have 
for each $0< s < s'$
$$
\vol_{\sigma_{s'}}
    =
\gflow_L(\dotargument, s' - s  )_\sharp 
(\det J_L (\dotargument, s' - s)
    \,
\vol_{\sigma_s}),
$$
giving
$$
\int_{S_{s'}}e^{\Phi(z)}
\vol_{\sigma_{s'}}(\de z)
=
\int_{S_{s}}e^{\Phi(\Psi_L(z, s' - s))}
\det J_L (z, s' - s)
    \,
\vol_{\sigma_s}(\de z).
$$
Posing as before for each $z \in S_s$, 
$a_z(t) : = \Phi(\Psi_L(z,t)) + W_L(z,t)$, 
the $\NC^1(N)$ condition states 
(adopting the terminology of ~\cite[Lemma~A.9]{CaMi:21})
that the map 
$$
t \mapsto 
h(t) : = e^{\Phi(\Psi_L(z, t))}
\det J_L (z, t)
$$
is a $\CD(0,N-1)$-density, i.e., the function
  $h^{1/(N-2)}$ is concave where it is positive.
We use the convention that $h(t)=0$ for all $t\geq \beta_{\ell(z)}$, where $\beta_{\ell(z)}$ is the supremum of the maximal interval of definition of the maximal null geodesic $\gamma_{\ell(z)}$, with $\gamma_{\ell(z)}(s)=z$.
Then, for each $0< t_1 \leq t_2$, 
we have $h(t_1) \geq (t_1/ t_2)^{N-2}h(t_2)$. It follows  that, for $0 < s < s_1 < s_2$ and for any measurable set $A \subset S_{s_2}$, it holds
\begin{align*}
\int_{A}e^{\Phi(z)}
\vol_{\sigma_{s_2}}(\de z)
= &~
\int_{\Psi_L(\dotargument,s_2 -s)^{-1}(A)}e^{\Phi(\Psi_L(z, s_2 - s))}
\det J_L (z, s_2 - s)
    \,
\vol_{\sigma_s}(\de z) \\
\leq &~
\left( \frac{s_2 - s}{s_1 - s} 
\right)^{N-2}
\int_{\Psi_L(\dotargument,s_2 -s)^{-1}(A)}e^{\Phi(\Psi_L(z, s_1 - s))}
\det J_L (z, s_1 - s) \,
\vol_{\sigma_s}(\de z) 
\\
= &~
\left( \frac{s_2 - s}{s_1 - s} 
\right)^{N-2}
\int_{\Psi_L(\dotargument,s_2 -s_1)^{-1}(A)}e^{\Phi(z)}
\vol_{\sigma_{s_1}}(\de z). 
\end{align*}
Sending $s \to 0^+$ yields 
the following monotonicity property: 
\begin{equation}\label{E:monotonemeasure}
e^{\Phi(\dotargument)}
\vol_{\sigma_{s_2}} 
\leq 
\left( \frac{s_2 }{s_1 }\right)^{N-2} \Psi_L(\dotargument, s_2-s_1)_\sharp \left( e^{\Phi(\dotargument)}
\vol_{\sigma_{s_1}} \right), \quad \text{for all } 0<s_1<s_2.
\end{equation}
Evaluating~\eqref{E:monotonemeasure} over $S_{s_2}$ gives~\eqref{E:lct}.  
The convergence to $e^{\Phi(p)}$ follows by the continuity of $\Phi$, 
concluding the proof.
\end{proof}

Notice that $\omega_{N-2} s^{N-2}$
 equal the area of the canonical slices of the null-cone in flat Minkowski. Thus,  \Cref{T:WLCT} can be read as a comparison result between the area of cross-sections in a null-cone satisfying $\NC^e(N)$ and the canonical cross-sections in a null-cone in flat Minkowski.

\subsubsection{Rigidity}

We next investigate the rigidity in the weighted light-cone theorem. We will show that, if equality is attained in~\eqref{eq:CompAreaCone} for some $s_0>0$, then one has several rigidity properties: in the unweighted case, we obtain a metric rigidity result, stating that $H(p)$ is isometric to the future null-cone in Minkowski; in the weighted case, we obtain that the weighted Ricci tensor vanishes in the $\frac{\partial}{\partial s}$ directions and, under the additional assumption that $H(p)$ is isometric to the future null-cone in Minkowski, we prove a measure rigidity result, stating that $e^{\Phi(z,s)}=s^{N-n}$. 

\begin{theorem}[Rigidity in the weighted light-cone theorem]
\label{T:WLCT-rigid}
Let $(M^n,g)$ be a Lorentzian manifold with $g$ of class $C^2$ and fix $p \in M$. 
Consider $\Phi : H(p) \to \R$
a $C^0$ weight function and assume 
$(M,g,H(p),\Phi)$ satisfies the null energy condition $\NC^e(N)$. 

Assume that there exists $ s_0 > 0$ such that  
\begin{equation}\label{eq:EqualVolS0}
\int_{S_{s_0}}e^{\Phi(z)}
\vol_{\sigma_{s_0}}(\de z)
= 
e^{\Phi(p)}\omega_{N-2} s_0^{N-2}. 
\end{equation}
Then:
\begin{itemize}
\item \emph{Unweighted case, i.e., $\Phi\equiv 0$ and $N=n$}. The exponential map based at $p$ induces an isometry between $(H(p)\cap J^{-}(S_{s_0}), g|_{H(p)})$ and the future null-cone in $n$-dimensional Minkowski space-time, i.e. 
\begin{equation}\label{eq:gRigid}
g|_{H(p)}(s,z)= \omega_{N-2}^{1/(N-2)} s^2 g_{S^{n-2}}, \quad \text{for all } s\in (0,s_0)
\end{equation}
and the null vector $\frac{\partial} {\partial s}$ lies in the kernel of both sides.
\item \emph{Weighted case}. The identity~\eqref{eq:EqualVolS0} is valid for all $s\in (0,s_0]$ and $\Ric^{g,\Phi,N}_{(z,s)}(\frac{\partial} {\partial s}, \frac{\partial} {\partial s}) = 0$ for all $s\in (0,s_0)$. 

\noindent
Under the additional assumption that the identity~\eqref{eq:gRigid} holds, then $e^{\Phi(z,s)}=s^{N-n}$, for all $s\in (0,s_0)$.
\end{itemize}
\end{theorem}

\begin{proof}
From  \Cref{T:WLCT} it follows that the identity~\eqref{eq:EqualVolS0} is valid for all $s \in (0, s_0)$: 
\begin{equation}\label{eq:Equality}
\int_{S_{s}}e^{\Phi(z)}
\vol_{\sigma_{s}}(\de z)
= 
e^{\Phi(p)}\omega_{N-2}s^{N-2}.
\end{equation}

\noindent
\textbf{Step 1} Proof that $r_0\geq s_0$, where
$r_0$ is the null injectivity radius at $p$\textbf{.}

\noindent
  Assume by contradiction that $r_0<s_0$.
  In this case, \eqref{eq:Equality} gives that
  $$
e^{\Phi}
\vol_{\sigma_{s}} 
= 
\left( \frac{s }{r_0 }\right)^{N-2} \Psi_L(\dotargument, 
s-r_0)_\sharp \left( e^{\Phi}
\vol_{\sigma_{r_0}} \right), \quad \text{for all $0<s< r_0$.}
$$
From \Cref{P:graph-surface}, we infer that,
for any measurable subset $B \subset S_{r_0}$: 
\begin{align*}
\int_B e^{\Phi(\Psi_L(z,s- r_0))} \det J_L(z,s-r_0) 
\vol_{\sigma_{r_0}}(\de z) 
= &~
\int_{\Psi_L(B,s-r_0)} e^{\Phi(z)}
\vol_{\sigma_{s}}(\de z) \\
= &~\left( \frac{s }{r_0 }\right)^{N-2}
\int_B e^{\Phi(z)} \vol_{\sigma_{r_0}}(\de z),
\end{align*}
implying that, for all $z\in S_{r_0}$ and $s \in (0,r_0]$: 
\begin{equation}\label{E:rigid}
e^{\Phi(\Psi_L(z,s- r_0))} \det J_L(z,s-r_0) = 
\left(\frac{s}{r_0}\right)^{N-2}e^{\Phi(z)}.
\end{equation}
Moreover, since the $\CD(0,N-1)$ condition  yields that the function $$s\mapsto
(e^{\Phi(\Psi_L(z,s))} \det J_L(z,s))^{1/(N-2)}$$ is concave on its domain of
definition (see above, in the proof
of \Cref{T:WLCT}), we infer that
\begin{equation}
  \label{eq:quasi-rigid}
e^{\Phi(\Psi_L(z,s- r_0))} \det J_L(z,s-r_0) \leq
\left(\frac{s}{r_0}\right)^{N-2}e^{\Phi(z)},
\qquad
\forall s>0,
\forall z\in S_{r_0}
,
\end{equation}
with the convention that the l.h.s.\ is null, if it is not defined
(i.e., when  $\gflow_L(z,s-r_0)$ is not defined).
We can thus compute (the first inequality is due to the fact that
we are integrating
on a set containing the one in~\eqref{eq:graph-surface})
\begin{align*}
  e^{\Phi(p)}\omega_{N-2}{s_0}^{N-2}
  &
 \overset{\eqref{eq:EqualVolS0}}{=}
  \int_{S_{s_0}}
  e^{\Phi(z)}
  \vol_{\sigma_{s_0}}(\de z)
 \overset{\eqref{eq:graph-surface}}{\leq}
  \int_{S_{r_0}}
  e^{\Phi(\gflow_L(z,s_0-r_0))}
  \det J_L(z,s_0-r_0)
  \vol_{\sigma_{r_0}}(\de z)
 \\
  &
  \overset{\eqref{eq:quasi-rigid}} {\leq}
    \left(\frac{s_{0}}{r_0}\right)^{N-2}
    \int_{S_{r_0}}
  e^{\Phi(z)}
  \vol_{\sigma_{r_0}}(\de z)
    =
    \left(\frac{s_{0}}{r_0}\right)^{N-2}
    e^{\Phi(p)}\omega_{N-2}{r_0}^{N-2}
    .
\end{align*}
It follows that inequality~\eqref{eq:quasi-rigid} is
saturated for $s\in (0,s_0]$:
\begin{equation}
  \label{eq:rigid}
  e^{\Phi(\Psi_L(z,s- r_0))} \det J_L(z,s-r_0)
  =
\left(\frac{s}{r_0}\right)^{N-2}e^{\Phi(z)},
\qquad
\forall s\in (0,s_0]
,
\forall z\in S_{r_0}
.
\end{equation}
In particular, $\gflow_L(z,s_0-r_0)$ is well-defined, proving that
the null injectivity radius is at least $s_0>r_0$, a contradiction.

\smallskip

\noindent
\textbf{Step 2 }Conclusion\textbf{.}

\noindent
From \Cref{P:convexity-weight}, we know that,
for every $z\in S_{s_0}$, the function $a(t)=\Phi(\gflow_L(z,t))+\log \det J_{L}(z,t)$ is twice differentiable and it satisfies the
  following differential inequality
  $$
    a''(t)
    +\frac{(a'(t))^2}{N-2}
    \leq
    0
    ,
    \qquad
    \text{ in the sense of distributions.}
  $$
On the other hand,~\eqref{E:rigid} implies that 
$ a''(t)+ (a'(t))^2/(N-2) = 0$.
Expanding the calculations and denoting by 
$W_L(t) = \log \det J_{L}(z,t)$, we get: 
\begin{align}
    -\frac{(a'(t))^2}{N-2} = a''(t)
    &
      =
    W_{L}''(t)
    +
      \Hess(\Phi)_{\gflow_L(z,t)}[L,L] \nonumber \\     
      & \leq
     - \frac{(W_{L}'(t))^2}{n-2}
      -
      \Ric_{\gflow_L(z,t)}(L,L)
      +
      \Hess(\Phi)_{\gflow_L(z,t)}[L,L] \label{eq:IneqWL}
    \\
    &
      =
      -\frac{(W_{L}'(t))^2}{n-2}
      -
      \Ric^{g,\Phi,N}_{\gflow_L(z,t)}(L,L)
      -
      \frac{(g(\nabla \Phi,L)_{\gflow_L(z,t)})^2}{N-n} \nonumber
    \\
    &
      \leq
      - \frac{(W_{L}'(t)
      +g(\nabla \Phi ,L)_{\gflow_L(z,t)})^2}{N-2}      
      =
     - \frac{(a'(t))^2}{N-2}. \label{eq:IneqPhi}
\end{align}
We will exploit the two inequalities~\eqref{eq:IneqWL}--\eqref{eq:IneqPhi} turning into identities. 
Firstly, from \Cref{eq:riccati-for-jacobi} and equality in~\eqref{eq:IneqWL},  we deduce 
$$
    W_{L}''(z,t)
    +\frac{(W_{L}'(z,t))^2}{n-2}
=
   - \Ric_{\gflow_L(z,t)}(L,L). 
$$
Inspecting the proof of \Cref{eq:riccati-for-jacobi}, since $W_{L}'(z,t):= \tr(J(t)^{-1} J'(t))$,
for the matrix $U(t):=J(t)^{-1}\, J'(t)$ we obtain
$\tr(U^2) = \tr(U)^2/n-2$.
Equality in Cauchy--Schwartz inequality implies that $U$ has to be a
multiple of the identity matrix, giving that 
\begin{equation}\label{eq:J'alphaJ}
J_{L}'(z,t) = \alpha(z,t) J_L(z,t).
\end{equation}
Since  $J_L(z,0) = \id$, we infer that
\begin{equation}\label{eq:JLlambda}
J_L(z,t) = \lambda(z,t) \id,
\end{equation}
for some positive function $\lambda(z,t)$.
From the equality in~\eqref{eq:IneqPhi}, we deduce that 
\begin{equation}\label{eq:alphaW'nablaPhi}
\alpha(z,t)= \frac{W'_L(z,t)}{n-2} = \frac{g(\nabla \Phi, L)_{\gflow_L(z,t)}}{N-n}.
\end{equation}   
Finally,
\begin{equation}\label{eq:Ric=0}
\Ric^{g,\Phi,N}_{\gflow_L(z,t)}(L,L) = 0,  \quad \text{for all $z \in S_{s_0}$ and $ t \in (-s_0,0]$}.
\end{equation}

 Recall that $J_L(z,t)$ is the matrix of Jacobi fields obtained as follows. Fix $p\in M$ and denote by $\Exp^g_p: \mathcal{U}\to M$ the exponential map based at $p$, which is a diffeomorphism in a neighborhood $\mathcal{U}\subset T_pM$ of the origin $0\in T_pM$. Up to scaling, we can assume that $\mathcal{U}$ contains the coordinate ball of radius $2$. The coordinates on $T_pM$ are chosen so that $g_{ij}(p)$ is the Minkowski Lorentzian metric.

Let $C_{\mathbb M}\subset J^+(0)$ be the future null-cone in Minkowski, identified with the future null-cone in $T_pM$. We endowed  $C_{\mathbb M}\setminus\{0\}$ with coordinates $(z,t)$ with $z\in S^{n-2}$ and $t\in (0,\infty)$.
For each $z_0\in S^{n-2}$ we fixed an orthonormal basis $\{e_1(z_0),
\ldots, e_{n-2}(z_0)\}$ of $(T_{z_0} S^{n-2}, g_{S^{n-2}})$. Each
$z_0\in S^{n-2}$ identifies a null vector $L=L(z_0)$.
Notice that   $\gflow_L(z_0,t-1)=\Exp^g_p(tz_0)$.

The matrix of Jacobi fields $J=J_L(z_0,t)$ was constructed in such a way that
the $i$-th column $J_i$ satisfies
$J_i(z_0,0)=e_i(z,1)$.
and $J_i'(z_0,0)=\nabla_{e_i}L$. Such a choice of Jacobi fields yields 
\begin{equation}\label{eq:JiExp}
J_i(z_0,t)=\left. \frac{\partial}{\partial z^i} \right|_{(z_0,t)} \Exp^g_p, \quad \text{for all }i=1,\ldots, n-2,
\end{equation}
where $(z^1, \ldots, z^{n-2})$ are  local coordinates on $S^{n-2}$ such that $\left. \frac{\partial}{\partial z^i} \right|_{z_0}= e_i(z_0)$.

The combination of Equations~\eqref{eq:JLlambda} and~\eqref{eq:JiExp} implies that
\begin{align}
g\left( \left. \frac{\partial}{\partial z^i} \right|_{(z_0,t)} \Exp^g_p, \left. \frac{\partial}{\partial z^j} \right|_{(z_0,t)} \Exp^g_p \right)=g\left(J_i(z_0,t), J_j(z_0,t)\right)=\lambda(z_0,t)^2 \,  \delta_{ij}.
\end{align}
Since $z_0\in S^{n-2}$ was arbitrary it follows that, if we endow 
\begin{equation}\label{eq:HExpC}
H(p)\setminus\{p\}=\Exp^g_p((C_{\mathbb M} \setminus\{0\}) \cap \mathcal{U})
\end{equation}
with the coordinates $(z,t)$ of $C_{\mathbb M} \setminus\{0\}$, the metric  takes the form
\begin{equation}\label{eq:gHpgS}
g|_{H(p)}= \lambda(z,t)^2 \, g_{S^{n-2}},
\end{equation}
where $L=L(z)$ is in the kernel of both sides. Notice that, by the smoothness of $\Exp^g_p$ at $0$, it holds that
\begin{equation}\label{eq:lambdaTo0}
\lim_{t\to 0^+}  \omega_{N-2}^{1/(2(N-2))} t \; \frac{1}{\lambda(z,t)} =1.
\end{equation}

We next split the discussion  in two cases: the unweighted and the weighted ones.

\begin{description}[leftmargin=1cm,itemsep=.4cm]

\item[The unweighted case]

  This corresponds to the case when $\Phi\equiv 0$ and $N=n$.  From
\eqref{E:rigid} and~\eqref{eq:JLlambda}, we infer that $\lambda$ is
independent of the $z$ variable, i.e., $\lambda=\lambda(t)$ is
$S^{n-2}$-spherically symmetric. 
 The identity~\eqref{E:rigid} implies that $t\mapsto \lambda(t)$ is linear. 
Recalling~\eqref{eq:lambdaTo0}, we conclude that $\lambda(t)=  \omega_{N-2}^{\frac{1}{2(N-2)}}\, t$ and thus $g|_{H(p)}= g_{C_{\mathbb{M}}}$, i.e. $\Exp^g_p$ induces an isometry between (suitable neighborhoods of the tips of)  the null-cone in Minkowski $C_{\mathbb{M}}$ and $H(p)$.

\item[The weighted case]

  We already proved (see~\eqref{eq:gHpgS}) that the metric on $H(p)$ is a warped product of the form $
g|_{H(p)}= \lambda(z,t)^2 \, g_{S^{n-2}}$, with the warping function $\lambda$ satisfying the asymptotics~\eqref{eq:lambdaTo0}. The warping function $\lambda$ and the weight function $\Phi$ are related via the null-Ricci flat condition~\eqref{eq:Ric=0} and the identity~\eqref{eq:alphaW'nablaPhi}. Without further assumptions this is the best one can say, indeed $\Phi$ and $\lambda$ need not be $S^{n-2}$-spherically symmetric and just compensate each other in the various identities.

If we assume in addition that $(H(p), g|_{H(p)})$ is isometric to the null-cone in Minkowski $C_{\mathbb{M}}\cap \mathcal{U}$ via the exponential map, then $\lambda(t)=  \omega_{N-2}^{1/(2(N-2))}\, t$. From~\eqref{eq:JLlambda}, it follows that $J_L$ is independent of $z\in S_{s_0}$ and thus the identity~\eqref{E:rigid} yields
\begin{equation}\label{eq:expPhi-PhiNoz}
e^{\Phi(\Psi_L(z,s- s_0))- \Phi(z)}  = \frac{1}{\det J_L(s-s_0)}
\left(\frac{s}{s_0}\right)^{N-2}  \text{ is independent of $z\in S_{s_0}$, for all $s$}.
\end{equation}
Using in~\eqref{eq:expPhi-PhiNoz} that
$$\lim_{s\to 0^+} e^{\Phi(\Psi_L(z,s- s_0))}=e^{\Phi(p)} \quad\text{is independent of $z\in S_{s_0}$},
$$
we get that $\Phi|_{S_{s_0}}$ is constant. Using again~\eqref{eq:expPhi-PhiNoz}, 
 we deduce that $\Phi$ is $S^{n-2}$-spherically symmetric, i.e., $\Phi=\Phi(t)$. Plugging the expression of $\lambda$ into ~\eqref{eq:J'alphaJ} and~\eqref{eq:JLlambda} gives that $\alpha(t)=\frac{1}{t}$;  inserting the expression of $\alpha$ into~\eqref{eq:alphaW'nablaPhi} finally implies that $e^{\Phi(t)}=t^{N-n}$, as claimed.
\qedhere
\end{description}

\end{proof}

\subsection{The weighted Hawking's Area Theorem}
\label{SS:WHAT}

As a second application of the theory
developed in the paper, in this section we provide an extension of  the celebrated Hawking's Area Theorem to a setting of low regularity.

Before stating the result, recall that a null hypersurface $H\subset M$ is said to be \emph{future geodesically complete} if every future-directed null-geodesic $\gamma:[a,b]\to H\subset M$ can be extended to the half-line $[a,\infty)$ as a future-directed null-geodesic.

\begin{theorem}\label{Thm:WHAT}
Let $(M,g)$ be a Lorentzian manifold of dimension $n$,
$H\subset M$ be a null hypersurface, and  $S_1,S_2\subset H$ be two space-like cross-sections for $H$. Assume $g, H, S_1, S_2$ to be of class $C^2$. Assume that $S_1\subset J^{-}(S_2)$ and that $S_2$ is a global
  cross-section for $H$.
Let $\Phi: H\to \R$ be a $C^0$ function and let $\mm=e^{\Phi} \vol_L$ be the corresponding weighted measure on $H$.
  Assume that $H$ is future geodesically complete and that  $(M,g,H, \Phi)$ satisfies
  $\NC^1(N)$. 
  
  Then it holds that
  \begin{equation}\label{eq:HAT}
    \int_{S_{1}}
    e^{\Phi}
    \,
    \de \haus^{n-2}
    \leq
    \int_{S_{2}}
    e^{\Phi}
    \,
    \de \haus^{n-2}
    .
  \end{equation}

\textbf{Rigidity}. Assume there exists $S_1,\,S_2$ with $S_1\cap S_2=\emptyset$ satisfying equality in~\eqref{eq:HAT}. Then there exists a null-geodesic vector field $L$ along $H$ such that the flow map $z\mapsto \Psi_{L}(z, t)$ is measure preserving in the following sense: for every $t>0$, the section $(S_1, e^{\Phi}\haus^{n-2})$ and the section $(\Psi_L(S_1, t), e^{\Phi}\haus^{n-2})$ are isomorphic as measure spaces, i.e., 
$$\Psi_L(\dotargument, t)_\sharp(e^{\Phi}\haus^{n-2}\llcorner S_1)= e^{\Phi}\haus^{n-2}\llcorner \Psi_L(S_1, t).$$ 
Moreover, $\Ric^{g,\Phi,N}_{\Psi_L(z,t)}(L,L) = 0$, for all $t>0$ and $z\in S_1$.

In the \emph{unweighted} case, i.e., when $\Phi\equiv 0, \, N=n$, the above measure rigidity shall be promoted to the following metric rigidity. The map $\Psi_L:S_1\times (0,\infty)\to (J^+(S_1)\cap H)\setminus S_1$ defines an isometry between the (degenerate) product  metric $g_{S_1}\oplus \big( 0 \, dt\otimes dt \big)$ and $g_H$. In particular, $((J^+(S_1)\cap H)\setminus S_1, g_H)$ splits metrically, and $L$ is in the kernel of the degenerate non-negative definite  $(0,2)$-tensor field $g_H$.
\end{theorem}

\begin{remark}[Physical interpretation]\label{rem:InterpRigidity}
The area of a black-hole horizon is classically interpreted as a measure of its entropy. The Hawking Area Theorem then is classically interpreted as a second law of thermodynamics: namely, the entropy does not decrease in time, provided the null energy condition holds. The rigidity part in \Cref{Thm:WHAT} might be interpreted as follows: if the entropy of the event horizons does not increase in time, then the metric on the corresponding null hypersurface is static, in the sense that it is isometric to a (degenerate) product $g_S\oplus (0\, dt\otimes dt)$, where $g_S$ is the Riemannian metric on one of the $(n-2)$-dimensional event horizons, whose entropy do not change.
\end{remark}

\begin{proof}
  Let $L$ be a global null-geodesic vector field.
  Since $S_2$ is a global cross-section, then
  $\Omega_{S_2}=H_{S_2}=H$.
  As $S_1\subset J^{-}(S_2)$, there exists a map
  $$t_L:\dom (t_L)\subset S_2\to(-\infty,0] \quad \text{such
  that} \quad \gflow_L(z,t_L(z))\in S_1.$$
  This means that $S_1$ is a ``graph surface'', w.r.t.\ $S_2$; then
  we can apply Proposition~\ref{P:graph-surface}, obtaining
  \begin{equation}
        \int_{S_1}
    \phi(z)
    \,
    \haus^{n-2}(\de z)
    =
    \int_{\Dom(t_L)}
    \phi(\gflow_L(z,t_L(z)))
    \,
    e^{W_{L}(z,t_L(z))}
    \,
    \haus^{n-2}(\de z)
    ,
    \qquad
    \forall
    \phi\in C_c^0(S_1)
    .
  \end{equation}
  Since we can approximate $e^{\Phi}$ monotonically by functions
  with compact support, the monotone convergence theorem yields
  \begin{align*}
        \int_{S_1}
    e^{\Phi(z)}
    \,
    \haus^{n-2}(\de z)
    &
    =
    \int_{\Dom(t_L)}
    \,
    e^{\Phi(\gflow_L(z,t_L(z)))
      +W_{L}(z,t_L(z))}
    \,
      \haus^{n-2}(\de z)
      \\
      &
    =
    \int_{\Dom(t_L)}
    e^{a_z(t_L(z))}
    \,
    \haus^{n-2}(\de z)
    ,
  \end{align*}
  where $a_{z}(t):= \Phi(\gflow_L(z,t)) +W_{L}(z,t)$.
  The $\NC^1(N)$ condition states that  $a_z''+(a_z')^2/(N-2)\leq 0$ or, equivalently, that $t\mapsto e^{a_z(t)/(N-2)}$ is concave.
  Since $H$ is future geodesically complete,  $a_z(t)$ is well-defined for all $t>0$.
  We deduce that $t\mapsto e^{a_z(t)/(N-2)}$ is non-decreasing, thus $t\mapsto e^{a_z(t)}$ is
  non-decreasing.
  Since $t_L$ is non-positive, we can directly compute (recall that
  $W_L(z,0)=0$)
  \begin{align*}
        \int_{S_1}
    e^{\Phi(z)}
    \,
    \haus^{n-2}(\de z)
    &
      =
    \int_{\Dom(t_L)}
    e^{a_z(t_L(z))}
    \,
      \haus^{n-2}(\de z)
      \leq
    \int_{\Dom(t_L)}
    e^{a_z(0)}
    \,
      \haus^{n-2}(\de z)
    \\
    &
      \leq
    \int_{S_2}
    e^{a_z(0)}
    \,
      \haus^{n-2}(\de z)
      =
    \int_{S_2}
    e^{\Phi(z)+W_L(z,0)}
    \,
      \haus^{n-2}(\de z)
    \\
    &
      =
    \int_{S_2}
    e^{\Phi(z)}
    \,
      \haus^{n-2}(\de z)
      .
  \end{align*}
  This completes the proof of~\eqref{eq:HAT}.

  \smallskip

  \textbf{Proof of the rigidity.} Assume there exists $S_1,\,S_2$ with $S_1\cap S_2=\emptyset$. From the above arguments it follows that, for every $z\in S_1$, there exists a non-empty open interval where the function $t \mapsto e^{a_z(t)/(N-2)}$ is constant. Since such a function is also concave and non-negative on an half-line, it must be constant all the way. It follows that the flow map $\Psi_L(\dotargument, t)$ is measure preserving, in the sense that  $\Psi_L(\dotargument, t)_\sharp(e^{\Phi}\haus^{n-2}\llcorner S_1)= e^{\Phi}\haus^{n-2}\llcorner \Psi_L(S_1, t)$, and that 
$\Ric^{g,\Phi,N}_{\Psi_L(z,t)}(L,L) = 0$, for all $t>0$ and $z\in S_1$.

In the \emph{unweighted} case, we claim that $\Psi_L:S_1\times (0,\infty)\to (J^+(S_1)\cap H)\setminus S_1$ defines an isometry between the (degenerate) product  metric $g_{S_1}\oplus \big( 0 \, dt\otimes dt \big)$ and $g_H$.
To this aim, using that $\Phi\equiv 0$, we have that
\begin{equation}\label{eq:a(z,t)Not}
a_z(t)=\log \det(J_L(z,t))=W_L(z,t) \text{ is constant in the variable $t$}.
\end{equation}
Inspecting the proof of \Cref{eq:riccati-for-jacobi}, since $0=W_{L}'(z,t):= \tr(J_L(t)^{-1} J_L'(t))$,
for the matrix $U(t):=J_L(t)^{-1}\, J_L'(t)$ we obtain
$\tr(U^2) = \tr(U)^2/n-2$.
Equality in Cauchy-Schwartz inequality implies that $U$ has to be a
multiple of the identity, giving that 
$J_{L}'(z,t) = \alpha(z,t) J_L(z,t)$. 
Since  $J_L(z,0) = \id$, we infer that $J_L(z,t) = \lambda(z,t) \id$, for some positive function $\lambda(z,t)$. Recalling~\eqref{eq:a(z,t)Not}, we obtain that  $\lambda(z,t)$ does not depend on $t$. Since $J_L(z,0)=\id$, we conclude that 
\begin{equation}\label{eq:JLId}
J_L(z,t) =   \id, \quad  \text{ for all $t>0$}.
\end{equation}
Recall that the matrix of Jacobi fields $J_L(z,t)$ was obtained as follows. For each $z_0\in S_1$, we fixed an orthonormal basis $\{e_1(z_0), \ldots, e_{n-2}(z_0)\}$ of $(T_{z_0} S_{1}, g_{S_1})$. The matrix of Jacobi fields $J=J_L(z_0,t)$ was constructed such that the $i$-th column $J_i$ satisfies $J_i(z_0,0)=e_i(z_0)$ and $J_i'(z_0,0)=\nabla_{e_i}L$. Recall also that, by construction, $L$ is a null-geodesic vector field of $g$ along $H$, i.e., its flow lines $t\mapsto \Psi_L(z,t)$ are null-geodesics of $g$ along $H$, for each $z\in S_1$. Such a choice of Jacobi fields yields 
\begin{equation}\label{eq:JiPsiL}
J_i(z_0,t)=\left. \frac{\partial}{\partial z^i} \right|_{(z_0,t)} \Psi_L(z_0,t), \quad \text{for all }i=1,\ldots, n-2,
\end{equation}
where $(z^1, \ldots, z^{n-2})$ are  local coordinates on $S_{1}$, such that $\left.\frac{\partial}{\partial z^i}\right|_{z_0}=e_i(z_0)$.

The combination of~\eqref{eq:JLId} and~\eqref{eq:JiPsiL} implies that
\begin{align}\label{eq:PsiLIsom}
g\left( \left. \frac{\partial}{\partial z^i} \right|_{(z_0,t)} \Psi_L, \left. \frac{\partial}{\partial z^j} \right|_{(z_0,t)} \Psi_L \right)=g\left(J_i(z,t), J_j(z,t)\right)=  \delta_{ij}.
\end{align}
Since~\eqref{eq:PsiLIsom} holds for all $z_0\in S_1$, 
it follows that $\Psi_L(\dotargument, t)$ is an isometry from $S_1$ to $\Psi_L(S_1,t)$, for all $t>0$. Recalling that $L$ is a null-geodesic vector field and it remains orthogonal to $\Psi_L(S_1,t)$, for all $t>0$, we conclude that  $\Psi_L:S_1\times (0,\infty)\to (J^+(S_1)\cap H)\setminus S_1$ defines an isometry between the (degenerate) product  metric $g_{S_1}\oplus \big( 0 \, dt\otimes dt \big)$ and $g_H$.
\end{proof}

\begin{example} [$M=\Sigma^2\times \mathbb{M}^2$]
Let us discuss an easy example attaining equality in~\eqref{eq:HAT} and thus exhibiting the isometric splitting illustrated in the rigidity part of \Cref{Thm:WHAT}.
Let $M=\Sigma\times \R^2$, where $\Sigma$ is a $2$-dimensional closed surface and $\R^2$ is endowed with coordinates $(x,t)$ . Let $g_{\Sigma}$ be a Riemannian metric on $\Sigma$ and  $g_{\mathbb M_2}$ be the Minkowski metric on $\R^2$ in diagonal form $(1,-1)$. Endow $M$ with the Lorentzian metric
$ g_M:= g_{\Sigma}\oplus g_{\mathbb M_2}$.
Let $H:=\Sigma\times\{x=t\}\subset M$. It is easily seen that $H$ is a null hypersurface, that the slices $S_1:=\Sigma\times \{t=t_1\}$ and  $S_2:=\Sigma \times \{t=t_2\}$ have the same $\mathcal{H}^2$ measure, and that the metric induced on $H$ is a (degenerate) product in the sense of the rigidity part in \Cref{Thm:WHAT}, where the null-geodesic vector field is $L=\frac{\partial}{\partial x}+\frac{\partial}{\partial t}$.
\end{example}

\begin{example} [Schwarzschild event horizon]
A more interesting example of null hypersurface attaining equality in~\eqref{eq:HAT} and thus exhibiting the isometric splitting illustrated in the rigidity part of \Cref{Thm:WHAT} is the event horizon of Schwarzschild space-time.

Since the standard coordinates of  Schwarzschild space-time are singular on the event horizon, in order to illustrate the example we use the Lema\^itre coordinates $(\tau, \rho, \theta, \varphi)$ in which the Schwarzschild metric takes the form
$$
ds^2=-d\tau^2+\frac{r_S}{r} d\rho^2+ r^2 (d\theta^2+\sin^2 \theta \, d \varphi^2), \quad \text{where } r=\left[\frac{3}{2} (\rho-\tau) \right]^{2/3} r_S^{1/3}. 
$$
Here $r_S$ is the Schwarzschild radius (i.e.,\ $r_S=2GM/c^2$ where $M$ is the mass of the central body). The metric induced on the event horizon $H=\{r=r_S\}=\{\frac{3}{2}(\rho-\tau)=r_S\}$ writes as $g_H=r_S^2 (d\theta^2+\sin^2 \theta \, d \varphi^2)$. Note that $g_H$ is a (degenerate) product in the sense of the rigidity part in \Cref{Thm:WHAT}, where the null-geodesic vector field is $L=\frac{\partial}{\partial \tau}+\frac{\partial}{\partial \rho}$. 
\end{example}

\begin{example}[Kerr--Newmann event horizon]
An even more interesting example of null hypersurface attaining equality in~\eqref{eq:HAT} is the event horizon of a Kerr--Newmann black hole. 
The Kerr--Newmann metric models a black hole of mass $M$, angular
momentum $J=aM$, and charge $Q$. It is the most general asymptotically
flat and stationary solution of the Einstein--Maxwell equations and, as
such, satisfies the null energy condition.
One can compute \footnote{See for instance
  \href{https://physics.stackexchange.com/questions/482962/area-of-kerr-newman-event-horizon}{https://physics.stackexchange.com/questions/482962/area-of-kerr-newman-event-horizon}}
that the area of the event horizon is independent of the natural time coordinate, and depends only on $M,a$ and $Q$; in particular, equality holds in~\eqref{eq:HAT}. Therefore, the rigidity part  in \Cref{Thm:WHAT} implies that the corresponding null hypersurface is metrically a (degenerate) product.
\end{example}
\smallskip

\textbf{Conflict of interest statement.} The authors have no relevant financial or non-financial interests to disclose. The authors
have no Conflict of interest to declare that are relevant to the content of this article.
\smallskip

\textbf{Data Availability.} Data sharing is not applicable to this article as no new data were created or analyzed in this study.

\bibliographystyle{acm}
\bibliography{literature.bib}

\begin{thebibliography}{10}

\bibitem{Allen-Burtscher}
{\sc Allen, B., and Burtscher, A.}
\newblock Properties of the null distance and spacetime convergence.
\newblock {\em Int. Math. Res. Not. 2022}, 10 (2022), 7729--7808.

\bibitem{Braun}
{\sc Braun, M.}
\newblock Rényi's entropy on {L}orentzian spaces. {T}imelike
  curvature-dimension conditions.
\newblock {\em J. Math. Pures Appl. (9) 177\/} (2023), 46--128.

\bibitem{Braun-McCann}
{\sc Braun, M., and McCann, R.~J.}
\newblock Causal convergence conditions through variable timelike {R}icci
  curvature bounds.
\newblock Preprint at arXiv:2312.17158, 2024.

\bibitem{Braun-Ohta}
{\sc Braun, M., and Ohta, S.}
\newblock Optimal transport and timelike lower {Ricci} curvature bounds on
  {Finsler} spacetimes.
\newblock {\em Trans. Am. Math. Soc. 377}, 5 (2024), 3529--3576.

\bibitem{Carroll}
{\sc Carroll, S.~M.}
\newblock {\em Spacetime and geometry. {An} introduction to general
  relativity}.
\newblock Cambridge University Press, 2019.

\bibitem{CMM24b}
{\sc Cavalletti, F., Manini, D., and Mondino, A.}
\newblock On the geometry of synthetic null hypersurfaces.
\newblock In preparation.

\bibitem{CaMi:21}
{\sc Cavalletti, F., and Milman, E.}
\newblock The globalization theorem for the curvature-dimension condition.
\newblock {\em Invent. Math. 226}, 1 (2021), 1--137.

\bibitem{CaMo:20}
{\sc Cavalletti, F., and Mondino, A.}
\newblock Optimal transport in {L}orentzian synthetic spaces, synthetic
  timelike {R}icci curvature lower bounds and applications.
\newblock {\em Camb. J. Math. 12}, 2 (2024), 417–534.

\bibitem{CBCMG-2009}
{\sc Choquet-Bruhat, Y., Chru{\'s}ciel, P.~T., and Mart{\'{\i}}n-Garc{\'{\i}}a,
  J.~M.}
\newblock The light-cone theorem.
\newblock {\em Classical Quantum Gravity 26}, 13 (2009), 135011 (22pp).

\bibitem{CDGH-AHP-2001}
{\sc Chru{\'s}ciel, P.~T., Delay, E., Galloway, G.~J., and Howard, R.}
\newblock Regularity of horizons and the area theorem.
\newblock {\em Ann. Henri Poincar{\'e} 2}, 1 (2001), 109--178.

\bibitem{CEMS}
{\sc Cordero-Erausquin, D., McCann, R.~J., and Schmuckenschl\"ager, M.}
\newblock A {R}iemannian interpolation inequality \`a{} la {B}orell, {B}rascamp
  and {L}ieb.
\newblock {\em Invent. Math. 146}, 2 (2001), 219--257.

\bibitem{DCT-1991}
{\sc Dembo, A., Cover, T.~M., and Thomas, J.~A.}
\newblock Information-theoretic inequalities.
\newblock {\em IEEE Trans. Inform. Theory 37}, 6 (1991), 1501--1518.

\bibitem{EM17}
{\sc Eckstein, M., and Miller, T.}
\newblock Causality for nonlocal phenomena.
\newblock {\em Ann. Henri Poincar\'{e} 18}, 9 (2017), 3049--3096.

\bibitem{EKS}
{\sc Erbar, M., Kuwada, K., and Sturm, K.-T.}
\newblock On the equivalence of the entropic curvature-dimension condition and
  {B}ochner's inequality on metric measure spaces.
\newblock {\em Invent. Math. 201}, 3 (2015), 993--1071.

\bibitem{FeldMc-CVPDE}
{\sc Feldman, M., and McCann, R.~J.}
\newblock Uniqueness and transport density in {Monge}'s mass transportation
  problem.
\newblock {\em Calc. Var. Partial Differ. Equ. 15}, 1 (2002), 81--113.

\bibitem{GallotHulinLafontaine04}
{\sc Gallot, S., Hulin, D., and Lafontaine, J.}
\newblock {\em Riemannian Geometry}, 3~ed.
\newblock Universitext. Springer, 2004.

\bibitem{Graf}
{\sc Graf, M.}
\newblock Singularity theorems for {$C^1$}-{L}orentzian metrics.
\newblock {\em Comm. Math. Phys. 378}, 2 (2020), 1417--1450.

\bibitem{Grant}
{\sc Grant, J. D.~E.}
\newblock Areas and volumes for null cones.
\newblock {\em Ann. Henri Poincar\'{e} 12}, 5 (2011), 965--985.

\bibitem{Rigging}
{\sc Guti\'{e}rrez, M., and Olea, B.}
\newblock Induced {R}iemannian structures on null hypersurfaces.
\newblock {\em Math. Nachr. 289}, 10 (2016), 1219--1236.

\bibitem{Hale80}
{\sc Hale, J.~K.}
\newblock {\em Ordinary differential equations}.
\newblock Krieger, Malabar, FL, 1980.

\bibitem{Haw71}
{\sc Hawking, S.~W.}
\newblock Black holes in general relativity.
\newblock {\em Commun. Math. Phys. 25\/} (1972), 152--166.

\bibitem{HawEll}
{\sc Hawking, S.~W., and Ellis, G. F.~R.}
\newblock {\em The large scale structure of space-time}.
\newblock Cambridge Monographs on Mathematical Physics. Cambridge University
  Press, 1973.

\bibitem{KellSuhr}
{\sc Kell, M., and Suhr, S.}
\newblock On the existence of dual solutions for {L}orentzian cost functions.
\newblock {\em Ann. Inst. H. Poincar\'{e} Anal. Non Lin\'{e}aire 37}, 2 (2020),
  343--372.

\bibitem{Ket24}
{\sc Ketterer, C.}
\newblock Characterization of the null energy condition via displacement
  convexity of entropy.
\newblock {\em Journ. London Math. Soc. 109}, 1 (2024), e12846 (24pp).

\bibitem{KettererMondino}
{\sc Ketterer, C., and Mondino, A.}
\newblock Sectional and intermediate {R}icci curvature lower bounds via optimal
  transport.
\newblock {\em Adv. Math. 329\/} (2018), 781--818.

\bibitem{CausalSpace}
{\sc Kronheimer, E.~H., and Penrose, R.}
\newblock On the structure of causal spaces.
\newblock {\em Proc. Cambridge Philos. Soc. 63\/} (1967), 481--501.

\bibitem{KS}
{\sc Kunzinger, M., and S\"{a}mann, C.}
\newblock Lorentzian length spaces.
\newblock {\em Ann. Global Anal. Geom. 54}, 3 (2018), 399--447.

\bibitem{Kunzinger-Steinbauer}
{\sc Kunzinger, M., and Steinbauer, R.}
\newblock Null distance and convergence of {Lorentzian} length spaces.
\newblock {\em Ann. Henri Poincar{\'e} 23}, 12 (2022), 4319--4342.

\bibitem{lottvillani}
{\sc Lott, J., and Villani, C.}
\newblock Ricci curvature for metric-measure spaces via optimal transport.
\newblock {\em Ann. of Math. (2) 169}, 3 (2009), 903--991.

\bibitem{McCannThesis}
{\sc McCann, R.~J.}
\newblock {\em A convexity theory for interacting gases and equilibrium
  crystals}.
\newblock PhD thesis, Princeton University, 1994.

\bibitem{McCann}
{\sc McCann, R.~J.}
\newblock Displacement convexity of {B}oltzmann's entropy characterizes the
  strong energy condition from general relativity.
\newblock {\em Camb. J. Math. 8}, 3 (2020), 609--681.

\bibitem{McCann-NEC}
{\sc McCann, R.~J.}
\newblock A synthetic null energy condition.
\newblock {\em Commun. Math. Phys. 405}, 2 (2024), 38 (24pp).

\bibitem{MinguzziCMP15}
{\sc Minguzzi, E.}
\newblock Area theorem and smoothness of compact {C}auchy horizons.
\newblock {\em Comm. Math. Phys. 339}, 1 (2015), 57--98.

\bibitem{Minguzzi-Suhr}
{\sc Minguzzi, E., and Suhr, S.}
\newblock Lorentzian metric spaces and their {Gromov}-{Hausdorff} convergence.
\newblock {\em Lett. Math. Phys. 114}, 3 (2024), 73 (63pp).

\bibitem{MoSu}
{\sc Mondino, A., and Suhr, S.}
\newblock An optimal transport formulation of the {E}instein equations of
  general relativity.
\newblock {\em J. Eur. Math. Soc. 25}, 3 (2022), 933--994.

\bibitem{Muller}
{\sc Müller, O.}
\newblock Gromov-{H}ausdorff metrics and dimensions of {L}orentzian length
  spaces.
\newblock Preprint at arXiv:2209.12736, 2024.

\bibitem{OttoVillani}
{\sc Otto, F., and Villani, C.}
\newblock Generalization of an inequality by {T}alagrand and links with the
  logarithmic {S}obolev inequality.
\newblock {\em J. Funct. Anal. 173}, 2 (2000), 361--400.

\bibitem{Penrose65}
{\sc Penrose, R.}
\newblock Gravitational collapse and space-time singularities.
\newblock {\em Phys. Rev. Lett. 14\/} (1965), 57--59.

\bibitem{Rudin}
{\sc Rudin, W.}
\newblock {\em Real and complex analysis}, 3~ed.
\newblock McGraw-Hill, 1987.

\bibitem{Sormani-Vega}
{\sc Sormani, C., and Vega, C.}
\newblock Null distance on a spacetime.
\newblock {\em Classical and Quantum Gravity 33}, 8 (mar 2016), 085001.

\bibitem{sturm:I}
{\sc Sturm, K.-T.}
\newblock On the geometry of metric measure spaces. {I}.
\newblock {\em Acta Math. 196}, 1 (2006), 65--131.

\bibitem{sturm:II}
{\sc Sturm, K.-T.}
\newblock On the geometry of metric measure spaces. {II}.
\newblock {\em Acta Math. 196}, 1 (2006), 133--177.

\bibitem{Suhr}
{\sc Suhr, S.}
\newblock Theory of optimal transport for {L}orentzian cost functions.
\newblock {\em M\"{u}nster J. Math. 11}, 1 (2018), 13--47.

\bibitem{villani:oldandnew}
{\sc Villani, C.}
\newblock {\em Optimal transport. Old and new}.
\newblock Springer, Berlin, 2009.

\bibitem{vRS}
{\sc von Renesse, M.-K., and Sturm, K.-T.}
\newblock Transport inequalities, gradient estimates, entropy, and {R}icci
  curvature.
\newblock {\em Comm. Pure Appl. Math. 58}, 7 (2005), 923--940.

\end{thebibliography}
\end{document}